\documentclass[a4paper, 10pt, DIV = 13]{scrartcl}
\usepackage[utf8]{inputenc}
\usepackage[english]{babel}

\usepackage{amsmath,amsfonts,amssymb,amsthm,mathrsfs,todonotes}
\usepackage{amsfonts}
\usepackage{amssymb}
\usepackage{wasysym}
\usepackage{lmodern}
\usepackage{hyperref}
\usepackage{xcolor}
\usepackage{hyperref}
\usepackage{graphicx}
\usepackage[normalem]{ulem}
\hypersetup{breaklinks=true}

\usepackage{mathtools}
\mathtoolsset{showonlyrefs}

\definecolor{darkred}{rgb}{0.5,0,0}
\definecolor{darkgreen}{rgb}{0,0.5,0}
\definecolor{darkblue}{rgb}{0,0,0.5}
\hypersetup{colorlinks,
linkcolor=darkblue,
filecolor=darkgreen,
urlcolor=darkred,
citecolor=darkblue}

\usepackage{centernot}
\numberwithin{equation}{section}

\newtheorem{proposition}{Proposition}[section]
\newtheorem{lemma}[proposition]{Lemma}
\newtheorem{corollary}[proposition]{Corollary}

\newtheorem{claim}[proposition]{Claim}

\theoremstyle{definition}
\newtheorem{theorem}{Theorem}
\newtheorem*{theorem*}{Theorem}
\newtheorem{definition}[proposition]{Definition}
\newtheorem{remark}[proposition]{Remark}

\usepackage{enumitem}
\setlist{nosep}
\setlist{noitemsep}
\setlist{leftmargin=*}

\usepackage{mathtools}

\newcommand\smbullet[1][.5]{\mathbin{\vcenter{\hbox{\scalebox{#1}{$\bullet$}}}}}

\newcommand{\sbullet}{{\smbullet[0.6]}}

\def\XXint#1#2#3{{\setbox0=\hbox{$#1{#2#3}{\int}$}
     \vcenter{\hbox{$#2#3$}}\kern-.5\wd0}}

\newcommand{\R}{\mathbb{R}}
\newcommand{\Z}{\mathbb{Z}}

\renewcommand{\epsilon}{\varepsilon}
\newcommand{\hal}{\frac{1}{2}}

\newcommand{\dd}{\mathtt{d}}

\newcommand{\E}{\mathbb{E}}

\newcommand{\dist}{\mathrm{dist}}

\newcommand{\La}{\Lambda}

\newcommand{\Cov}{\mathrm{Cov}}

\newcommand{\corT}{}

\begin{document}
\title{\LARGE{The link between hyperuniformity, Coulomb energy, and Wasserstein distance to Lebesgue for two-dimensional point processes}}

\author{\large{Martin Huesmann\thanks{Institute for Mathematical Stochastics, University of Münster, \texttt{martin.huesmann@uni-muenster.de}}, Thomas Leblé\thanks{ Université de Paris-Cité, CNRS, MAP5 UMR 8145, F-75006 Paris, France \texttt{thomas.leble@math.cnrs.fr} }}}

\date{\small{\today}}
\maketitle
\begin{abstract}
We investigate the interplay between three possible properties of stationary point processes: i) Finite Coulomb energy with short-scale regularization, ii) Finite $2$-Wasserstein transportation distance to the Lebesgue measure and iii) Hyperuniformity. In dimension $2$, we prove that i) implies ii), which is known to imply iii), and we provide simple counter-examples to both converse implications. However, we prove that ii) implies i) for processes with a uniformly bounded density of points, and that i) - finiteness of the regularized Coulomb energy - is equivalent to a certain property of quantitative hyperuniformity that is just slightly stronger than hyperuniformity itself.

Our proof relies on the classical link between $H^{-1}$-norm and $2$-Wasserstein distance between measures, on the screening construction of \cite{sandier20152d} for Coulomb gases (of which we present an adaptation to $2$-Wasserstein space which might be of independent interest), and on the recent necessary and sufficient conditions given in \cite{sodin2023random} for the existence of stationary “electric” fields compatible with a given stationary point process.
\end{abstract}

\newcommand{\bX}{\mathbf{X}}
\newcommand{\bbX}{\mathbf{X}^{\sbullet}}

\newcommand{\X}{\mathrm{X}}
\renewcommand{\d}{\mathsf{d}}
\newcommand{\Rd}{\R^\d}
\newcommand{\Zd}{\Z^\d}
\newcommand{\B}{\mathrm{B}}
\newcommand{\BR}{\B_r}
\newcommand{\Var}{\mathrm{Var}}
\newcommand{\SL}{\mathbf{L}}
\newcommand{\PSL}{\mathbf{PL}}
\renewcommand{\P}{\mathbb{P}}
\newcommand{\p}{\mathrm{p}}
\newcommand{\bp}{\bar{\p}}

\renewcommand{\O}{\mathcal{O}}

\newcommand{\kk}{\mathsf{k}}
\newcommand{\1}{\mathsf{1}}
\newcommand{\2}{\mathsf{1}}
\newcommand{\Hess}{\mathrm{Hess}}

\renewcommand{\div}{\mathrm{div}}
\newcommand{\sigmaX}{\sigma}
\newcommand{\rhoX}{\rho_{\bX}}
\newcommand{\El}{\mathrm{E}}
\newcommand{\bEl}{\mathrm{E}^{\sbullet}}
\newcommand{\jr}{j_r}
\newcommand{\bPi}{\Pi^{\sbullet}}
\newcommand{\Proba}{\mathbb{P}}

\newcommand{\Leb}{\mathrm{Leb}}
\newcommand{\Cpl}{\mathsf{cpl}}
\newcommand{\PCpl}{\mathsf{Pcpl}}
\newcommand{\tPp}{\widetilde{\Pp}}
\newcommand{\Conf}{\mathrm{Conf}}
\newcommand{\Pro}{\mathcal{P}}
\newcommand{\tbX}{\widetilde{\bX}}
\newcommand{\btbX}{\widetilde{\bX}^\sbullet}

\newcommand{\tEl}{\widetilde{\El}}
\newcommand{\btEl}{\widetilde{\El}^\sbullet}

\newcommand{\HU}{\texttt{HU}}
\newcommand{\SHU}{\texttt{HU}_\star}
\newcommand{\EU}{\texttt{Coul}}
\newcommand{\WD}{\texttt{Wass}_2}
\newcommand{\WU}{\texttt{Wass}_1}
\newcommand{\SC}{\texttt{SC}}

\newcommand{\Pp}{\mathrm{P}}
\newcommand{\Coul}{\mathsf{Coul}}
\newcommand{\Ww}{\mathsf{Wass}}
\newcommand{\cd}{c_{\d}}
\newcommand{\AK}{\mathcal{A}_{k}}

\section{Introduction}
\label{sec:intro}
\subsection{Setting and main results}
Let $\d \geq 1$ be the ambient dimension and let $\bbX$ be a stationary point process, i.e.\ a random locally finite configuration of points in $\Rd$ whose distribution is invariant under translations.

\begin{description}
   \item[($\HU$)] For $r > 0$, we let $\BR$ be the closed Euclidean ball of radius $r$ centered at the origin and if $\bX$ is a point configuration we denote by $|\bX \cap \BR|$  the number of its points in $\BR$. We say that a point process $\bbX$ is \emph{hyperuniform} when (see Section \ref{sec:HU} for more context):
   \begin{equation}
   \label{def:HU}
\sigmaX(r) := \frac{\Var[|\bbX \cap \BR|]}{|\BR|} \longrightarrow 0 \text{ as } r \to \infty.
   \end{equation}

   \item[($\WD$)] We say that $\bbX$ is at finite $2$-Wasserstein distance of the Lebesgue measure when there exists a stationary random coupling between the points of $\bbX$ and the Lebesgue measure, whose quadratic cost per unit volume is bounded on average (see Section \ref{sec:DefWass}).

   \item[($\EU$)] We say that a stationary point process $\bbX$ has \emph{finite regularized Coulomb energy} when there exists a stationary \emph{electric field} $\bEl$ compatible with $\bbX$ (a random vector field which could be the gradient of the Coulomb potential generated by the random points of $\bbX$) that satisfies: $\E\left[|\bEl_\eta(0)|^2\right] < + \infty$, where $\eta \in (0,1)$ is any short-distance truncation lengthscale. We refer to Section \ref{sec:CoulombEnergy} for a precise definition.

\end{description}
There has been recent interest in studying the link between those three notions when the dimension $\d = 2$. In \cite[\corT{Thm. $1$}]{HUPL}, it is proven that $\WD$ implies $\HU$, confirming statements in the statistical physics literature. In \cite{lachieze2024hyperuniformity}, a partial converse is stated as follows: if a point process is hyperuniform ($\HU$)  and if its two-point correlation function satisfies \emph{“a mild integrability condition”} then it is at finite distance to Lebesgue ($\WD$). We discuss those results in Section \ref{sec:discussion}.

We argue here that the integrability condition of \cite{lachieze2024hyperuniformity} essentially corresponds to the property $\EU$ of having finite regularized Coulomb energy, and that $\EU$ implies indeed $\WD$. Moreover, we introduce a slight reinforcement of the notion of hyperuniformity, which we call $\SHU$ (see Definition~\ref{def:SHU}), and we show that $\SHU$ implies $\EU$. In summary, for $\d=2$:
\begin{equation*}
\SHU \implies \EU \implies \WD \implies \HU.
\end{equation*}
Moreover, we prove that the first implication is an equivalence but that the other two are not.

\subsubsection*{From finite Coulomb energy to finite $2$-Wasserstein distance (...to hyperuniformity)}
Our first result is the implication $\EU \implies \WD$, which is in fact true for all $\d \geq 2$.
\begin{theorem}[$\EU \implies \WD$]
\label{theo:EUWD}
In any dimension $\d \geq 2$, $\EU$ implies $\WD$. More precisely, there exists a constant $C$ depending only on the dimension such that, for all stationary point processes $\bbX$:
\begin{equation}
\label{EUWD}
\Ww^2_2(\bbX, \Leb) \leq C\left(\Coul_1(\bbX) + 1\right),
\end{equation}
thus effectively bounding the $2$-Wasserstein distance to Lebesgue (defined in Section \ref{sec:DefWass}) by the regularized Coulomb energy $\Coul_1$ (see Definition \ref{def:RegularizedEnergy}).
\end{theorem}
The Wasserstein-Coulomb inequality \eqref{EUWD} can be thought of as the infinite-volume equivalent of classical inequalities bounding the $2$-Wasserstein distance between measures by the Sobolev $H^{-1}$-norm of their difference (see \cite[Sec. 2]{Ledoux_2019} and the references therein). Combined with \cite[\corT{Thm. 1}]{HUPL} this yields (in dimension $2$):
\begin{equation}
\label{EUWDHU}
\EU \implies \WD \implies \HU.
\end{equation}
We prove Theorem \ref{theo:EUWD} in Section \ref{sec:E1W2}. 

It is natural to ask whether the converse implications hold in \eqref{EUWDHU}. Thanks to simple counter-examples, we show that the answer is negative:
\begin{theorem}[$\WD \centernot \implies \EU$, $\HU \centernot \implies \WD$]
\label{theo:counter}
In dimension $2$: 
\begin{enumerate}
   \item There exists a stationary point process which has finite $2$-Wasserstein distance to Lebesgue but infinite regularized Coulomb energy. Hence $\WD$ does not imply $\EU$.
   \item There exists a stationary point process which is hyperuniform yet has infinite $1$-Wasserstein distance to Lebesgue. In particular, its $2$-Wasserstein distance is also infinite, and thus $\HU$ does not imply $\WD$.
\end{enumerate}
\end{theorem}
The constructions leading to Theorem \ref{theo:counter} are presented in Section \ref{sec:constructions}.
\newcommand{\brho}{\bar{\rho}}

\subsubsection*{Finite $2$-Wasserstein distance plus density bound implies finite Coulomb energy}
The fact that $\EU$ implies $\WD$ but not vice-versa is related to the slight difference in behavior between the Coulomb energy and the $2$-Wasserstein distance with respect to an abnormal concentration of points at a given place. In short, the cost of such situations in terms of Coulomb energy is larger than the cost in $2$-Wasserstein distance by a logarithmic factor, which gives us enough freedom to construct counter-examples. In a related direction, it is observed in \cite[\corT{Cor. 3.8}]{HUPL} that finiteness of the $(2+\epsilon)$-Wasserstein distance for some $\epsilon > 0$ is enough to imply $\SHU$ and thus, by our results, $\Coul$.

However, if we assume that the points of $\bbX$ \emph{never} gather too much, in the sense of a uniform, almost sure bound on the local density of points for $\bbX$, then we get a converse statement. 
\begin{theorem}[$\WD$ plus density bound implies $\EU$]
\label{theo:WDdensity}
Let $\bbX$ be a two-dimensional stationary point process satisfying $\WD$. Assume that there exists $\brho \geq 1$ such that ($\B_1(x)$ being the unit disk around $x$):
   \begin{equation}
   \label{eq:densitybound}
   \sup_{x \in \Rd} \left|\bbX \cap \B_1(x)\right| \leq \brho \text{ almost surely.}
   \end{equation}
Then $\EU$ is satisfied and in fact, for some constant $C$ not depending on $\bbX$, we have:
\begin{equation}
\label{eq:CoulWwbrho}
\Coul_1(\bbX) \leq C \brho \left(  \Ww^2(\bbX, \Leb) + 1\right)
\end{equation}
\end{theorem}
Clearly, $B_1(x)$ can be replaced in \eqref{eq:densitybound} by $B_r(x)$ for any fixed $r$.
We prove Theorem \ref{theo:WDdensity} in Section \ref{sec:proofWDdensity}. It can be thought of as the infinite-volume version of inequalities relating the $H^{-1}$-norm of the difference of two measures to their $2$-Wasserstein distance under certain density assumptions as in \cite[Prop.~3.1]{loeper2006uniqueness}.

\subsubsection*{From a slight reinforcement of hyperuniformity to finite Coulomb energy}
In dimension $2$, we have $\Coul \implies \HU$ (via $\WD$, see \eqref{EUWDHU}) but some point processes that are only very weakly hyperuniform (in the sense that their rescaled number variance $\sigmaX(r)$ goes to $0$ very slowly as $r \to \infty$) do not have finite Coulomb energy, see Theorem \ref{theo:counter}. In order to formulate a converse statement, we introduce the following property, which is slightly stronger than hyperuniformity itself.
\begin{definition}[$\SHU$]
\label{def:SHU}
We say that a stationary point process satisfies $\SHU$ when the rescaled number variance $\sigmaX(r)$ (see \eqref{def:HU}) not only goes to $0$ as $r \to \infty$ but does so \emph{in such a way that}: 
\begin{equation}
\label{eq:SHU}
\sum_{n \geq 0} \sigmaX(2^n) < + \infty.
\end{equation}
\end{definition}
Note that if we know a priori that $r \mapsto \sigmaX(r)$ is non-increasing, then convergence of the series implies hyperuniformity in the usual sense, however we cannot rule out\footnote{\corT{One of the referees pointed out that for the stationary version of $\Z^2$, $\sigmaX$ is explicit and its main order is in fact non-monotonous in $r$.}} non-monotonous behaviors for $\sigmaX$ hence we prefer to define $\SHU$ as $\HU$ \emph{plus} said convergence.

We call condition \eqref{eq:SHU} a “slight reinforcement of hyperuniformity” because any point process such that $\sigmaX(r) \to 0$ as some negative power of $r$, or even as $(\log r)^{-c}$ with $c > 1$, satisfies $\SHU$. However, it is a fact that $\SHU$ \emph{is strictly stronger than} $\HU$ since there exists point processes such that $\lim_{r \to \infty} \sigma(r) = 0$ but $\liminf_{r \to \infty} \sigma(r) \log r > 0$, see \cite[Thm. 3]{HUPL} - amusingly, such a behavior is never discussed in the statistical physics literature about hyperuniformity, presumably because those decays are so slow that they are deemed “unphysical”. 

In dimension $2$, we find that making the ever-so-slightly stronger assumption $\SHU$ instead of $\HU$ is enough to deduce $\EU$ - in fact $\SHU$ is \emph{equivalent} to finiteness of the regularized Coulomb energy.
\begin{theorem}[$\SHU \iff \EU$]
\label{theo:SHUEU}
Let $\bbX$ be a two-dimensional stationary point process.
\begin{enumerate}
   \item If $\bbX$ satisfies $\SHU$, then it has finite regularized Coulomb energy i.e.\ satisfies $\EU$. More precisely, there exists a constant $C$ (not depending on $\bbX$) such that:
\begin{equation}
\label{Coul1Sigma}
\Coul_1(\bbX) \leq C\left(\sum_{n \geq 0} \sigmaX(2^n) + 1 \right).
\end{equation}
\item Conversely, $\EU$ implies $\SHU$.
\end{enumerate}
\end{theorem}
The proof of Theorem \ref{theo:SHUEU} relies on a recent result by Sodin-Wennman-Yakir and goes by showing that $\SHU \iff \SC \iff \EU$, where $\SC$ is a “spectral condition” introduced in \cite{sodin2023random}, see Section \ref{sec:Sodin}. 
\begin{remark}
Our proof of the implication $\EU \implies \SHU$ uses results which are not explicitly written down in a quantitative way in \cite{sodin2023random}. However, an inspection of \cite{sodin2023random} together with our Lemmas~\ref{lem:equivSCSHU} and \ref{lem:SodinEnough} suggests that one could write the implication $\EU \implies \SHU$ as an effective inequality of the form:
\begin{equation*}
\sum_{n \geq 0} \sigmaX(2^n) \leq C\left(\Coul_1(\bbX) + 1 \right).
\end{equation*}
\end{remark}

\subsubsection*{Applications \corT{and open questions}}
We now mention some consequences of our results (in dimension $2$).
\begin{itemize}
   \item Since the Poisson point process is not hyperuniform, it must have infinite regularized Coulomb energy, and thus also infinite renormalized Coulomb energy in the sense of Sandier-Serfaty (see Section \ref{sec:CoulombEnergy}), which answers a question raised in \cite[Sec.~3.2]{leble2017large}. This could also be deduced directly from the results of \cite{sodin2023random}.

   \item Following the phenomenological classification of Torquato \cite{torquato2018hyperuniform}, all hyperuniform point processes of class-I (i.e.\ $\sigmaX(r) = \O(r^{-1})$) , class-II (i.e.\ $\sigmaX(r) = \O(\log r \times r^{-1})$) and class-III (i.e.\ $\sigmaX(r) = \O(r^{-\alpha})$ with $\alpha \in (0,1)$) satisfy $\SHU$ and thus have both finite regularized Coulomb energy and finite $2$-Wasserstein distance to the Lebesgue measure. In particular, they can all be written as the perturbation of a stationary lattice by stationary displacements with finite second moment - in total agreement with the phrasing of \cite{lachieze2024hyperuniformity}: \emph{“most planar hyperuniform point processes are $L^2$-perturbed lattices”}. Of course, by the triangular inequality, this also means that they are at finite  $2$-Wasserstein distance of each other.

   \item Minimizers of the $2d$ Coulomb gas free energy functional introduced in \cite{leble2017large} all have finite renormalized Coulomb energy (by definition), hence they have finite regularized Coulomb energy (which is a less stringent condition), and are thus hyperuniform - in fact, they even satisfy $\SHU$.

   Note, however, that those point processes describe the limiting behavior of objects named “empirical fields” in the statistical mechanics literature, which are defined as “averages” (over translations) of the finite $2d$ Coulomb gas. It is not at all obvious whether they coincide or not with the actual limit points (in the space of point processes) of the finite random point configuration corresponding to the Gibbs measure of a $2d$ Coulomb gas \emph{without spatial averaging} (the mere existence of those limit points is hard to prove see \cite[Cor.~1.1]{armstrong2021local} \corT{and \cite[Cor. 1.3]{thoma2024overcrowding}}). Hyperuniformity in this case has been recently proven \cite{leble2021two} - interestingly, the upper bound on the rescaled number variance given there is not enough to ensure $\SHU$ and finiteness of the Coulomb energy for those limit points is yet an open question (although a positive answer is very plausible).

   \item Hardcore point processes, having an almost sure uniform lower bound on the distance between two points (one can e.g. obtain such systems by considering Gibbs point processes with hardcore interactions), satisfy a density bound of the form \eqref{eq:densitybound}, and our Theorem \ref{theo:WDdensity} applies to them.

   \item \corT{Going beyond the \emph{uniform} density bound assumption, it would be interesting to consider the following question: if $\bbX$ satisfies $\sup_{x \in \Lambda_n} |\bbX \cap \B_1(x)| \leq C \log^{\O(1)}(n)$ almost surely for some constant $C$, does the implication $\WD \implies \Coul$ remain true? This milder assumption is satisfied by a much wider class of point processes (the Ginibre ensemble, zeroes of the Gaussian Entire Function, etc). We thank the referee for this suggestion.}
\end{itemize}

\subsection{Hyperuniformity}
\label{sec:HU}
The notion of \emph{hyperuniform} or \emph{super-homogeneous} distributions of points originated in the theoretical chemistry and statistical physics literature twenty years ago \cite{gabrielli2002glass,torquato2003local}. We refer to the more recent survey by S. Torquato \cite{torquato2018hyperuniform} for many examples of hyperuniform processes and for an overview of the considerable literature. A simple way to think of the Definition \eqref{def:HU} is to say that for a random point configuration to be hyperuniform means that its number variance (in large balls) is asymptotically much smaller than that of a Poisson point process. \\
One reason that makes hyperuniform systems interesting is that they can exhibit a form of “order within disorder”. Indeed, the smallness of the number variance expressed by \eqref{def:HU} is a form of order which is not necessarily coming from any long-range persistence of correlations within the system. For instance:
\begin{enumerate}
   \item The Poisson point process: is not hyperuniform, has absolutely no correlations, it represents “disorder”.
   \item The stationary lattice: is “as hyperuniform as it gets”, possesses strong correlations, it represents “order”.
   \item The Ginibre ensemble (a two-dimensional point process arising from non-Hermitian random matrix theory \cite{ginibre1965statistical}): is “as hyperuniform” as a lattice and yet has a (truncated) $2$-point correlation function that decays faster than any exponential, it represents “order within disorder”.
\end{enumerate}

\subsection{Wasserstein distances for stationary point processes}
\label{sec:DefWass}
\newcommand{\PZ}{\mathsf{P}^{\Z^\d}}
\newcommand{\WWW}{\mathrm{W}}
\newcommand{\w}{\mathsf{w}}
\newcommand{\Action}{\mathcal{A}}
\renewcommand{\rm}{\mathrm{m}}
\newcommand{\mm}{\mathbf{m}}

We start by a brief reminder on some notions of “finite-volume” optimal transport that will be useful for us. This is by now a classical topic for which there exists numerous expository texts, e.g.\ \cite{Sa15, Villani2, Figalli_2023}.  We then recall the definition (and some properties) of Wasserstein distances in \emph{infinite-volume} for stationary point processes as recently developed in \cite{erbar2023optimal}.

\subsubsection*{Wasserstein distance in finite volume}
Let $\mm_0, \mm_1$ be two measures on $\Rd$ with the same finite mass. We denote by $\Cpl(\mm_0, \mm_1)$ the set of all \emph{couplings} between $\mm_0$ and $\mm_1$, i.e.\ the set of all  measures on $\Rd \times \Rd$ having $\mm_0, \mm_1$ as marginals. We recall that for  $p \in [1, + \infty)$, the “usual” $p$-Wasserstein distance between $\mm_0$ and $\mm_1$ is defined as:
\begin{equation}
\label{def:WWW}
\WWW_p(\mm_0, \mm_1) := \left(\inf_{\Pi \in \Cpl(\mm_0, \mm_1)} \int_{\Rd \times \Rd} |x-y|^p \dd \Pi(x,y)\right)^{1/p}.
\end{equation}
We will mostly deal here with the $2$-Wasserstein distance, associated to the “quadratic cost” $|x-y|^2$.

\subsubsection*{The Benamou-Brenier formulation of optimal transport}
Let $\La \subset \Rd$ be an hypercube (or in fact any convex set, or the whole $\Rd$) and let $\mm_0$ and $\mm_1$ be two probability measures on $\La$ possessing bounded densities $\rm_0, \rm_1$ with respect to the Lebesgue measure. We call “an admissible pair” any family $(\mm_t, v_t)_{t \in [0,1]}$ (indexed by the “time” parameter $t$), where $t \mapsto \mm_t$ is an absolutely continuous family of probability measures on $\La$ interpolating between $\mm_0$ and $\mm_1$, with bounded densities $t \mapsto \rm_t$, and $t \mapsto v_t$ is a measurable family of vector fields on $\Rd$ (the \emph{velocity fields}), such that:
\begin{enumerate}
   \item The pair $(\mm_t, v_t)$ solves the \emph{continuity equation}
\begin{equation}\label{continuity}
\frac{\dd}{\dd t} \rm_t + \div(\rm_t v_t) = 0 
\end{equation}
in the sense of distributions, i.e.\ for all smooth compactly supported test function $f : (0,1) \times \Rd \to \R$ we have:
\begin{equation*}
\int_{0}^1 \left(\int_{\Rd} \partial_t f(t, x) \rm_t(x) \dd x \right) \dd t + \int_{0}^1 \left(\int_{\Rd} \nabla_x f(t, x) \cdot v_t(x) \rm_t(x) \dd x \right) \dd t = 0.
\end{equation*}
\item The Neumann condition $\rm_t v_t \cdot \vec{n} = 0$ is satisfied on $\partial \La$ for all $t \in [0,1]$.
\end{enumerate}
The \emph{action} of such a pair is then defined by:
\begin{equation}
\label{action}
\Action\left((\mm_t, v_t)_{t} \right) := \int_{0}^1 \left(\int_{\La} |v_t|^2 \dd \mm_t\right) \dd t.
\end{equation}
The celebrated Benamou-Brenier formula of \cite{benamou2000computational} (see \cite[Sec.\ 5.3]{Sa15} for a presentation) asserts that the Wasserstein distance between $\mm_0$ and $\mm_1$ can be achieved as the infimum of the action over admissible pairs:
\begin{equation}
\label{eq:BenamouBrenier}
\WWW^2_2(\mm_0, \mm_1) = \inf_{(\mm_t, v_t)_{t}  \text{ admissible}} \Action\left((\mm_t, v_t)_{t} \right) .
\end{equation}
It is known (\cite[Corollary 19.5]{Villani2}) that for instance given second moments of $\mm_0, \mm_1$ the infimum is attained and the minimizer satisfy the following density bound :
\begin{equation}
\label{eq:uniform_dens_bound}
\sup_{t \in [0,1]} \|\rm_t\|_{L^\infty} \leq \max\left( \|\rm_0\|_\infty, \|\rm_1\|_{\infty} \right).
\end{equation}
Moreover, the continuity equation and Benamou-Brenier formula localise to convex subsets of $\Lambda$ and $m_tv_t$ admit internal (and external) traces on $\partial B_r$ (for $r > 0$) such that we can safely consider $m_tv_t$ on $\partial B_r$, see \cite[Sec.\ 2.3]{GoHuOt21}.

\subsubsection*{Optimal transport and Wasserstein distances for stationary point processes}

In order to define a notion of Wasserstein distance between invariant point processes or, more generally, invariant random measures, we introduce suitable cost functions, resembling a Wasserstein distance themselves: 

\begin{itemize}
   \item For two Radon measures $\mathrm{M}_0,\mathrm{M}_1$ we let $\Cpl(\mathrm{M}_0,\mathrm{M}_1)$ be the set of all Radon measures on $\Rd\times\Rd$ with marginals $\mathrm{M}_0$ and $\mathrm{M}_1$. Then we define the cost function $\w_p$ as a “Wasserstein cost per unit volume” between $M_0$ and $M_1$:
\begin{equation}
\label{wpperunit}
\w_p(\mathrm{M}_0,\mathrm{M}_1) := \left(\inf_{q \in \Cpl(\mathrm{M}_0,\mathrm{M}_1)} \limsup_{n \to \infty} \frac{1}{n^\d} \int_{\La_n \times \R^2} |x-y|^p \dd q(x,y) \right)^{{\frac{1}{p}}},
\end{equation}
\corT{where $\La_n := \left[-\frac{n}{2}, \frac{n}{2}\right]^\d $.}
\item For laws $\mathsf{P}_0,\mathsf{P}_1$ of shift-invariant random measures, i.e.\ stationary probability measures on the space of Radon measures, we define (as in \cite{erbar2023optimal}) their $p$-Wasserstein distance by
\begin{equation}
\label{eq:Wass_p_def}
\Ww_p(\mathsf{P}_0, \mathsf{P}_1) := \left(\inf_{Q\in\mathsf{Cpl}_s(\mathsf{P}_0,\mathsf{P}_1)} \int \w^{p}_p(\mathrm{M}_0,\mathrm{M}_1)\ Q(d\mathrm{M}_0,d\mathrm{M}_1) \right)^\frac{1}{p},
\end{equation}
where $\mathsf{Cpl}_s(\mathsf{P}_0,\mathsf{P}_1)$ denotes the set of all couplings between $\mathsf{P}_0$ and $\mathsf{P}_1$ which are stationary under the diagonal action of $\Rd$, i.e.\ $Q=Q\circ(\theta_z,\theta_z)^{-1}$ with $(\theta_z\mathrm{M})(A)=\mathrm{M}(A+z)$ for all Radon measures $\mathrm M$, Borel sets $A\subset \Rd$ and $z\in\Rd$.
\end{itemize}

We are particularly interested in the distance between (the law of a) random point configuration $\bbX$, or of its “spread-out” version $\bbX_\eta$ (see \eqref{def:deta}) to either the Lebesgue measure $\Leb$ (seen as a deterministic Radon measure) or to the “stationary lattice” $\PZ$ defined as (the law of) $\sum_{x\in \mathbb{Z}^\d} \delta_{x+\tau}$ (with a $[0,1]^\d$-uniform random variable $\tau$). We will often make a slight abuse of notation and write $\Ww_p(\bbX, \Leb)$ or $\Ww_p(\bbX, \PZ)$ although $\bbX$ is a random measure, $\Leb$ is a deterministic measure, and $\PZ$ \emph{the law of} a random measure. The definition of $\Ww_p(\bbX,\Leb)$ reduces to:
\begin{equation}
\label{WassLeb}
\Ww_p(\bbX,\Leb) = \E \left[ \w^p_p(\bbX, \Leb) \right]^{1/p},
\end{equation}
with $\w_p$ the $p$-cost “per unit volume” between a realization of $\bbX$ and the Lebesgue measure as in \eqref{wpperunit}.
Let us make two additional comments:
\begin{itemize}
   \item The perturbed lattice and the Lebesgue measure are at finite $p$-Wasserstein distance of each other for all $p$ (in fact for $p = + \infty$) hence, by the triangular inequality, it is equivalent to be at finite distance of either.
   \item Thanks to the alternative representation of the Wasserstein distance between point processes as the cost of optimal \emph{matchings} \corT{given in} \cite[Prop. 2.11.]{erbar2023optimal}, the condition $\Ww_p(\bbX,\PZ) < + \infty$ is equivalent to the fact that $\bbX$ can be written as a \emph{perturbed lattice} i.e. as the point process $\bbX = \sum_{x \in \Zd} \delta_{x + v_x}$ where the displacements $(v_x)_{x \in \Zd}$ are random perturbations whose law is invariant under shifts in $\Zd$ and such that \emph{the $p$-th moment of $v_0$ is finite.}
\end{itemize}

\paragraph{Summary.}
We end this section with a quick summary of the various Wasserstein distances, the corresponding notation, and some relevant properties from \cite{erbar2023optimal}:
\begin{itemize}
   \item $\WWW_p$ denotes the usual $p$-Wasserstein distance between two finite measures on $\R^n$ (see \eqref{def:WWW}).
   \item $\w_p$ is the $p$-Wasserstein distance \emph{per unit volume} between two point configurations, or more generally infinite Radon measures (see \eqref{wpperunit}).
   \item $\Ww_p$ is the $p$-Wasserstein distance between (the law of) two random point configurations (or two random Radom measures), defined in \eqref{eq:Wass_p_def}.
   \item { If $\Ww_p(\mathsf{P}_0, \mathsf{P}_1)<\infty$, the infimum in \eqref{eq:Wass_p_def} is attained \cite[Proposition 2.10]{erbar2023optimal}. Moreover:
   \begin{itemize}
   \item  $\Ww_p$ defines a geodesic metric between invariant random measures of equal intensity \cite[Lemma 2.13, Proposition 2.15]{erbar2023optimal}
   \item  $\Ww_p$ is lower semicontinuous w.r.t. weak convergence (where the set of Radon measures over $\R^d$ is equipped with the vague topology) \cite[Proposition 2.10]{erbar2023optimal}.
   \end{itemize}}
\end{itemize}

\begin{remark}
Connections between the usual $2$-Wasserstein distance $\WWW_2$ and the Coulomb energy (which is basically a $H^{-1}$ norm) have been studied before. A finite-volume $\WWW_2$-Coulomb inequality is given \cite[Thm. 1]{steinerberger2021wasserstein} for points on a compact manifold - it is sharp for $\d \geq 3$, but off by some logarithmic factor in dimension $\d =2$.
\end{remark}

\subsection{Infinite-volume Coulomb energy}
\newcommand{\g}{\mathsf{g}}
\newcommand{\bXN}{\bX_N}
\newcommand{\XN}{\mathrm{X}_N}
\newcommand{\HN}{\mathrm{H}_N}
\newcommand{\EXN}{\El}
\label{sec:CoulombEnergy}

We let $\g$ be (a certain multiple of) the Coulomb kernel in dimension $\d$, namely:
\begin{equation*}
\g(x) := \begin{cases}
- \log |x| & (\d = 2)\\
|x|^{-(\d-2)} & (\d \geq 3)
\end{cases}.
\end{equation*}

\subsubsection*{From finite to infinite systems}
We recall the basic setup for (one-component) Coulomb systems in finite volume. Let $N \geq 1$ and let $\XN = (x_1, \dots, x_N)$ be a $N$-tuple of points in the box $\La_N := \left[-\frac{N}{2}, \frac{N}{2}\right]^\d$. The Coulomb energy of $\XN$ can be defined as:
\begin{equation}
\label{finiteHN}
\HN(\XN) := \hal \iint_{(x,y) \in \La_N \times \La_N, x \neq y} \g(x-y) \dd \left( \bXN - \Leb\right)(x) \dd \left( \bXN - \Leb \right)(y),
\end{equation}
with $\bXN := \sum_{i=1}^N \delta_{x_i}$ a finite point configuration. The usual physical interpretation of \eqref{finiteHN} is that the $N$ point charges $x_1, \dots, x_N$, all living in $\La_N$, interact both with each other and with a “neutralizing background” given by the Lebesgue measure on the box $\La_N$. The energy $\HN$ is bounded below by a constant times $N$, but it is not bounded above due to the singularity of $\g$.

The study of systems with pairwise Coulomb interactions is an old topic. However, due to the long-range nature of the Coulomb kernel, properly defining the Coulomb energy (per unit volume) for an infinite point configuration, or for stationary point processes is challenging. 

\subsubsection*{Infinite systems}
The notion of a “renormalized energy” was introduced by Sandier-Serfaty \cite{sandier2012ginzburg}, and applied to two-dimensional Coulomb systems in \cite{sandier20152d}, with various extensions (see in particular \cite{rougerie2016higher} for higher dimensional Coulomb gases and \cite{petrache2017next} for an generalization to Riesz gases). It provides a/the framework to work with infinite Coulomb systems. 

Before presenting this formalism, let us briefly evoke what a “naive” approach could look like. In view of the finite interaction energy \eqref{finiteHN} and of standard practice in statistical mechanics (for \emph{short-range} interactions!) one could (try to) define the Coulomb energy per unit volume of a point configuration $\bX$ as:
\begin{equation}
\label{naive1}
\lim_{N \to \infty} \frac{1}{|\La_N|} \ \hal \iint_{(x,y) \in \La_N \times \La_N, x \neq y} \g(x-y) \dd \left( \bX - \Leb\right)(x) \dd \left( \bX - \Leb \right)(y),
\end{equation}
or perhaps more accurately as (note the change in the domain of integration):
\begin{equation}
\label{naive2}
\lim_{N \to \infty} \frac{1}{|\La_N|} \ \hal \iint_{(x,y) \in \La_N \times \Rd, x \neq y} \g(x-y) \dd \left( \bX - \Leb\right)(x) \dd \left( \bX - \Leb \right)(y).
\end{equation}
The second expression is in fact closer to the truth, as it takes into account the effect of the whole configuration on the points in $\La_N$, but precisely because of this, and the long-range nature of the Coulomb kernel, the finiteness of this expression for fixed $N$ is very much unclear and might only be obtained through a careful control of the overall repartition of charges. Ultimately, it is simply not possible in general to rigorously derive “infinite-volume” energies of the type \eqref{naive2} from the finite $N$ energy \eqref{finiteHN}.

The general spirit of the works by Serfaty et al. can be phrased as an “electric approach”. Instead of working directly with the pairwise interaction energy $\g(x-y)$ between points as in \eqref{naive1}, \eqref{naive2}, one puts the emphasis on “electric fields”, which are intrinsically global objects whose energy is nonetheless easy to localize. Near each point of the configuration, those fields have a singularity which needs to be taken care of, hence the adjective “renormalized”.

In the rest of Section \ref{sec:CoulombEnergy}, we recall the main lines of the “electric approach”, our goal being to define the “regularized Coulomb energy” $\Coul_\eta(\bbX)$ ($\eta \in (0,1)$) for a stationary point process $\bbX$, which is presented in Definition \ref{def:RegularizedEnergy}.

\subsubsection*{Electric approach in finite volume}
The starting point of the “electric” approach to \emph{infinite} Coulomb systems is the following consideration valid for \emph{finite} systems. Let us introduce the electric potential $\Phi$ and the \emph{true electric field}~$\EXN$ generated by $\XN$ as:
\begin{equation}
\label{eq:true}
\Phi : x \mapsto \int_{y \in \La_N} \g(x-y)  \dd \left( \bXN - \Leb \right)(y), \quad \EXN(x) := \nabla \Phi = \int_{y \in \La_N} \nabla \g(x-y)  \dd \left( \bXN - \Leb \right)(y).
\end{equation}
Note that $\Phi$ is a scalar field, while $\EXN$ is a (gradient) vector field. The “true electric field” satisfies the distributional identity:
\begin{equation}
\label{eq:divEXN}
- \div(\EXN) = \cd \left( \bXN - \Leb \right), \text{ with } \EXN(x) \to 0  \text{ as } |x| \to \infty
\end{equation}
for some constant $\cd$ depending on $\d$ (the value is $2\pi$ for $\d = 2$).
Returning to \eqref{finiteHN}, integrating by parts, and neglecting the condition $\{x \neq y\}$ in the integral, one would get:
\begin{equation}
\label{Heuris}
\HN(\XN) \approx - \frac{1}{2 \cd} \int_{\Rd} \Phi\ \div(\EXN) \approx \frac{1}{2 \cd} \int_{\Rd} |\EXN|^2.
\end{equation}
The heuristic identity \eqref{Heuris} does not make sense as such because $\El$ fails to be in $L^2$ due to the point charges self-interactions (which are avoided in \eqref{finiteHN} by imposing the condition $\{x \neq y\}$). However, \eqref{Heuris} points to the fact (which is old and common knowledge in physics) that the Coulomb interaction energy can be phrased in terms of the integral of an “electric” energy density $|\EXN|^2$, where $\EXN$ is the right “electric field” satisfying \eqref{eq:divEXN}.

\subsubsection*{Electric fields: global, local, screened}
We refer to vector fields in $\cap_{p \in [1, 2)} L^p_{loc}(\Rd, \Rd)$ as (global) “electric fields”.  If $\bX$ is a point configuration in $\Rd$, and $\El$ is an electric field, we say that $\El$ is compatible with $\bX$ when the following identity holds:
\begin{equation}
\label{def:compatible}
- \div(\El) = \cd \left(\bX - \Leb \right).
\end{equation}
We also work with \emph{local} electric fields: if $\La$ is a square in $\Rd$ and $\bX$ a point configuration in $\La$, we say that $\El \in \cap_{p \in [1, 2)} L^p_{loc}(\La, \Rd)$ is a (local) electric field compatible with $\bX$ when \eqref{def:compatible} holds in $\La$. To be clear, by \eqref{def:compatible} we mean that
\begin{equation*}
\int \varphi(x) \dd \left(\bX - \Leb\right)(x) = - \frac{1}{\cd} \int \nabla \varphi \cdot \El
\end{equation*}
for all $\varphi \in C^{\infty}_c(\Rd)$ (global electric fields), or for all $\varphi \in C^{\infty}_c(\La)$ (local electric fields in $\La$). We say that a local electric field $\El$ is \emph{screened} when it has Neumann boundary conditions i.e.\ its normal component satisfies 
\begin{equation}
\label{def:screened}
\El \cdot \vec{n} = 0 \text{ along $\partial \La$}.
\end{equation}
\subparagraph{The importance of screened fields}
Screened fields are of outmost importance because they can be glued together in the following sense: if $\La^1, \La^2$ are two squares with one side in common, $\bX^1, \bX^2$ are two point configurations in $\La^1, \La^2$ and $\El^1, \El^2$ are two local electric fields compatible with $\bX^1, \bX^2$, then defining:
\begin{equation*}
\bX^t := \bX^1 + \bX^2, \quad \El^t := \El^1 \1_{\La_1} + \El^2 \1_{\La_2} 
\end{equation*}
yields a “total” point configuration and a “total” electric field, but \emph{in general, $\El^t$ is not compatible with $\bX^t$}. However, if $\El^1, \El^2$ are both screened, or more generally if \emph{their normal components agree on the common side of $\La_1, \La_2$}, then $\El^t$ is compatible with $\bX^t$.

\subparagraph{The role of abstract fields}
An important remark is that we deal with “abstract” fields $\El$, not necessarily of the form (cf.\ \eqref{eq:true}):
\begin{equation*}
\El(x) \overset{?}{=} \int_{\Rd} \nabla \g(x-y) \dd \left( \bX - \Leb\right)(y),
\end{equation*}
which could be a candidate to solve \eqref{def:compatible}, but whose convergence is unclear (note that $\nabla \g(x) \propto \frac{1}{|x|^{\dd - 1}}$ is far from being integrable at infinity). Considering the whole family of compatible electric fields is a fundamental feature of the approach by Sandier-Serfaty. The price to pay is that in general we do not know “who $\El$ is”, besides that it solves \eqref{def:compatible}.

\newcommand{\deta}{\delta^{(\eta)}}
\newcommand{\ceta}{\chi_\eta}
\newcommand{\bXeta}{\bX_\eta}
\newcommand{\bbXeta}{\bbX_\eta}

\newcommand{\Eleta}{\El_{\eta}}
\newcommand{\bEleta}{\bEl_{\eta}}

\newcommand{\feta}{\mathsf{f}_\eta}

\subsubsection*{Spreading out charges and regularizing point configurations.}
To handle the singularity of the electric field near a point charge, one introduces a short-scale regularization, or “truncation” of the field, and a corresponding operation of “spreading out” the point charges. Let us start with the latter, as it is simpler to describe. The usual way to proceed is as follows (see \cite[Sec. 2.1]{rougerie2016higher}): 
\begin{enumerate}
   \item Take some radial non-negative “bump function” $\chi$ supported on the unit ball, with mass $1$, and for $\eta > 0$ let 
   \begin{equation*}
\ceta := \frac{1}{\eta^\d} \chi\left( \frac{\cdot}{\eta} \right).
   \end{equation*}
   In general, the choice of the bump function $\chi$ is not really important and does not affect the limiting objects or the values of the limiting quantities as $\eta \to 0$. A common choice is to take $\chi = \1_{\B_1}$, which would work for us here, but we may also take $\chi$ to be smooth. We fix such a $\chi$ for the rest of the paper.
  
   \item If $x \in \Rd$, define $\deta_x$ as a Dirac mass “spread-out” according to $\ceta$, namely as the measure with density $\ceta(\cdot - x)$. If we think of a Dirac mass $\delta_x$ as a point charge placed at $x$, then $\deta_x$ is a “spread-out” charge at $x$. 

   This operation goes back at least to Onsager \cite{onsager1939electrostatic}, see e.g.\ \cite[Sec. 9.7]{lieb2001analysis} for a study of its properties in the context of Coulomb energies. In short: 
\begin{enumerate}
   \item By Newton's theorem, the field generated by a spread-out charge coincides with the field generated by the original charge outside the support of the “spread” (here $\B_{\eta}$).
   \item Spreading out all charges slightly decreases  the pairwise interaction energy between them. 
   \item After this operation, a single charge self-interaction becomes \emph{finite} - but it blows up proportionally to $\g(\eta)$ as $\eta \to 0$, as one could expect.
\end{enumerate}
   \item If $\bX$ is a point configuration, replacing each Dirac mass $\delta_x$ by its “spread-out” version $\deta_x$ for $x \in \bX$ yields a “regularized” version of $\bX$, denoted by $\bXeta$. 
\end{enumerate}
Formally, we have defined two unambiguous operations:
\begin{equation}
\label{def:deta}
\delta_x \to \deta_x := \ceta \ast \delta_x \text{ (for $x \in \Rd$)}, \quad \bX \to \bXeta := \sum_{x \in \bX} \deta_x.
\end{equation}
Let us emphasize that $\bX \to \bXeta$ is an operation at the level of Radon measures on $\Rd$. The image of a point configuration $\bX$ is no longer purely atomic, it is a locally finite measure with locally bounded density with respect to the Lebesgue measure (the density of one regularized Dirac is bounded by $\|\chi\|_{\infty} \eta^{-\corT{\d}}$ but the density of $\bXeta$ depends also on the number of points per unit volume).

\paragraph{Truncation of the field.} 
We now study a corresponding operation at the level of electric fields. With $\chi$ chosen as before and for $\eta > 0$, let $\feta$ be the mean-zero solution to:
\begin{equation}
\label{def:feta}
- \Delta \feta = \cd \left(\deta_0 - \delta_0\right), \quad \feta \equiv 0 \text{ outside } \B_\eta,
\end{equation}
\corT{which is given by $\feta : x \mapsto \max\left(0, \log \frac{\eta}{|x|} \right)$.} If $\El$ is an electric field (global or local) compatible with $\bX$, we define the “truncated field” $\Eleta$ as:
\begin{equation}
\label{def:Eleta}
\Eleta := \El + \sum_{x \in \bX} \nabla \feta(x - \cdot).
\end{equation}
\begin{claim}[Properties of the truncated field] It is straightforward that:
\label{claim:propertiesOF}
\begin{enumerate}
   \item If $\El$ is compatible with $\bX$ in the sense of \eqref{def:compatible}, then $\Eleta$ is compatible with $\bXeta$ in the sense:
\begin{equation}
\label{divEleta}
- \div(\Eleta) = \cd \left(\bXeta - \Leb \right).
\end{equation}
   \item If $\El$ is a gradient vector field, then so is $\Eleta$.
   \item (Local version.) If $\El$ is compatible with $\bX$ in $\La$ and $\min_{x \in \bX} \dist(x, \partial \Omega) > \eta$, then \eqref{divEleta} holds in~$\La$. Moreover, if $\El$ is “screened”, then so is $\Eleta$.
\end{enumerate}
\end{claim}

\newcommand{\hEleta}{\check{\El}_{\eta}}
\newcommand{\rEleta}{\tilde{\El}_{\eta}}
\newcommand{\bhEleta}{\check{\El}_{\eta}^\sbullet}
\newcommand{\brEleta}{\tilde{\El}_{\eta}^\sbullet}

\newcommand{\hEl}{\check{\El}}
\newcommand{\bhEl}{\check{\El}^\sbullet}
\newcommand{\Pelec}{\mathrm{P}^{\mathrm{elec}}}
\newcommand{\Couleta}{\Coul_{\eta}}
\newcommand{\hCouleta}{\Coul_{\hat{\eta}}}

\subsubsection*{Stationary random electric fields}
We call a random electric field $\bEl$ stationary when its law is invariant under translations $\El \to \El(\cdot - t)$ for all $t \in \Rd$. We say that a random electric field $\bEl$ is compatible with a point process $\bbX$ when the random measure $\frac{1}{\cd} \div(\bEl) + \Leb$ and $\bbX$ have the same distribution. Note however that if $\bbX$ is stationary, it does not imply that so is $\bEl$.

\subsubsection*{Regularized Coulomb energy}
Finally, we may define the notion that will be relevant for us, namely “finite regularized Coulomb energy”.
\begin{definition}
\label{def:RegularizedEnergy}
Let $\bbX$ be a stationary point process. We say that it has \emph{finite regularized Coulomb energy} when there exists $\eta \in (0,1]$ and a \emph{stationary} random electric field $\bEl$ compatible with $\bbX$ such that, using the truncation/regularization presented in \eqref{def:Eleta}: 
\begin{equation}
\label{eq:FiniteRegul}
\E\left[|\bEleta(0)|^2\right] < + \infty.
\end{equation}
Since stationary, compatible electric fields are not unique, we define:
\begin{equation}
\label{def:Couleta}
\Coul_\eta(\bbX) := \inf_{\bEl \text{ stat., comp. with $\bbX$}} \E\left[|\bEleta(0)|^2\right].
\end{equation}
\end{definition}

\begin{lemma}[Two properties of the regularized energy]
\label{lem:twoproperties}
The following holds:
\begin{enumerate}
   \item If $\Coul_\eta(\bbX)$ is finite for some $\eta \in (0,1]$, then the number of points of $\bbX$ in $\BR$ has a finite second moment for all $r > 0$ and more precisely:
   \begin{equation}
   \label{discrenergy}
   \E\left[ \left(|\bbX \cap \BR| - |\BR|\right)^2 \right] \leq C  \left(\Coul_\eta(\bbX) + 1 \right) r^2.
   \end{equation}
   \item If $\Coul_1(\bbX)$ is finite, then so is $\Coul_\eta(\bbX)$ for all $\eta \in (0,1)$
   and we have:
\begin{equation}
\label{CouletaCoul1}
\Coul_\eta(\bbX) \leq C_\eta \left(\Coul_1(\bbX) + 1 \right), \quad \corT{\Coul_1(\bbX) \leq \Coul_\eta + C \eta}.
\end{equation}
\end{enumerate}
\end{lemma}
\begin{proof}[Proof of Lemma \ref{lem:twoproperties}]
The first item (control of the discrepancies by the Coulomb energy) was proven in \cite[Sec. 8.5]{leble2017large}. \corT{For convenience, we sketch the argument - which is essentially the key for one of the counter-examples - in Section \ref{WDnotEU}.}

Now, in view of \eqref{def:Eleta}, we can control the difference $\bEl_{1}(0) - \bEl_{\eta}(0)$ almost surely by:
\begin{equation*}
\left|\bEl_{1}(0) - \bEl_{\eta}(0) \right| \leq |\bbX \cap \B_2| \times \left| \feta - \mathsf{f}_1 \right|_{\infty}, 
\end{equation*}
and thus $\E[|\bEl_{\eta}(0)|^2] \leq C \E[|\bEl_{1}(0)|^2] + C_\eta \E[|\bX \cap \B_2|^2]$, which yields \eqref{CouletaCoul1} using \eqref{discrenergy} (for $\eta =1, r = 2$).

\corT{The last point follows from the fact that $\Coul_\eta$ is “almost” decreasing with respect to the truncation $\eta$, see \cite[Prop. 2.4.]{petrache2017next}.}
\end{proof}

\subsubsection*{Renormalized energy}
For completeness, we conclude by recalling the definition of the “renormalized” energy of Sandier-Serfaty in this context. The renormalized Coulomb energy of a stationary random electric field $\bEl$ is obtained by taking the limit $\eta \to 0$ of $\E\left[|\bEleta(0)|^2\right]$ in a renormalized fashion, namely by substracting a divergent part of order $\g(\eta)$ \corT{(the following limit as $\eta$ to $0$ exists by “almost”-monotonicity, see \cite[Sec. 2.3]{petrache2017next})}:
\begin{equation*}
\Coul(\bEl) := \lim_{\eta \to 0} \left(\E\left[|\bEleta(0)|^2\right] - \cd \g(\eta)\right),
\end{equation*}
and the Coulomb energy of a point process is then again phrased as an infimum:
\begin{equation*}
\Coul(\bbX) := \inf_{\bEl \text{ stat., comp. with  $\bbX$}} \Coul(\bEl).
\end{equation*}
The fact that this is the “correct” notion of an infinite-volume Coulomb energy is supported by a result of $\Gamma$-convergence connecting the finite-volume energy $\HN$ (as defined in \eqref{finiteHN}) to $\Coul$. We refer e.g.\ to \cite[Sec.~2.7~\&~3.]{leble2017large} for more details.

\subsubsection*{“Intrinsic” Coulomb energy}
\newcommand{\Coulint}{\Coul^{\mathrm{int}}}
\newcommand{\rhoD}{\rho_{2}}
\newcommand{\Dlog}{\mathcal{D}_{\mathrm{log}}}
One frustrating aspect of the electric approach to defining infinite-volume Coulomb energies is that it lacks any explicit formula (besides perhaps in the case of lattices). Given a certain point process, computing - or simply bounding above - its average Coulomb energy requires to construct a compatible stationary electric field and then compute its average (renormalized) energy... which sounds overwhelming.

In \cite{leble2016logarithmic}, inspired by computations done in \cite{borodin2013renormalized}, the second author tried to remedy  this issue by introducing a so-called “intrinsic” version of the Coulomb energy for stationary point processes, denoted by $\mathbb{W}^{\mathrm{int}}$ in \cite{leble2016logarithmic} and which we will here denote by $\Coulint$. The main advantage of $\Coulint$ is that it is \emph{completely explicit and depends only on the two-point correlation function of the process}. In the “good cases” (and focusing on dimension $2$):
\begin{enumerate}
   \item The intrinsic energy can be written as:
\begin{equation}
\label{eq:CoulintNice}
\Coulint(\bbX) := \int_{\R^2} -\log |v| (\rhoD - 1)(v) \dd v,
\end{equation}
 where $\rhoD$ is the two-point correlation function (or measure) of the process \corT{(informally, $\rhoD(v)$ is the  probability density of finding two points at the locations $0$ and $v$ - a rigorous definition is recalled in \cite[Sec. 2.3]{leble2016logarithmic})}. 
 \item  The following “electric-intrinsic” inequality holds, which provides a condition for finiteness of $\Coul$:
 \begin{equation}
 \label{ElecIntrin}
 \Coul(\bbX) \leq \Coulint(\bbX).
 \end{equation} 
\end{enumerate}
 There are, however, two drawbacks in that regard:
 \begin{enumerate}
    \item The nice expression \eqref{eq:CoulintNice} for $\Coulint$ is only valid when $\rhoD - 1$ has enough integrability at infinity.
    \item In general, in dimension $2$, the “electric-intrinsic” inequality \eqref{ElecIntrin} has to be corrected and written as:
    \begin{equation}
    \label{ElecIntrins}
     \Coul(\bbX) \leq \Coulint(\bbX) + \Dlog(\bbX),
    \end{equation}
    where $\Dlog$ was defined in \cite[Sec. 5]{leble2016logarithmic} as: 
    \begin{equation}
    \label{def:Dlog}
    \Dlog(\bbX) :=  C \limsup_{r \to \infty} \left( \sigmaX(r) \log r\right)
    \end{equation}
   where $\sigmaX$ is the rescaled number variance as in \eqref{def:HU} and $C$ is some constant whose value is irrelevant.
 \end{enumerate}
 It is clear that our condition $\SHU$ (see \eqref{eq:SHU}) is related to the condition $\Dlog = 0$. We will also explain below (see Proposition~\ref{prop:consequencesofmild}) how the “intrinsic” approach to Coulomb energy connects (some of) our results with (some of) the results of \cite{lachieze2024hyperuniformity}. 

\subsection{Connection with the literature}
\label{sec:discussion}
\subsubsection*{Finite Wasserstein distance to Lebesgue implies hyperuniformity}
Since finite Wasserstein distance to the Lebesgue measure or to a stationary lattice are equivalent, let us adopt the latter point of view. It has been known for a long time that:
\begin{enumerate}
   \item Stationary lattices are hyperuniform of class-I i.e.\ $\sigma(r) = \O(r^{-1})$.
   \item Stationary lattices with i.i.d. perturbations admitting a \emph{finite first moment} remain hyperuniform of class-I, as \corT{can be deduced from} \cite{gacs1975problem}, \corT{see also e.g. \cite[Thm. 3]{yakir2021fluctuations} and a discussion in \cite[Sec.~1.5]{yakir2021fluctuations}.}
\end{enumerate}
The question of whether \emph{dependent} perturbations could break hyperuniformity has been raised and studied in the statistical physics literature \cite{gabrielli2004point}, with arguments and computations hinting at the fact that - in dimension $2$ - applying dependent perturbations with \emph{finite second moment} to a stationary lattice should always preserve hyperuniformity. In the recent paper \cite{HUPL} this assertion is proven, and examples are given showing that the moment condition is sharp. In the language of Wasserstein distances, \cite[\corT{Thm. 1}]{HUPL} shows that:
\begin{enumerate}
   \item If $\Ww_2(\bbX, \PZ) < + \infty$ then $\bbX$ is hyperuniform i.e.\ the rescaled number variance $\sigma(r) \to 0$ as $r \to \infty$.
   \item For all $p < 2$, there exists $\bbX$ such that $\Ww_p(\bbX, \PZ) < + \infty$ and yet (!) $\sigma(r) = + \infty$ for all $r > 0$.
\end{enumerate}
It is remarkable that the dimension $2$ plays here a crucial role: for $\d \geq 3$ it is possible to break hyperuniformity with arbitrarily small dependent perturbations. Note also that for $\d = 1$, the situation is in a sense better, as finiteness of the $1$-Wasserstein distance to the stationary lattice is enough to imply hyperuniformity.

\subsubsection*{A partial converse: the result of \cite{lachieze2024hyperuniformity}}
\newcommand{\Cc}{\mathcal{C}}
In the recent paper \cite{lachieze2024hyperuniformity}, the following result is proven (in dimension $2$):
\begin{proposition}[Thm. 2, (ii) in \cite{lachieze2024hyperuniformity}]
\label{prop:LRY}
Let $\bbX$ be a stationary point process such that:
\begin{enumerate}
   \item $\bbX$ is hyperuniform.
   \item The “correlation measure” $\Cc$ of $\bbX$ is well-defined as a signed measure, has finite mass, and satisfies:
\begin{equation}
\label{eq:mild}
\int_{v \in \R^2, |v| \geq 1} \log|v| |\Cc|(\dd v)< + \infty.
\end{equation}
\end{enumerate}
Then the process $\bbX$ can be written as a stationary perturbation, with finite second moment, of a stationary lattice (and thus $\WD$ is satisfied).
\end{proposition}
The “correlation measure” mentioned here is given by $\Cc := (\rhoD - \Leb) + \delta_0$, where $\rhoD$ is the “two-point (or pair) correlation measure” as above, hence the condition \eqref{eq:mild}, which is presented in \cite{lachieze2024hyperuniformity} as \emph{“a mild logarithmic integrability condition on the reduced pair correlation measure”}, can be rephrased as:
\begin{equation*}
\int_{v \in \R^2, |v| \geq 1} \log|v| |\rhoD(v) - 1| \dd v < + \infty
\end{equation*}
(at least when $\rhoD$ has a density with respect to the Lebesgue measure), which is very much reminiscent from \eqref{eq:CoulintNice} and seems to be controlling some kind of logarithmic (hence Coulomb) interaction energy of the underlying point process. Indeed, we find:

\begin{proposition}[Consequences of \eqref{eq:mild}]
\label{prop:consequencesofmild}
Under the assumptions of Proposition \ref{prop:LRY}:
\begin{itemize}
   \item $\bbX$ satisfies $\SHU$, i.e.\ it has finite regularized Coulomb energy, and thus (by our Theorem \ref{theo:EUWD}) finite $2$-Wasserstein distance to Lebesgue/a lattice. 
   \item Moreover, we have $\Dlog(\bbX) = 0$ (see \eqref{def:Dlog}).
\end{itemize}
\end{proposition}
\begin{proof}[Proof of Proposition \ref{prop:consequencesofmild}]
It is known (see \cite[Sec. 1]{coste2021order}) that the rescaled number variance in $\BR$ can be expressed for $r > 0$ as:
\begin{equation}
\label{rescaled}
\sigmaX(r) = 1 + \int_{\R^2} \jr(v) \Cc(\dd v),
\end{equation}
with $\jr := \frac{1}{|\BR|} \1_{\BR} \ast \1_{\BR}$, and that if $\Cc$ is well-defined as a signed measure, then hyperuniformity is equivalent (see \cite[Sec.~2.2]{coste2021order}) to the “sum rule”:
\begin{equation}
\label{HUT2V}
\int_{\R^2} \Cc(\dd v) = -1.
\end{equation}
The function $\jr$ is supported on $\B_{2r}$, with $\jr(0) = 1$, and is $\O(r^{-1})$-Lipschitz, thus it satisfies:
\begin{equation*}
\jr(v) = 1 + \O\left(\frac{|v|}{r} \right) \text{ for $|v| \leq 2r$}, \quad \jr(v) = 0 \text{ for $|v| \geq 2r$}.
\end{equation*}
If $\bbX$ is hyperuniform, we may use \eqref{rescaled} and \eqref{HUT2V} to write:
\begin{multline}
\label{sigmaXT2}
\sigmaX(r) = 1 + \int_{v \in \R^2, |v| \leq 2r} \Cc(\dd v) + \int_{v \in \R^2, |v| \leq 2r} \O\left(\frac{|v|}{r} \right) |\Cc|(\dd v) \\
= - \int_{v \in \R^2, |v| > 2r} \Cc(\dd v) + \int_{v \in \R^2, |v| \leq 2r} \O\left(\frac{|v|}{r} \right) |\Cc|(\dd v).
\end{multline}
 Now, if the condition \eqref{eq:mild} of \cite{lachieze2024hyperuniformity} holds, we can write, on the one hand:
\begin{equation*}
\left| \int_{v \in \R^2, |v| > 2r} \Cc(\dd v)\right| \leq \frac{1}{\log r} \int_{v \in \R^2, |v| > 2r} \log|v| |\Cc|(\dd v) = o\left(\frac{1}{\log r}\right) \text{ as } r \to \infty,
\end{equation*}
and on the other hand:
\begin{multline*}
\int_{v \in \R^2, |v| \leq 2r} \frac{|v|}{r}  |\Cc|(\dd v) \leq \frac{1}{\sqrt{r}} \int_{v \in \R^2, |v| \leq \sqrt{r}}  |\Cc|(\dd v) + \frac{2}{\log(r)} \int_{v \in \R^2,  |v| > \sqrt{r}}  \log|v| |\Cc|(\dd v) \\
= o\left(\frac{1}{\log r}\right) \text{ as } r \to \infty,
\end{multline*}
which gives us $\sigmaX(r)\log r \to 0$ as $r \to \infty$ and thus indeed $\Dlog = 0$ (see \eqref{def:Dlog}). This proves the second item.

Moreover, going back to \eqref{sigmaXT2}, taking $r = 2^n$ and summing over $n \geq 1$, we find:
\begin{multline*}
\sum_{n = 1}^{+ \infty} \sigmaX(2^n) \leq \sum_{n=1}^{+\infty} \int_{v \in \R^2, |v| > 2^{n+1}} |\Cc|(\dd v) + C \sum_{n=1}^{+\infty} \int_{v \in \R^2, |v| \leq 2^{n+1}} \frac{|v|}{2^n}  |\Cc|(\dd v) \\
\leq \int_{v \in \R^2, |v| \geq 1} |\Cc|(\dd v) \left(\sum_{n=1}^{\log |v|} 1 \right)  + C \int_{v \in \R^2} |v| |\Cc|(\dd v) \left(\sum_{n=\log |v|}^{+ \infty} 2^{-n}  \right) \\
\leq  \int_{v \in \R^2, |v| \geq 1} \log|v| |\Cc|(\dd v) + C \int_{v \in \R^2} |\Cc|(\dd v) < + \infty, 
\end{multline*}
because of the assumptions on $\Cc$ and the condition \eqref{eq:mild}. Thus $\bbX$ satisfies indeed $\SHU$.
\end{proof}

\begin{remark}
Since the condition \eqref{eq:mild} (together with hyperuniformity) is found to imply $\SHU$ and since on the other hand we prove that $\SHU$ implies $\EU$, and thus $\WD$, we do recover the result of \cite[Thm.~2,~(ii)]{lachieze2024hyperuniformity}. Of course, here we need to first pass from $\SHU$ to $\EU$, and the proof of this implication uses the analysis of \cite{sodin2023random} which is a much longer route than the one taken in \cite{lachieze2024hyperuniformity} (which is, in spirit, very close to the implication $\EU \implies \WD$). 

Because (as we just proved) the condition of \cite{lachieze2024hyperuniformity} also implies that $\Dlog = 0$, it seems to us that another way to deduce $\EU$ (and thus $\WD$) from it would be to return to the analysis of \cite{leble2016logarithmic} and to find an “intrinsic” formulation for the \emph{regularized} logarithmic energy (which is in fact easier to handle than the renormalized one, so the arguments of \cite{leble2016logarithmic} would readily apply). Without a doubt, since the regularization corresponds to a short-scale truncation of the interaction, one would be led to something like:
\begin{equation*}
\Coulint_\eta(\bbX) := \int_{v \in \R^2} - \widetilde{\log}(v) \ (\rho_2 - 1)(v) \dd v, 
\end{equation*}
with $\widetilde{\log}(v) =  \log|v| + \feta(v)$, which coincides with $\log|v|$ outside $\B_\eta$ and is bounded on $\B_\eta$,  and presumably to an “electric-intrinsic” inequality of the form:
\begin{equation*}
\Coul_\eta(\bbX) \leq \Coulint_\eta(\bbX) + \Dlog(\bbX).
\end{equation*}
Then hyperuniformity together with \eqref{eq:mild} would imply that $\Dlog(\bbX) = 0$ and $\Coulint_\eta(\bbX) < + \infty$, and therefore that $\EU$, and thus $\WD$, are satisfied.
\end{remark}

\subsubsection*{Wasserstein distance versus discrepancies in finite volume: results and questions of \cite{Butez:2024aa}} \newcommand{\bsigma}{\bar{\sigma}}
In the recent preprint \cite{Butez:2024aa}, the authors estimate the Wasserstein distance in large square boxes between two-dimensional point processes and their “mean” in terms of controls on the size of “discrepancies” i.e.\ the difference between the number of points and its expectation in any square box. Among their results, the one closest to our concerns is \cite[Thm.2, case $p=2$]{Butez:2024aa}, which we reformulate here in order to make a comparison easier. Let us consider a stationary point process $\bbX$ in $\R^2$, and assume for convenience that $\bbX$ has intensity $1$, i.e. $\E[|\bbX \cap \B_1|] = |\B_1|$. If $\La$ is a square, we let $\bbX_\La$ be the restriction of the point process to the square $\La$. For $r>0$ let $\La_r=\left[-\frac{r}{2},\frac{r}{2} \right]^2.$
\begin{proposition}[{\cite[Thm.2]{Butez:2024aa}}]
\label{prop:Butez}
 For $r > 0$, let $\bsigma(r)$ be (a non-increasing upper bound on) the rescaled number variance of $\bbX$ in \emph{the square} $\La_r$. Then for all $N > 1$, we have:
\begin{equation}
\label{eq:Butez}
\frac{1}{|\La_N|} \E\left[ \WWW_2^2\left(\bbX_{\La_N}, \frac{|\bbX \cap \La_N|}{|\La_N|} \Leb_{\La_N} \right) \right] \leq C \left(1 + \int_{1}^N \frac{\sqrt{\bsigma(s)}}{s} \dd s \right).
\end{equation}
\end{proposition}
Note that the quantity $\bsigma$ does not control the variance in disks (as $\sigma$ does), but in squares. In general, the notion of hyperuniformity can be shape-dependent (see \cite[Sec. 2.1]{coste2021order} for a discussion), but although the technique of \cite{Butez:2024aa} does use the fact that $\La_N$ is a square, this does not seem like a major difference. For non-increasing $\bsigma$, it is equivalent to have $\int_{1}^{+\infty} \frac{\sqrt{\bsigma(s)}}{s} \dd s < + \infty$ and $\sum_{n \geq 0} \sqrt{\bsigma(2^n)} < + \infty$. Hence the result of \cite{Butez:2024aa} implies:
\begin{equation}
\label{butez2}
\text{If } \sum_{n \geq 0} \sqrt{\bsigma(2^n)} < + \infty, \text{ then } \limsup_{N \to \infty} \frac{1}{|\La_N|}  \E\left[  \WWW_2^2\left(\bbX_{\La_N}, \frac{|\bbX \cap \La_N|}{|\La_N|} \Leb_{\La_N} \right) \right] < + \infty.
\end{equation}
The assumption on the left-hand side of \eqref{butez2} looks very similar to $\SHU$, with one minor difference (computing the number variance on squares rather than on disks) and a more significant one which is the presence of a square root (since $\bsigma \to 0$, we have $\sqrt{\bsigma} \gg \bsigma$). Of course, if $\bsigma$ behaves as a power-law (which is by far the most common case among hyperuniform point processes) then both $\sum_{n \geq 0} \sqrt{\bsigma(2^n)}$ and $\sum_{n \geq 0} \bsigma(2^n)$ will be finite, but strictly speaking the former assumption is a bit stronger. The conclusion in the right-hand side of \eqref{butez2} is reminiscent of the definitions \eqref{wpperunit}, \eqref{eq:Wass_p_def} given in Section \ref{sec:DefWass}.
Indeed, arguing as for instance in \cite[Section 4]{HuSt13}, see in particular Corollary 4.5, the right hand side of \eqref{butez2} implies finiteness of $\Ww_2^2(\bbX, \Leb)$.

Let us turn to questions raised in \cite{Butez:2024aa}. The authors ask: \emph{“Can we find conditions on two-dimensional point processes for which the Wasserstein distance is of the optimal order?”}, to which our Theorems \ref{theo:EUWD} and \ref{theo:SHUEU} answer that (in the case of stationary point processes) $\SHU$, or equivalently $\EU$, are sufficient conditions for having $2$-Wasserstein distances of optimal order.

Moreover, we show in Lemma \ref{lem:costPoissonHU} the existence (for any $p\geq 1$) of a hyperuniform point process $\bbX$ such that $\Ww_p(\bbX,\Leb)=\infty$, which provides a negative answer to the possibility raised \cite[Sec. 4]{Butez:2024aa} that \emph{“it may still be true that a general hyperuniform point process attains the optimal transport cost”}.

\subsubsection*{$L^p$ version of Theorem \ref{theo:EUWD} }
As explained above, Theorem \ref{theo:EUWD} can be thought of as an infinite volume equivalent of the classical $\WWW_2-H^{-1}$ inequality. This inequality has natural $L^p$ versions, the $\WWW_p-H^{-1,p}$ inequalities, see \cite[Theorem 2]{Ledoux_2019}. Hence, one could expect that in our setup $\Ww_p(\bbX,\Leb)$ should be upper bounded by $\E[| \bEl_1(0)|^p] +1 $ for any electric field compatible with $\bbX$. However, there is no obvious interpretation of $\E[| \bEl_1(0)|^p]$ so the case  $p=2$ is the interesting and natural one for us - thus we formulated our results for $p=2$ only. \corT{In this vein, one should however mention the paper \cite{sodin2010uniformly}, which (roughly speaking) studies  the case $p = + \infty$ of a very closely related problem. In \cite[Thm 1.4]{sodin2010uniformly}, it is shown that finiteness of the $\WWW_\infty$ transport distance to the Lebesgue measure is equivalent to existence of a compatible electric field with some good $L^\infty$ properties.}

It is also worth noting that the Riesz energy (when the Coulomb potential $\g$ is replaced by a power law $\frac{1}{|\cdot|^s}$ for some $s$) does not seem to have a natural connection to Wasserstein distances.

\subsection{Notation, plan of the paper}
We let $\La_R$ be the square $[-\hal R, \hal R]^\dd$ and write $\Leb_{A}$ for the restriction of the Lebesgue measure to the set $A$. 

We denote deterministic point configurations (resp.\ fields) with bold face letters as $\bX$ (resp.\ $\El$) and random point configurations (resp.\ fields) with additional bullets as superscripts as $\bbX$ (resp.\ $\bEl$). 

We sometimes write $A \preceq B$ to express the fact that $A$ is smaller than $B$ times some universal constant. 

\subsubsection*{Plan of the paper.}
\begin{itemize}
   \item In Section \ref{sec:E1W2} we prove that finite regularized Coulomb energy implies finite Wasserstein distance in any dimension (Theorem \ref{theo:EUWD}).
   \item In Section \ref{sec:proofWDdensity} we prove that finite Wasserstein distance plus a uniform density bound implies finite regularized Coulomb energy (Theorem \ref{theo:WDdensity}). We introduce (see Section \ref{sec:ScreenWass}) a new procedure of “screening in Wasserstein space”, strongly inspired by the “screening” lemma of Sandier-Serfaty, which is a key tool for the study of Coulomb gases.
   \item In Section \ref{sec:Sodin}, we prove the equivalence between $\SHU$ and $\EU$ (Theorem \ref{theo:SHUEU}). To do so, we rely on the results of \cite{sodin2023random}, which we present, and we show that both $\SHU$ and $\EU$ are equivalent to the “spectral condition” put forward in \cite{sodin2023random}.
   \item Finally, in Section \ref{sec:constructions}, we construct simple counter-examples to the implications $\WD \implies \EU$ and $\HU \implies \WD$ (Theorem \ref{theo:counter}).
\end{itemize}

\newcommand{\Cost}{C_2}
\newcommand{\bXint}{\bX_{\eta_0}^{\mathrm{int}}}
\renewcommand{\rm}{\mathrm{m}}
\newcommand{\Ralpha}{\mathcal{R}_\alpha}
\newcommand{\bXext}{\bX^{\mathrm{ext}}}
\newcommand{\muext}{\mm^{\mathrm{ext}}}
\newcommand{\tmu}{\tilde{\mm}}
\newcommand{\rmuext}{\rm^{\mathrm{ext}}}
\newcommand{\Eext}{\El^{\mathrm{ext}}}

\section{From Coulomb energy to finite Wasserstein distance}
\label{sec:E1W2}
\subsection{The basic \texorpdfstring{$W_2$-$H^{-1}$ inequality}{W2-H-1 inequality}}
The basic tool is the following classical lemma (\cite{Ledoux_2019,peyre2018comparison}):
\begin{lemma}
\label{lem:W2Hm1}
Let $\mm$ be a probability measure on $\La_1$ with bounded density with respect to the Lebesgue measure and let $v$ be a vector field on $\La_1$ satisfying:
\begin{equation}
\label{vmuleb}
-\div(v) = \mm - \Leb_{\La_1}, \quad v \cdot \vec{n} = 0 \text{ on } \partial \La_1.
\end{equation}
Then there exists a coupling $\Pi$ between $\mm$ and $\Leb_{\La_1}$ (as probability measures on $\La_1$) such that:
\begin{equation}
\label{eq:W2EnergElec}
\iint_{\La_1 \times \La_1} |x-y|^2 \dd \Pi(x,y) \leq C \int_{\La_1} |v|^2.
\end{equation}
\end{lemma}
\begin{proof}[Proof of Lemma \ref{lem:W2Hm1}]
The idea is to find a simple interpolation $(\mm_t)_{t \in [0,1]}$ between $\Leb_{\La_1}$ and $\mm$ and to use $v$ to construct an associated velocity field $(v_t)_{t \in [0,1]}$ such that $(\mm_t, v_t)_{t \in [0,1]}$ solves the continuity equation \eqref{continuity}, so that the associated action gives an upper bound on the $2$-Wasserstein distance between the measures. A usual choice consists in simply taking a linear interpolation $\mm_t := (1-t) \Leb_{\La_1} + t \mm$, however if the density of $\mm$ does not have a positive lower bound then an integrability issue arises as $t \to 1$. We use instead the choice suggested by \cite[Proof of Thm. 2]{Ledoux_2019} and set:
\begin{equation}
\label{eq:detmutsmart}
\mm_t := \left(1-\theta(t)\right) \Leb_{\Lambda_1} + \theta(t) \mm, \quad \theta(t) := 1 - (1-t)^2.
\end{equation}
It is clear that $\mm_0 = \Leb_{\Lambda_1}$, $\mm_1 = \mm$, moreover we can easily compute, using \eqref{vmuleb}:
\begin{equation*}
\frac{\dd}{\dd t} \mm_t = \theta'(t) (\mm - \Leb_{\Lambda_1}) = -\div\left(\theta'(t)  v \right)
\end{equation*}
 and thus, setting $v_t := \frac{\theta'(t)  v}{\rm_t}$ ($\rm_t$ denoting the density of $\mm_t$), we do ensure that $(\mm_t, v_t)_{t \in [0,1]}$ solves the continuity equation. Observe that the density $\rm_t$ is bounded below by $1 - \theta(t)$ for all $t \in [0,1]$ and compute the action:
\begin{equation*}
\int_{0}^1 \int_{\Rd} |v_t|^2 \dd \mm_t \dd t \leq \int_{0}^1 \left( \int_{\La_1} \frac{\left(\theta'(t)\right)^2 |v(x)|^2}{\rm^2_t(x)} \rm_t(x) \dd x \right)  \dd t \leq \int_{\La_1} |v(x)|^2 \dd x \int_{0}^1 \frac{\left(\theta'(t)\right)^2}{1 - \theta(t)} \dd t,
\end{equation*}
and a direct computation shows that the last integral over $\dd t$ is finite.
\end{proof}

We deduce the following.
\begin{corollary}
\label{coro:W2Hm1bigbox}
Let $n \geq 1$, let $\bX$ be a $n^\d$-point configuration in $\La_n$ - seen here as a purely atomic measure with mass $n^\d$, and let $\eta \in (0,1)$. Assume that $\min_{x \in \bX} \dist(x, \partial \La_n) > \eta$, and let $\bX_{\eta}$ be the “spread-out” version of $\bX$ at distance $\eta$ as in \eqref{def:deta}. Let $\El$ be an electric field compatible with $\bX$ in $\La_n$ and screened in the sense of \eqref{def:screened}, i.e.
\begin{equation}
\label{EletaNeumann}
-\div(\El) = \cd \left(\bX - \Leb \right) \text{ in $\La_n$,} \quad \El \cdot \vec{n} = 0 \text{ on } \partial \La_n.
\end{equation}
Then there exists a coupling $\Pi(\cdot, \cdot ; \bX)$ between $\bX$ and $\Leb_{\La_n}$ (as measures of mass $n^\d$ on $\La_n$) such that:
\begin{equation}
\label{eq:W2EnergElecConfig}
\frac{1}{n^\d} \int_{\La_n \times \La_n} |x-y|^2 \dd \Pi(x,y ; \bX) \leq C \left( \frac{1}{n^\d} \int_{\La_n} |\Eleta|^2 + 1 \right).
\end{equation}
\end{corollary}
\begin{proof}[Proof of Corollary \ref{coro:W2Hm1bigbox}]
Since $\El$ is compatible with $\bX$, the regularization $\Eleta$ is compatible with $\bXeta$ (Claim \ref{claim:propertiesOF}) and can play the role of $v$ in \eqref{vmuleb} (the multiplicative constant $\cd$ is irrelevant). Thanks to Lemma \ref{lem:W2Hm1} and a simple scaling, this gives an upper bound on the Wasserstein distance between $\bXeta$ and $\Leb$. Then the cost of transporting the spread out charges $\bX_{\eta}$ to their centers is of order $n^\d \times \eta^2$.
\end{proof}

\subsection{Screening of electric fields in large boxes}
\label{sec:screening}
In order to apply Corollary \ref{coro:W2Hm1bigbox} to stationary point processes, we would like to be able to work with the restriction of configurations (and fields) to “large boxes”. Observe that Corollary \ref{coro:W2Hm1bigbox} requires a \emph{screened} electric field, and (which is in fact a consequence of the mere existence of such a screened field) that the number of points of $\bX$ in $\La_n$ agree perfectly with the volume of $\La_n$ - a fact that is absolutely not guaranteed for random samples.

To fix this, the key tool is an approximation procedure in arbitrarily large boxes, provided by the “screening” of Sandier-Serfaty, here in the general version of \cite[Prop. 6.1]{petrache2017next} (see also \cite[Prop. 5.6]{rougerie2016higher})
\begin{lemma}
\label{lem:screening}
There exists $R_0$, $\eta_0 > 0$, and $C$ (all depending only on $\d$) such that the following holds.

Let $\epsilon \in (0, \hal)$ and let $R$ be such that $R \geq \frac{R_0}{\epsilon^2}$. Assume that $|\La_R|$ is an integer. Let $\bX$ be a point configuration in $\La_R$ and let $\El$ be an electric field compatible with $\bX$ in $\La_R$. Under the following \emph{screenability condition}:
\begin{equation}
\label{ScreenCond}
\int_{\La_R} |\El_{\eta_0}|^2 \leq R^\d \times \left( \epsilon^{2\d + 3} \frac{R}{R_0}  \right), 
\end{equation}
there exists a point configuration $\tbX$ in $\La_R$ and a local electric field $\tEl$ such that:
\begin{enumerate}
   \item $\tbX$ has exactly $|\La_R|$ points in $\La_R$.
   \item $\tbX$ and $\bX$ coincide on $\La_{R (1 - \epsilon)}$.
   \item All the points of $\tbX$ are at distance at least $\eta_0$ from $\partial \La_R$.
   \item $\tEl$ is compatible with $\tbX$ in $\La_n$ and is \emph{screened}.
   \item We have:
   \begin{equation}
   \label{EnergieTbX}
   \int_{\La_R} |\tEl_{\eta_0}|^2 \leq C \int_{\La_R} |\El_{\eta_0}|^2.
   \end{equation}
\end{enumerate}
\end{lemma}
\begin{remark}
\label{remark_constant_screening1}
The screening condition as written in \eqref{ScreenCond} is not optimal because we have condensed two conditions (see \cite[(6.2)]{petrache2017next}) into a single one. However, this simplifies the exposition and is enough for our purposes. Also, the constant $C$ in \eqref{EnergieTbX} is in fact of the form $1 + C' \epsilon$, but we are not trying here to derive optimal multiplicative constants for our inequalities.
\end{remark}

Let us also record the following fact, which is not part of the original screening statement but which follows directly from items 3., 4., 5. of Lemma \ref{lem:screening} in view of Corollary \ref{coro:W2Hm1bigbox}.
\begin{corollary}
With the assumptions and the notation of Lemma \ref{lem:screening}, we can ensure that there exists a coupling $\Pi(\cdot, \cdot ; \tbX)$ between $\tbX$ and $\Leb_{\La_R}$ such that:
\begin{equation}
\label{eq:WassScreening}
\frac{1}{R^\d} \int_{\La_R \times \La_R} |x-y|^2 \dd \Pi(x,y ; \tbX) \leq C \left(\frac{1}{R^\d} \int_{\La_R} |\El_{\eta_0}|^2 + 1\right).
\end{equation}
\end{corollary}

\subsection{An approximate sequence of point processes}
We may reduce the proof of Theorem \ref{theo:EUWD} to the following proposition.
\begin{proposition}
\label{prop:approximate}
Let $\bbX$ be a stationary point process such that $\Coul_1(\bbX)$ is finite. Then there exists a sequence $(\bbX_n)_{n \geq 1}$ of stationary point processes such that:
\begin{enumerate}
   \item $\bbX_n$ converges in distribution to $\bbX$ as $n \to \infty$ \corT{(with respect to the vague topology on point configurations, seen as Radon measures)}.
   \item $\Ww^2_2(\bbX_n, \Leb) \leq C\left(\Coul_1(\bbX) + 1\right) (1+o_n(1))$, for some constant $C$ depending only on the dimension.
\end{enumerate}
\end{proposition}
Since $\Ww^2_2$ is lower semi-continuous on the space of stationary point processes (\cite[Prop. 2.10.]{erbar2023optimal}), this implies~\eqref{EUWD} and proves Theorem \ref{theo:EUWD}.
\newcommand{\shift}{\tau}
\newcommand{\bshift}{\tau^\sbullet}
\newcommand{\bXav}{\bX^{\mathrm{av}}}
\newcommand{\bXglob}{\bX^{\mathrm{glob}}}
\newcommand{\bbXav}{\bX^{\sbullet, \mathrm{av}}}
\newcommand{\bbXglob}{\bX^{\sbullet, \mathrm{glob}}}
\newcommand{\bbXk}{\bX^{\sbullet, (k)}}
\begin{proof}[Proof of Proposition \ref{prop:approximate}]
Let $\bbX$ be a stationary point process.

\paragraph{Step 1. Applying the screening.}
\begin{lemma}
\label{lem:applyscreening}
There exists some constant $c > 0$ (depending on the dimension) such that the following holds. For all $n \geq 1$, there exists an event $\Omega_n$ and a random point configuration $\tbX_n$, such that:
\begin{enumerate}
   \item $\lim_{n \to \infty} \Proba(\bbX \in \Omega_n) = 1$
   \item Almost every configuration $\btbX_n$ has exactly $n^\d$ points in $\La_n$.
   \item $\btbX_n \cap \La_{n - n^{1-c}} =\bbX \cap \La_{n - n^{1-c}}$ 
   on $\Omega_n$.
   \item For almost every configuration $\btbX_n$, there exists a coupling $\Pi(\cdot, \cdot ; \btbX_n)$ between $\btbX_n$ and $\Leb_{\La_n}$ such that:
   \begin{equation}
   \label{W2tPpn}
   \frac{1}{n^\d} \E \left[ \iint_{\La_n \times \La_n} |x-y|^2 \dd \Pi(x,y; \btbX_n) \right] \leq C \left( \Coul_1(\bbX) + 1\right).
   \end{equation}
\end{enumerate}
\end{lemma}
\begin{proof}[Proof of Lemma \ref{lem:applyscreening}]
Since we assume that $\Coul_1(\bbX)$ is finite, there exists a stationary random field $\bEl$ compatible with $\bbX$ such that:
\begin{equation*}
\Coul_1(\bbX) = \E[|\bEl_{1}(0)|^2] < + \infty.
\end{equation*}

Choose $c := \frac{1}{2\d + 4} > 0$ and for $n \geq 1$, let $\epsilon_n := n^{-c}$. Let $\Omega_n$ be the event corresponding to the screening condition \eqref{ScreenCond} with $R = n$ and $\epsilon = \epsilon_n$ (note that $n \geq \frac{R_0}{\epsilon_n^2}$ for $n$ large enough so the first condition of the screening lemma is satisfied). Then Markov's inequality gives:
\begin{equation*}
\Proba\left(\bbX \notin \Omega_n\right) \leq \frac{ \E\left[ \int_{\La_n} |\bEl_{\eta_0}|^2 \right]   }{ n^\d \times \epsilon_n^{2\d + 3} \frac{n}{R_0} } =  \frac{ \E \left[ \int_{\La_n} |\bEl_{\eta_0}|^2 \right]   }{ n^\d } \times o_n(1),
\end{equation*}
because $\epsilon_n$ has been chosen such that $\epsilon_n^{2\d + 3} n \to + \infty$. Moreover, since $\Coul_1(\bbX)$ is finite, we know that: 
\begin{equation*}
\limsup_{n \to \infty} \frac{1}{n^\d} \E \left[ \int_{\La_n} |\bEl_{\eta_0}|^2 \right] \corT{= \E \left[ |\bEl_{\eta_0}(0)|^2 \right]} < + \infty,
\end{equation*}
(indeed using \corT{stationarity and Fubini's theorem, we have $\E \left[ \int_{\La_n} |\bEl_{\eta_0}|^2 \right] = \E \left[ |\bEl_{\eta_0}(0)|^2 \right]$ for any $n$)} and we thus obtain $\lim_{n \to + \infty} \Proba(\bbX \notin \Omega_n) = 0$, which proves the first item.

If $\bX \in \Omega_n$, we can apply the screening construction of Lemma \ref{lem:screening} in $\La_n$, and we denote by $\tbX$ the resulting point configuration. 
{ We define a new random point field by putting $\btbX_n=\tbX$ on $\Omega_n$ and $\btbX_n=\sum_{x\in \La_n \cap \Z^2} \delta_x$ on $\Omega_n^c$.}

By design, the screened configurations $\tbX$ have exactly $n^\d$ points in $\La_n$ and coincide with the original configuration in $\La_{n - \epsilon_n n}$. This proves the second and third items. 

Finally, using \eqref{eq:WassScreening} \corT{on $\Omega_n$ and a rough estimate on the distance between our choice of $\btbX_n$ and $\Leb_{\La_n}$ on $\Omega_n^c$,} and taking the expectation we get the existence of a random coupling $\bPi(\cdot, \cdot ; \btbX_n)$ coupling $\btbX_n$ and $\Leb_{\La_n}$, such that:
\begin{equation*}
\frac{1}{n^\d} \E \left[ \iint_{\La_n \times \La_n} |x-y|^2 \dd \bPi(x,y; \btbX_n) \right] \leq C \left(\frac{1}{n^\d}  \E \left[ \int_{\La_n} |\bEl_{\eta_0}|^2 \right] +1 \right)
\end{equation*}
which proves \eqref{W2tPpn}.
\end{proof}

\paragraph{Step 2. Forming a stationary approximation.}
For $n \geq 1$, let $\btbX_n$ be given by Lemma \ref{lem:applyscreening}. We use it to form a stationary point process using a standard procedure (see e.g.\ \cite[Sec.\ 5]{leble2016logarithmic} for an extensive use in the same context) as follows:
\begin{enumerate}
   \item We take a family of i.i.d.\ random variables $\{\bbXk_n\}_{k \in n \Zd}$ distributed as $\btbX_n$.
   \item We pave $\Rd$ by copies of $\La_n$ centered on the lattice points $k \in n \Zd$. We also let $\bshift_n$ be some random variable, independent from the previous ones and uniformly distributed in $\La_n$.
   \item We define a global random point configuration $\bbXglob_n$ by:
   \begin{equation}
   \label{def:Xglob}
    \bbXglob_n := \bigcup_{k \in n \Zd} \bigcup_{x \in \bbXk_n} \{x + k\},
   \end{equation}
   which amounts to pasting the random point configuration $\bbXk_n$ onto the translated square $k + \La_n$.
   \item We define {$\bbX_n=\bbXav_n$} as $\bbXglob_n + \bshift_n$, which amounts to shifting $\bbXglob$ by an independent random vector uniformly distributed in $\La_n$.
\end{enumerate}

\paragraph{Step 3. Conclusion.}
It remains to check that {$\bbXav_n$} satisfies the properties stated in Proposition \ref{prop:approximate}. By construction, it is stationary. Moreover, by the first and third item of Lemma \ref{lem:applyscreening}, we can guarantee that it converges in distribution to $\bbX$ as $n \to \infty$.

For $k \in n \Zd$, let $\Pi(\cdot, \cdot; \bbXk_n)$ be the (random) coupling between $\bbXk_n$ (seen as a (random) measure on $\La_n$) and $\Leb_{\La_n}$ given by Lemma \ref{lem:applyscreening}. We can easily combine those couplings to form a (random) coupling $\Pi(\cdot, \cdot; \bbXglob_n)$ between $\bbXglob_n$ (seen as a (random) measure on $\Rd$) and $\Leb$. This is due to the fact that two couplings between two measures supported on two squares that have a side in common will not interfere with each other - this is somehow equivalent to “screened electric fields” which can be glued together along a common side.

Finally, if $\tau \in \La_n$ we can use $\Pi(\cdot, \cdot; \bbXglob_n)$ to define a coupling $\Pi_\tau(\cdot, \cdot; \bbXglob_n)$ between $\bbXglob_n + \tau$ and $\Leb$, through a push forward $\Pi(\cdot, \cdot; \bbXglob_n)$ by the map $(x,y) \mapsto (x+\tau, y+\tau)$.

\newcommand{\Rr}{\mathcal{R}}

\begin{claim}
\label{claim:CoutReparti}
For all $t \in \La_n$, the following convergence holds almost surely and in $L^1$:
\begin{equation}
\label{coutreparti}
\lim_{R \to + \infty} \frac{1}{|\La_R|}  \iint_{\La_R \times \Rd} |x-y|^2 \dd \Pi_t(x, y; \bbXglob_n) =  \frac{1}{n^\d} \E\left[\iint_{\La_n \times \La_n} |x-y|^2 \dd \Pi(x,y ; \btbX_n) \right],
\end{equation}
where $\Pi(\cdot, \cdot ; \btbX_n)$ is as given in Lemma \ref{lem:screening}.
\end{claim}
\begin{proof}[Proof of Claim \ref{claim:CoutReparti}]
 Let $\tau \in \La_n$ be fixed and let $R \geq 1$. There exists two families $\Rr, \Rr'$ of indices in $n\Zd$,  with $\Rr \subset \Rr'$, both having cardinality $\frac{R^\d}{n^\d} + \O(R^{\d-1})$ (with an implicit constant that may depend on $n$), such that:
\begin{equation}
\label{eq:inclusionLaR}
 \bigcup_{k \in \Rr} \left(\La_n + k + t\right) \subset \La_R \subset \bigcup_{k \in \Rr'} \left(\La_n + k + t\right),
\end{equation}
and these inclusions readily imply that:
\begin{multline*}
\sum_{k \in \Rr} \iint_{(\La_n + k + t) \times \Rd} |x-y|^2 \dd \Pi_t(x, y; \bbXglob_n)   \leq \iint_{\La_R \times \Rd} |x-y|^2 \dd \Pi_t(x, y; \bbXglob_n) \\
\leq \sum_{k \in \Rr'} \iint_{(\La_n + k + t) \times \Rd} |x-y|^2 \dd \Pi_t(x, y; \bbXglob_n). 
\end{multline*}
By construction of the coupling $\Pi_t(\cdot, \cdot; \bbXglob_n)$ as a gluing of couplings in each translated copy of $\La_n$, this amounts to:
\begin{equation*}
\sum_{k \in \Rr} \iint_{\La_n  \times \La_n} |x-y|^2 \dd \Pi(x, y; \bbXk_n)   \leq \iint_{\La_R \times \Rd} |x-y|^2 \dd \Pi_t(x, y; \bbXglob_n) \\
\leq \sum_{k \in \Rr'} \iint_{\La_n \times \La_n } |x-y|^2 \dd \Pi(x, y; \bbXk_n),
\end{equation*}
and since the $\bbXk_n$ are i.i.d. of law $\btbX_n$, the law of large numbers implies that both:
\begin{equation*}
\frac{1}{|\Rr|}  \sum_{k \in \Rr} \iint_{\La_n \times \La_n} |x-y|^2 \dd \Pi(x, y; \bbXk_n), \quad
 \frac{1}{|\Rr'|}  \sum_{k \in \Rr'} \iint_{\La_n \times \La_n} |x-y|^2 \dd \Pi(x, y; \bbXk_n) 
\end{equation*}
converge a.s. and in $L^1$ to $\E\left[\iint_{\La_n \times \La_n} |x-y|^2 \dd \Pi(x,y ; \btbX_n) \right]$ as $R \to \infty$. Since both families have cardinality equal to $\frac{R^\d}{n^\d} + \O(R^{\d-1})$, this proves the claim.
\end{proof} 
{Using $\Pi_\tau(x, y; \bbXav_n)$ as a competitor in the definition of $\Ww_2$, \eqref{W2tPpn} to estimate the right-hand side of \eqref{coutreparti}, and taking the expectation, we find:
\begin{multline*}
\E \left[ \inf_{\tilde\Pi \in \Cpl(\bbXav_n, \Leb)} \limsup_{R \to + \infty} \frac{1}{|\La_R|} \iint_{\La_R \times \Rd} |x-y|^2 \dd \tilde\Pi(x, y) \right] \\
\leq \E \left[ \limsup_{R \to + \infty} \frac{1}{|\La_R|} \iint_{\La_R \times \Rd} |x-y|^2 \dd \Pi_\tau(x, y; \bbXav_n) \right] \leq C \left(\Coul_1(\bbX) +  1\right)
\end{multline*}
which, in view of the definitions \eqref{wpperunit} and \eqref{WassLeb} concludes the proof of Proposition \ref{prop:approximate} and thus of Theorem \ref{theo:EUWD}.}
\end{proof}

\section{\texorpdfstring{$\WD$ plus density bound implies $\EU$}{Wass plus density bound implies Coul}}
\label{sec:proofWDdensity}
In the previous section, we have proven $\EU \implies \WD$. The proof of the converse implication, under the (quite strong) bounded density assumption \eqref{eq:densitybound}, follows a symmetrical route and uses two ingredients:
\begin{enumerate}
   \item A reverse inequality bounding the $H^{-1}$ norm of the difference of two measures in terms of their $2$-Wasserstein distance (cf.\ Lemma \ref{lem:ReverseH1W2}). This can be found in \cite[Prop. 3.1]{loeper2006uniqueness} and we simply reformulate it here in a form suited to our purposes.
   \item A new “screening lemma” (Lemma \ref{lem:ScreeningWass}) in $2$-Wasserstein space, allowing the construction of approximate point processes with good properties. This combines the techniques of Sandier-Serfaty with the Benamou-Brenier point of view on optimal transport.
\end{enumerate}
For simplicity when proving Lemma \ref{lem:ScreeningWass}, and because finiteness of $2$-Wasserstein distance to Lebesgue is a less interesting property in dimension $\d \geq 3$ (even the Poisson point process satisfies it), we restrict ourselves to $\d = 2$, but there is no obstruction to proceeding similarly in any dimension.

\subsection{A \texorpdfstring{$H^{-1}$-$W_2$}{H-1 - W2} inequality for measures with bounded densities}
\label{sec:H1W2}
\begin{lemma}
\label{lem:ReverseH1W2}
Let $\mm$ be a probability measure on $\La_1$ and let $\brho \geq 1$. Assume that $\mm$ has a density $\rm$ with respect to the Lebesgue measure, with $\|\rm\|_{\infty} \leq \brho < + \infty$. Then, there exists a vector field $v$ satisfying:
\begin{equation*}
- \div(v) = \left(\mm - \Leb\right) \text{ in } \La_1, \quad v \cdot \vec{n} = 0 \text{ on } \partial \La_1,
\end{equation*}
and such that:
\begin{equation}
\label{RevH1W2}
\int_{\La_1} |v|^2 \leq \brho \times \WWW_2^2\left(\mm, \Leb_{\La_1}\right).
\end{equation}
\end{lemma}
This is basically \cite[Prop. 3.1]{loeper2006uniqueness} and we sketch here the proof.
\begin{proof}
Let $(\mm_t, v_t)_{t \in [0,1]}$ be an optimal admissible pair in the Benamou-Brenier formula \eqref{eq:BenamouBrenier} interpolating between $\mm_0 := \Leb$ and $\mm_1 := \mm$, and define the vector field $v$ as $v := \int_{0}^1 \rm_t v_t \dd t$. Integrating the continuity equation \eqref{continuity} shows that $- \div(v) = \mm_1 - \mm_0 = \mm - \Leb$ in $\La_1$, and since for all $t$ the vector field $v_t$ satisfies a Neumann boundary condition, so does $v$. Moreover, we have by Cauchy-Schwarz:
\begin{equation*}
\int_{\La_1} |v|^2 \leq \int_{0}^1 \left(\int_{\La_1} \rm^2_t(x) |v_t|^2(x) \dd x \right) \dd t \leq \sup_{t \in [0,1]} \| \rm_t \|_{\infty} \WWW_2^2(\Leb_{\La_1}, \mm).
\end{equation*}
Using the uniform control \eqref{eq:uniform_dens_bound} on the density of $\rm_t$ concludes the proof.
\end{proof}

\begin{corollary}
\label{coro:pointprocessWassH1}
Let $n \geq 1$, let $\bX$ be a $n^\d$-point configuration in $\La_n$ and let $\eta \in (0,1)$. Assume that $\eta < \min_{x \in \bX} \dist(x, \partial \La_n)$ and that the following “density bound” holds:
\begin{equation}
\label{densitybX}
\sup_{x \in \La_n} \left|\bX \cap \B_1(x) \right| \leq \brho.
\end{equation}
Then there exists a screened electric field $\El$ compatible with $\bX$ in $\La_n$ such that: 
\begin{equation}
\label{eq:EnergElecW2}
\frac{1}{n^\d} \int_{\La_n} |\Eleta|^2 \leq C_\eta \brho \left(\frac{1}{n^\d} \WWW_2^2(\bX, \Leb_{\La_n}) + 1\right),
\end{equation}
with $C_\eta = C \|\chi\|_\infty \eta^{-\d}$ for some $C$ depending only on $\d$.
\end{corollary}
\begin{proof}[Proof of Corollary \ref{coro:pointprocessWassH1}]
Let $\bX_{\eta}$ be the “spread-out” version of $\bX$ at distance $\eta$ as in \eqref{def:deta}. The $2$-Wasserstein distance between the measures $\frac{1}{n^\d} \bX$ and $\frac{1}{n^\d} \bXeta$ is clearly bounded by $\eta$, thus by the triangular inequality we have:
\begin{equation}
\label{WW2bXetabX}
\WWW_2^2(\bXeta, \Leb_{\La_n}) \leq C\left( \WWW_2^2(\bX, \Leb_{\La_n}) + n^\d \right).
\end{equation}
Regularizing $\bX$ into $\bXeta$ produces a measure with bounded density with respect to the Lebesgue measure. Indeed, combining \eqref{densitybX} and the fact that spread-out Dirac masses $\deta_x$ have a bounded density, we obtain:
\begin{equation*}
\|\frac{\dd \bXeta}{\dd \Leb_{\La_n}}\|_{\infty} \leq C_\eta \brho.
\end{equation*}
We may now apply Lemma \ref{lem:ReverseH1W2} to $\frac{1}{n^\d} \widetilde{\bXeta}$, where $\widetilde{\bXeta}$ is the push-forward of $\bXeta$ by $x \mapsto n^{-1} x$. By construction, $\frac{1}{n^\d} \widetilde{\bXeta}$ is a probability measure on $\La_1$ with density bounded by $C_\eta \brho$ and such that:
\begin{equation*}
\WWW_2^2\left(\frac{1}{n^\d} \widetilde{\bXeta}, \Leb_{\La_1}\right) = \frac{1}{n^{\d+2}} \WWW_2^2\left({\bXeta}, \Leb_{\La_n}\right).
\end{equation*}
The vector field $v$ given by Lemma \ref{lem:ReverseH1W2} can then be turned back into a vector field $\Eleta$ compatible with $\bXeta$ by setting $\Eleta(x) := n v(n^{-1}x)$ and applying a simple scaling to \eqref{RevH1W2} yields:
\begin{equation*}
\int_{\La_n} |\Eleta|^2 \leq C_\eta \brho \WWW_2^2(\bXeta, \Leb_{\La_n}),
\end{equation*}
which combined with \eqref{WW2bXetabX} concludes the proof.
\end{proof}
Corollary \ref{coro:pointprocessWassH1} shows that for (finite) point configurations with a well-controlled density of points, a bound on the Wasserstein distance (to Lebesgue) can be turned into a bound on the energy of a compatible electric field. This opens the way to proving an implication of the type $\WD \implies \EU$, however our statement requires that the number of points and the volume of the box coincide, which is not true in general for a random sample of a large point configuration. We thus develop a “screening procedure” which should correct the number of points without changing the underlying configuration - or its distance to Lebesgue - too much. Note that the construction of \cite[Theorem 3.8]{erbar2023optimal} is somewhat similar in spirit, but it is not suited to our purposes: it ensures that two slightly modified point processes will have the same random number of points in large boxes, however, these random numbers are in general different from the Lebesgue mass of that box.

\subsection{“Screening” in Wasserstein space}
\label{sec:ScreenWass}
\newcommand{\tPi}{\widetilde{\Pi}}
\begin{lemma}[Wasserstein screening]
\label{lem:ScreeningWass}
We can choose $\eta_0$ small enough, and $K, C$ large enough such that the following holds. 
\newcommand\puteqnum{
  \refstepcounter{equation}\textup{(\theequation)}}

Let $\epsilon \in (0, \hal)$ and let $R \geq 10$. Assume that $|\La_R|$ is an integer. Let $\bX$ be a point configuration in $\La_R$, and let $\Pi$ be a coupling between $\bX$ and a measure $\mm_0$ on $\R^2$ with\footnote{\corT{Here we mean that $\mm_0(\Omega) \leq \Leb(\Omega)$ for all $\Omega$, which is stronger than mere absolute continuity.}} $\mm_0 \leq \Leb$, such that the mass of $\mm_0$ is equal to $|\bX|$. Assume that the points of $\bX$ satisfy a density bound of the form:
\begin{equation}
\label{eq:AssumpDensity}
\sup_{x \in \La_R} \left|\bX \cap \B_1(x)\right| \leq \brho.
\end{equation}
Assume that the following “screenability conditions” hold:
\begin{enumerate}
   \item The parameters $R, \epsilon$ and the bound \eqref{eq:AssumpDensity} are such that $\brho \leq \frac{\epsilon^2 R}{K}$    \hfill \puteqnum \label{eq:brhoKepsilon}
   \item The number of points of $\bX$ in $\La_R$ is smaller than $\frac{1}{K} \frac{\epsilon^7 R^4}{\brho}$  \hfill \puteqnum \label{WassCondNumberPoints}
   \item The quadratic cost of $\Pi$ is controlled by: 
   \begin{equation}
\label{screenWasscond}
\iint_{\La_R \times \R^2} |x-y|^2 \dd \Pi(x,y) \leq \frac{1}{K} \frac{\epsilon^7 R^4}{\brho} 
\end{equation}
   \item The restriction of $\mm_0$ to $\La_{R(1-\epsilon)}$ coincides with the Lebesgue measure on $\La_{R(1-\epsilon)}$.       \hfill \puteqnum \label{secondSC}
\end{enumerate}
Then there exists a point configuration $\tbX$ in $\La_R$ such that:
\begin{enumerate}
   \item $\tbX$ has exactly $|\La_R|$ points in $\La_R$.
   \item $\tbX$ and $\bX$ coincide on $\La_{R (1 - \epsilon)}$.
   \item All the points of $\tbX$ are at distance at least $\eta_0$ from $\partial \La_R$.
   \item The following density bound holds:
   \begin{equation}
   \label{eq:densitytbX}
\sup_{x \in \La} \left|\tbX \cap \B_1(x)\right| \leq C'' \brho.
   \end{equation}
   \item We control the (usual) $2$-Wasserstein distance between $\tbX$ and $\Leb_{\La_R}$ by:
   \begin{equation}
   \label{WassTbX}
   \WWW_2^2\left(\tbX, \Leb_{\La_R}\right) \leq C \left(1 + \epsilon \brho^{\corT{^2}} \right) \times \iint_{\La_R \times \R^2} |x-y|^2 \dd \Pi(x,y) + C'|\bX \cap \La_R| +  C' R^2 \left( \epsilon + \epsilon^2 \brho^2\right).
   \end{equation}
\end{enumerate}
\end{lemma}
\begin{remark}
It is tempting (e.g. in view of what holds for the “classical” screening, see Remark \ref{remark_constant_screening1}) to ask whether the first constant $C$ in the right-hand side of \eqref{WassTbX} could be taken equal to $1$, which would be of potential interest. It seems very plausible, as the main losses on $C$ are due to two opposite operations: at the beginning of the argument we “spread out the charges”, passing from $\bX$ to $\bX_\eta$, and at the end we “shrink” them back, reverting to a purely atomic point configuration. The Wasserstein costs of doing so are evaluated through two separate triangular inequalities, when in fact they should cancel each other (to first order, at least). However, due to some intermediate steps in our argument, obtaining $C =1$ is a bit tricky, and we do not pursue that goal here. 
\end{remark}

\begin{proof}[Proof of Lemma \ref{lem:ScreeningWass}]
We follow the strategy introduced by Sandier-Serfaty for their screening of electric fields in the context of Coulomb gases, and transpose it into the $2$-Wasserstein world using the following analogy:
\newcommand{\bv}{\mathrm{v}}
\newcommand{\lra}{\longleftrightarrow}
\begin{center}
\begin{tabular}{ c c c }
Coulomb &  & Wasserstein  \\
   Electric field $\El$ & $\lra$ & Total flux $\int_{0}^1 \mu_t \times \bv_t \dd t$ \\
   Electric energy $\int |\El|^2$ & $\lra$ &  Total cost $\int_{0}^1 \int \mu_t \times |\bv_t|^2 \dd t$\\
   Compatibility $- \div(\El) = \cd \left( \bX - \Leb\right)$ & $\lra$ & Continuity equation $-\div(\mu_t \bv_t) = \frac{\dd}{\dd t} \mu_t$ \\
 \end{tabular}
 \end{center}
 One may quickly notice two non-trivial differences: in the Coulomb world, the energy is the $L^2$-norm (squared) of the electric field, whereas on the Wasserstein side we rather have:
 \begin{equation*}
\text{(Cost) }  \int_{0}^1 \int \mu_t |\bv_t|^2 \dd t\text{ vs. }  \int \left(\int_{0}^1 \mu_t\bv_t \dd t \right)^2  \leq \int_{0}^1 \int |\mu_t|^2 |\bv_t|^2 \dd t  \text{ ($L^2$-norm squared of the total flux)},
 \end{equation*}
the integrands differing by one power of the density (hence the need to have density bounds to control the cost in terms of the total flux). Moreover, screening in the Coulomb world only involves finding \emph{one} compatible electric field, whereas the Wasserstein picture, with the Benamou-Brenier point of view, requires to construct a \emph{family of interpolating measures} and velocity fields. Most of the technical ingredients can nevertheless be borrowed from Coulomb screening as in \cite[Prop. 6.1]{petrache2017next}.

\paragraph{\corT{Summary of the screening construction.}}
In short: “screening” consists in making ad hoc constructions near the boundary of $\La_R$ to pass from the existing data to the desired one. Here we start with the data of $\mm_0, \bX$, \corT{to which we associate an optimal pair} \corT{(family of measures, family of velocity fields)} interpolating between $\mm_0$ and $\bX$ (or more precisely a regularized version of $\bX$) in the sense of Benamou-Brenier, \corT{see Section \ref{sec:DefWass}}. The coupling $\Pi$ given in the assumption \corT{of Lemma \ref{lem:ScreeningWass}} yields an upper bound on the Wasserstein distance, and thus on the action of that pair. 

We use this data to produce a \emph{new} pair \corT{of measures and fields}, interpolating this time between $\Leb_{\La_R}$, \corT{instead of $\mm_0$}, and a regularized version of a new point configuration $\tbX$, \corT{instead of $\bX$}. 

The new pair coincides with the old one within a large sub-square $\La_r \subset \La_R$, but on $\La_R \setminus \La_r$ we have constructed \emph{by hand} \corT{new families of configurations and velocity fields}. 

This construction is done by solving some elliptic equations and defining somewhat abstract fields. The “action” of the pair \corT{(which eventually controls the Wasserstein distance)} is evaluated thanks to “energy estimates” found in the Coulomb screening lemmas.

\corT{\begin{itemize}
   \item In Step 0, we fix a Benamou-Brenier pair of interpolating measures between $\mm_0$ and $\bX$ (in fact, a “spread-out” version of $\bX$, as it is more convenient to work with measures that are non singular). This simply uses some definitions.
   \item In Step 1, we choose a “good boundary”, which means that we find a value or $r$ close to $R$ such that the “boundary” $\partial \La_r$ is “good” - meaning mostly that the velocity fields have good properties along $\partial \La_r$. This boundary will separate between an “old” region $\La_r$ where we will not touch the measures or the fields, and a “new” region $\La_R \setminus \La_r$ where we will perform ad hoc constructions.
   \item In Step 2, we define the “goal measure”, namely the point configuration that will replace $\bX$. This only requires to work on $\La_R \setminus \La_r$, as we do not change anything withing $\La_r$. This is quite tedious: we split the region $\La_R \setminus \La_r$ into mesoscopic rectangles $H_i$ (“first subdivision”) then each $H_i$ is cut further into rectangles $\mathcal{R}_\alpha$ of size $\approx 1$ (“second subdivision”), and we place one point in each $\mathcal{R}_\alpha$. The areas of those rectangles are chosen carefully so that the total number of points when defining the “goal measure” is the desired one.
   \item In Step 3, we continue this ad-hoc construction by defining a “goal field”, i.e. a properly chosen electric field compatible with our goal measure. This requires again to work within each small rectangle, and to solve elliptic equations of the form $- \div \mathrm{E} = ...$ (expressing various “compatibilities” of the fields) with relevant boundary conditions. This is identical to constructions done for screening in the Coulomb world.
   \item In Step 4, we do something new: we extend the “goal measure / field” into a whole interpolating family of measures / field. While doing that, we control the “action” of this family, which boils down to estimating the $L^2$-norms of the auxiliary fields that we have constructed by hand. Thankfully, this can be done by using estimates designed for the Coulomb case.
\end{itemize}
}

\paragraph{Step 0. Regularizing the configuration and choosing an admissible pair}
\newcommand{\bXetta}{\bX_{\eta_0}}
For $\eta_0$ small enough (we choose $\eta_0 = \frac{1}{100}$ below) we let $\bXetta$ be the regularized point configuration as in \eqref{def:deta}. Of course, there is a natural coupling between $\bX$ and $\bXetta$ with quadratic cost bounded by $|\bX \cap \La_R| \times \eta_0^2$, so by the triangular inequality we have:
\begin{equation}\label{trianginegXetta}
\WWW_2^2(\bXetta, \mm_0) \preceq \WWW_2^2(\bX, \mm_0) + |\bX \cap \La_R| \times \eta_0^2  \preceq \iint_{\La_R \times \R^2} |x-y|^2 \dd \Pi(x,y) + |\bX \cap \La_R| \times \eta_0^2,
\end{equation}
the last inequality following from the fact that $\Pi$ is \emph{a} coupling between $\bX$ and $\mm_0$, which provides an upper bound on the Wasserstein distance. The density assumption \eqref{eq:AssumpDensity} on the number of points of $\bX$ translates into a bound for the density of the measure $\bXetta$ with respect to the Lebesgue measure, namely:
\begin{equation}
\label{densitybXetta}
\|\frac{\dd \bXetta}{\dd \Leb}\|_{\infty} \leq C \brho \eta_0^{-2} \leq C' \brho.
\end{equation} 
We now fix an optimal family $(\mm_t, \bv_t)_{t \in [0,1]}$, given by the Benamou-Brenier\footnote{The Wasserstein distance is a distance, hence, symmetric. We always write the point configuration (or its spread out version) first and then the reference measure. However, we will often use the Benamou-Brenier formulation in “the other direction” i.e.\ with an interpolation between the reference measure (at time $0$) and the point configuration (at time $1$).}  formula \eqref{eq:BenamouBrenier}, such that:
\begin{enumerate}
   \item The initial measure $\mm_0$ is as in the assumption (in particular $\mm_0$ coincides with the Lebesgue measure on $\La_{R(1-\epsilon)}$), and $\mm_1 = \bXetta$.
   \item The associated cost satisfies:
   \begin{equation}
   \label{costmmt}
   \int_{0}^1  \left(\int_{\R^2} |\bv_t|^2 \dd \mm_t\right) \dd t \leq \WWW_2^2(\bXetta, \mm_0).
   \end{equation}
   \item For all $t \in [0,1]$, $\mm_t$ has a bounded density $\rm_t$ with respect to the Lebesgue measure on $\R^2$, with:
   \begin{equation}
   \label{unifmt}
\sup_{t \in [0,1]} \|\rm_t\|_{\infty} \leq C \brho,
   \end{equation}
   which is possible in view of \eqref{densitybXetta} and \eqref{eq:uniform_dens_bound}. 
\end{enumerate}

\paragraph{Step 1. Finding a good boundary.}
\newcommand{\minteta}{\mm^{\mathrm{int}}_{\eta_0}}
\newcommand{\mint}{\mm^{\mathrm{int}}}
By a mean value argument, we find $r \in [R - \epsilon R, R- \hal \epsilon R]$ such that:
\begin{equation}
\label{eq:meanValue}
\int_{0}^1 \left( \int_{\partial \La_r}   |\bv_t|^2 \rm_t \right) \dd t \preceq  \frac{1}{\epsilon R} \int_{0}^1  \left(\int_{\R^2} |\bv_t|^2 \rm_t\right) \dd t \preceq \frac{1}{\epsilon R} \WWW_2^2(\bXetta, \mm_0),
\end{equation}
using \eqref{costmmt}.

We now define $\mint$ as the restriction to $\La_r$ of the measure $\bXetta$, which will not be modified. An important technical remark is that since we have spread out the Dirac masses, some might intersect $\partial \La_r$ and we will need to take care of them in the construction below.

For $t \in [0,1]$ we define the “integrated flux” at time $t$ as the following vector field:
\begin{equation}
\label{def:Phit}
\Phi_t := \int_{0}^t  \rm_s \bv_s \dd s,
\end{equation}
and we will be particularly interested in the normal component of $\Phi_t$ along $\partial \La_r$. Let us note already that:
\begin{itemize}
   \item (Compatibility.) We have $-\div \Phi_1 = \mm_1 - \mm_0$ in $\La_r$.
 \item (Flux along $\partial_{\La_r}$.) We have, using an integration by parts and the continuity equation:
   \begin{multline}
\label{flux_initial}
\int_{\partial \La_r} \Phi_1 \cdot \vec{n} = \int_{\partial \La_r} \left(\Phi_1 - \Phi_0 \right) \cdot \vec{n} = \int_{0}^1 \int_{\partial \La_r} \frac{\dd}{\dd t} \Phi_t \cdot \vec{n} = \int_{0}^1 \int_{\partial \La_r}  \rm_t \bv_t \cdot \vec{n}\\ = \int_{0}^1 \int_{\La_r}  -\div\left(\rm_t \bv_t\right) = \int_{0}^1 \int_{\La_r} \frac{\dd}{\dd t} \rm_t = \mm_1(\La_r) - \mm_0(\La_r) = |\bXetta \cap \La_r| - |\La_r|.
\end{multline}
   \item (Energy along $\partial_{\La_r}$.) We have, using \eqref{unifmt}:
   \begin{multline}
   \label{EnergyPhi1}
\int_{\partial \La_r} |\Phi_1|^2  = \int_{\partial \La_r} \left| \int_{0}^1  \rm_s \bv_s \dd s \right|^2   \leq \int_{\partial \La_r} \int_{0}^1 \rm_s^2 |\bv_s|^2 \dd s \leq C \brho\corT{^2} \int_{0}^1 \left( \int_{\partial \La_r}   |\bv_s|^2 \rm_s \right) \dd s \\
\leq \frac{C \brho\corT{^2}}{\epsilon R} \WWW_2^2(\bXetta, \mm_0).
   \end{multline}
\end{itemize}

\paragraph{Step 2. Defining the “goal” measure}
\newcommand{\detaz}{\delta^{(\eta_0)}}
We now choose $\ell := \epsilon^2 R$, so that $1 \ll \ell \ll \epsilon R$.

\subparagraph{A first subdivision.}
We begin by splitting $\La_R \setminus \La_r$ into a finite family $\{H_i\}_{i \in I}$ of rectangles with sidelengths in $[\frac{\ell}{2}, 2 \ell]$ such that $m_i |H_i| \in \mathbb{N}$ for all $i$, with the “corrected density at time 1” $m_i$ defined as:
\begin{equation}
\label{eq:defmi}
m_i := 1 + \frac{1}{|H_i|} \left( \int_{\partial \La_r \cap \partial H_i} \Phi_1 \cdot \vec{n} - n_i \right), \quad \text{ where }n_i := \sum_{x \in \bX \cap \La_r} \int_{H_i} \detaz_x.
\end{equation}
The fact that we can split $\La_R \setminus \La_r$ into a finite family $\{H_i\}_{i \in I}$ of rectangles with sidelength in $[\frac{\ell}{2}, 2 \ell]$ is fairly elementary. The fact that we can moreover guarantee $m_i |H_i| \in \mathbb{N}$ comes from a perturbative argument: we justify below that $m_i$ is of order $1$, hence by “modifying the boundaries of the $H_i$ a little bit”, the quantity $m_i |H_i|$ quickly reaches an integer. This is a standard argument in Coulomb screening (see e.g.\ \cite[Proof of Prop. 6.1., Step 2.]{petrache2017next})

\subparagraph{Control on $|m_i - 1|$.}
\begin{claim}
\label{claim:sizemi}
Up to choosing $K$ large enough in the screenability condition \eqref{screenWasscond}, we have $|m_i -1| \leq \frac{1}{100}$, in fact we can ensure that $|m_i -1| \times \detaz_0 \leq \frac{1}{100}$, with $\detaz_0$ as in \eqref{def:deta}.
\end{claim}
\begin{proof}[Proof of Claim: \ref{claim:sizemi}]
First observe that $|\detaz_0|\leq \frac{1}{\eta_0^2}\|\chi\|_{\infty}.$
 On the one hand, using Cauchy-Schwarz's inequality we find:
 \begin{equation*}
\frac{1}{|H_i|} \left|\int_{\partial \La_r \cap \partial H_i} \Phi_1 \cdot \vec{n}\right|   \leq \frac{1}{\ell^2} \left(\int_{\partial \La_r \cap \partial H_i} 1 \right)^\hal \left(\int_{\partial \La_r} |\Phi_1|^2 \right)^\hal \leq \frac{1}{\ell^2} \ell^{\hal} \left(\int_{\partial \La_r} \int_{0}^1 |\bv_t|^2 \rm^2_t \right)^\hal,
\end{equation*}
using the rough bound $\int_{\partial H_i \cap \partial \La_r } |\Phi_1|^2  \leq \int_{\partial \La_r} |\Phi_1|^2$ and the definition of $\Phi_1$. Next, using the uniform bound \eqref{unifmt} on the density $\rm_t$ and \eqref{eq:meanValue}, we obtain:
\begin{equation*}
\left(\int_{\partial \La_r} \int_{0}^1 |\bv_t|^2 \rm^2_t \right)^\hal \leq \sup_{t \in [0,1]} \|\rm_t\|^\hal_{L^{\infty}} \times \left(\int_{\partial \La_r} \int_{0}^1 |\bv_t|^2 \rm_t \right)^\hal \leq C \brho^\hal \times \left(\frac{1}{\epsilon R} \WWW_2^2(\bXetta, \mm_0) \right)^\hal,
\end{equation*}
and thus, using \eqref{trianginegXetta}, we find:
\begin{equation*}
\frac{1}{|H_i|} \left| \int_{\partial \La_r \cap \partial H_i} \Phi_1 \cdot \vec{n} \right| \leq \frac{C}{\ell^{\frac{3}{2}}} \brho^\hal \times \frac{1}{\epsilon^\hal R^\hal}  \times \left( \iint_{\La_R \times \R^2} |x-y|^2 \dd \Pi(x,y) + |\bX \cap \La_R| \times \eta_0^2 \right)^\hal.
\end{equation*}
Since $\ell = \epsilon^2 R$, we see that the right-hand side can be made arbitrarily small by choosing the constant $K$ in the screenability conditions \eqref{WassCondNumberPoints} and \eqref{screenWasscond} large enough.

Controlling $n_i$ (as in \eqref{eq:defmi}) is easier: by definition $n_i$ is bounded by the number of points in a $1$-neighborhood of the boundary $\partial H_i$ of a rectangle whose sidelenghts are of order $\ell$, and thus $n_i \leq C \brho \ell$. In particular we have:
\begin{equation*}
\frac{n_i}{|H_i|} \leq C \frac{\brho \ell}{\ell^2} = \frac{C \brho}{\epsilon^2 R},
\end{equation*}
which can again be made arbitrarily small by choosing the constant $K$ in the screenability condition \eqref{eq:brhoKepsilon} large enough.
\end{proof}

\begin{remark} 
\label{remark:1stlayer}
It is helpful to realize that since $\ell \ll \epsilon R$, our subdivision of $\La_R \setminus \La_r$ into rectangles $\{H_i\}_{i \in I}$ consists in many ($\approx \frac{1}{\epsilon}$) concentric “layers” of width $\ell$. Among those layers, only the first one is actually concerned by the definition of $n_i, m_i$ as in \eqref{eq:defmi}, indeed for all rectangles $H_i$ besides those in this first layer, we have $n_i = 0$ and $m_i =1$ (since there is no sidelength of $H_i$ that touches $\partial \La_r$).
\end{remark}

\subparagraph{A second subdivision.}
Next, we subdivide each rectangle $H_i$ into a family of smaller rectangles $\{\Ralpha\}_{\alpha \in I_i}$ such that: 
\begin{equation}
\label{eq:propRalpha}
\text{ for all $i \in I$, for all $\alpha \in I_i$, } |\Ralpha| = \frac{1}{m_i}, \text{ the sidelengths of $\Ralpha$ are between $\frac{1}{3}$ and $3$.}
\end{equation}
This is again a standard construction, see e.g.\ \cite[Lemma 6.3]{petrache2017next}. Since we ensure that $m_i$ is very close to $1$, we can guarantee that the sidelengths will also be close to $1$ (between $\frac{1}{3}$ and $3$ is a conservative estimate, but the actual values do not matter anyway).

For each such rectangle, we keep track of the spread-out charges in $\La_r$ that might intersect $\partial \La_r$ near $\Ralpha$ and write:
\begin{equation}
\label{nalpha}
n_\alpha := \sum_{x \in \bX \cap \La_r} \int_{\Ralpha} \detaz_x.
\end{equation}
\begin{remark}
\label{rem:remark1stlayer2}
As in Remark \ref{remark:1stlayer}, it is here helpful to keep in mind that if $H_i$ is not in the “first layer” of rectangles neighboring $\partial \La_r$, then $m_i = 1$ and $n_\alpha = 0$ for all $\alpha \in I_i$. In fact, among each rectangle $H_i$ in the “first layer”, only the $\O(\ell)$ rectangles $\Ralpha$ which belong to “the first layer of the first layer”, i.e.\ whose sides intersect $\partial \La_r$, are such that $n_\alpha$ may be $\neq 0$.
\end{remark}

\subparagraph{Defining the goal measure.}
\newcommand{\muexteta}{\mm^{\mathrm{ext}}_{\eta_0}}
\newcommand{\tmueta}{\tbX_{\eta_0}}
\newcommand{\tmm}{\widetilde{\mm}}
\renewcommand{\tmu}{\tmm}
\newcommand{\tbv}{\widetilde{\bv}}

We define the “goal measure on the exterior” $\muext$ as being obtained by:
\begin{enumerate}
      \item Placing a spread-out Dirac mass at the center $p_\alpha$ of each rectangle $\Ralpha$ (for each $\alpha \in I_i$, $i \in I$).
   \item Completing (in some of the rectangles $\Ralpha$ which touch $\partial \La_r$) the (spread-out) charges that are in $\La_r$ and which may intersect the boundary $\partial \La_r$. Those should indeed not be thrown away when re-defining the measure outside $\La_r$.
\end{enumerate}
We choose $\eta_0 := \frac{1}{100}$, so that for each $i \in I$ and each $\alpha \in I_i$, we have $\dist(p_\alpha, \partial \Ralpha) \geq \frac{1}{3} \geq 10 \eta_0$. We then let $\muext$ be as follows:
\begin{equation}
\label{def:muext}
\muext := \sum_{i \in I} \sum_{\alpha \in I_i} \delta^{(\eta_0)}_{p_\alpha} + \sum_{i \in I} \sum_{\alpha \in I_i} n_{\alpha},
\end{equation}
keeping in mind (see Remark \ref{rem:remark1stlayer2}) that by construction $n_\alpha = 0$ besides the “first layer of first layer” rectangles $\Ralpha$ outside $\La_r$.

We also define the “total goal measure” $\tmueta$ as 
\begin{equation}
\label{eq:deftotal}
\tmueta := \mint + \muext.
\end{equation} 
Note that by construction (and in particular because we made sure to complete the spread-out charges across $\partial \La_r$), the measure $\tmueta$ is the regularization at scale $\eta_0$ of some underlying point configuration $\tbX$, which justifies our notation. 

\subparagraph{First properties of the total goal measure.}
We may already check a certain number of things.
\begin{enumerate}
   \item By construction, the mass of $\muext$ is given (using \eqref{flux_initial} to compute $\int_{\partial \La_r} \Phi_1 \cdot \vec{n}$, and \eqref{eq:defmi}) by:
\begin{multline*}
\label{massmuext}
|\muext| = \sum_{i \in I} m_i |H_i| + \sum_{i \in I} \sum_{\alpha \in I_i} n_{\alpha} = \sum_{i \in I} |H_i| + \int_{\partial \La_r} \Phi_1 \cdot \vec{n} - \sum_{i \in I} n_{i} + \sum_{i \in I} \sum_{\alpha \in I_i} n_{\alpha} \\
= \sum_{i \in I} |H_i| + \int_{\partial \La_r} \Phi_1 \cdot \vec{n} - \sum_{i \in I} n_{i} + \sum_{i \in I} n_{i}  = |\La_R \setminus \La_r| - \left(|\bXint \cap \La_r| - |\La_r|\right), 
\end{multline*}
and thus the total mass of $\tmueta$ is equal to:
\begin{equation*}
|\tmueta| = |\mint| + |\muext| = |\bXint \cap \La_r| + \left(|\La_R| - |\La_r| - |\bXint \cap \La_r| + |\La_r|\right) = |\La_R|,
\end{equation*}
which means that the underlying point configuration $\tbX$ has exactly $|\La_R|$ points and satisfies the first item of Lemma \ref{lem:ScreeningWass}.
\item Since we have chosen our “good boundary” $\La_r$ with $r \geq R - \epsilon R$, and have not touched the configuration inside $\La_r$, the second item of Lemma \ref{lem:ScreeningWass} is satisfied.
\item By construction, the points $p_\alpha$ placed in each $\Ralpha$ are all at distance at least $\frac{1}{3}$ from the boundary, and the other points are much further away, so the third item is satisfied. 
\item The density of points within $\La_r$ has not been changed, and the density of points outside $\La_r$ is governed by $m_i$, which is close to $1$. Hence the density bound \eqref{eq:densitytbX} holds.
\end{enumerate}
The rest of the proof is devoted to checking that $\tbX$ satisfies the last item in our statement, namely the $2$-Wasserstein bound \eqref{WassTbX}. To do so, we will produce a curve $(\tmm_t, \tbv_t)_{t \in [0,1]}$ satisfying the continuity equation, with a continuous family of measures  $(\tmm_t)_{t \in [0,1]}$ on $\La_R$ such that: $\tmm_0 = \Leb_{\La_R}$, $\tmm_1 = \tmueta$ and such that the associated quadratic cost is controlled. This will give us a bound on the $2$-Wasserstein distance between $\Leb_{\La_R}$ and $\tmueta$ - which can easily be converted into a a bound on the distance between $\Leb_{\La_R}$ and $\tbX$ by the triangular inequality.

\paragraph{Step 3. Defining the goal field}
We now define a vector field (which will eventually correspond to the value at time $1$ of the velocity field $\tbv$), meant to be “compatible” with $\tmueta$. We follow the plan used for Coulomb screening and will import the relevant estimates to control the “energy” of the field. We start by constructing an “exterior” field $\Eext$ corresponding to the “exterior” measure $\muext$ on $\La_R \setminus \La_r$, and which is defined as a sum of various fields related to the rectangles $H_i$ and $\Ralpha$ in our subdivisions.
\begin{enumerate}
   \item Let us start with the rectangles $H_i$ which are \emph{not} in the “first layer”, i.e.\ whose sides do not intersect $\partial \La_r$. These are easy to treat: since for them the “corrected” density $m_i$ is equal to $1$ there is no need to “correct the flux” (see below) and we simply define, for each $\alpha \in I_i$, a field $\El^{(\alpha, p)}$ corresponding to {p}lacing a charge, namely we solve:
   \begin{equation}
   \label{placingcharge}
   -\div(\El^{(\alpha, p)}) = \left(\delta^{(\eta_0)}_{p_\alpha} - 1\right) \text{ in } \Ralpha, \quad \El^{(\alpha, p)} \cdot \vec{n} = 0 \text{ on } \partial \Ralpha,
   \end{equation}
   which is possible because $|\Ralpha| = \frac{1}{m_i} = 1$ by construction.
A simple estimate for “placing a charge without flux creation” (see \cite[Lemma 6.5]{petrache2017next}) gives us:
\begin{equation}
\label{energyplace}
\int_{\Ralpha} \left| \El^{(\alpha, p)} \right|^2 \leq C,
\end{equation}
with a constant $C$ depending only on our choice of $\eta_0$.
   \item Let us now focus on rectangles $H_i$ which \emph{are} in the “first layer”.
   \begin{enumerate}
      \item Our first task will be to “{c}omplete charges” i.e.\ making sure that the parts of  spread-out charges coming from $\La_r$ and intersecting $\partial \La_r$ along $\partial H_i$ are not forgotten. For each $\alpha \in I_i$, we solve:
      \begin{equation}
      \label{eq:complete}
-\div(\El^{(\alpha, c)}) = n_\alpha  \text{ in } \Ralpha, \quad \El^{(\alpha, c)} \cdot \vec{n} = \frac{-n_\alpha}{\int_{\partial \Ralpha \cap \partial \La_r}1} \text{ on } \partial \Ralpha
      \end{equation}
  with $n_\alpha$ as in \eqref{nalpha} (note that we choose the boundary condition appropriately to make this solvable). The energy estimate for “completing a charge near the boundary” is found in \cite[Lemma 6.6]{petrache2017next} and gives:
  \begin{equation*}
\int_{\Ralpha} \left| \El^{(\alpha, c)} \right|^2 \leq C n_\alpha^2 \leq C \brho^2,
  \end{equation*} 
since $n_\alpha$ is controlled by the density bound for the points inside $\La_r$.
   \item The second task is to solve \eqref{placingcharge} again but with the correct density ($m_i$ instead of $1$), namely:
      \begin{equation}
   \label{placingcharge2}
   -\div(\El^{(\alpha, p)}) = \left(\delta^{(\eta_0)}_{p_\alpha} - m_i\right) \text{ in } \Ralpha, \quad \El^{(\alpha, p)} \cdot \vec{n} = 0 \text{ on } \partial \Ralpha,
   \end{equation}
   which is possible because $|\Ralpha| = \frac{1}{m_i}$. The energy estimate \eqref{energyplace} still holds because $m_i$ is close to~$1$ (see again \cite[Lemma 6.5]{petrache2017next}).
\item Finally, we need to “correct the {f}lux” on $H_i$, which means here solving:
\begin{equation}
\label{eq:Elif}
-\div (\El^{(i, f)}) = \left(m_i - 1 \right) \text{ in } H_i, \quad \El^{(i, f)} \cdot \vec{n} =  \begin{cases} \Phi_1 \cdot \vec{n} + \sum_{\alpha \in I_i} \frac{n_\alpha}{\int_{\partial \Ralpha \cap \partial \La_r}1} \1_{\partial \Ralpha} & \text{ on } \partial H_i \cap \partial \La_r, \\
 \El^{(i, f)} \cdot \vec{n} =  0 & \text{ on } \partial H_i \setminus \partial \La_r \end{cases}.
\end{equation}
Note that this is compatible with our definition \eqref{eq:defmi} of $m_i$. The energy of $ \El^{(i, f)}$ can be estimated using \cite[Lemma 6.4]{petrache2017next}, we obtain:
\begin{equation}
\label{eq:energyElif}
\int_{H_i} \left| \El^{(i, f)} \right|^2 \leq C \ell \left(\int_{\partial H_i \cap \partial \La_r} |\Phi_1|^2 + \sum_{\alpha \in I_i} n^2_\alpha \right)  \leq C \ell \int_{\partial H_i \cap \partial \La_r} |\Phi_1|^2 + C \ell^2 \brho^2, 
\end{equation}
using again our density assumption to bound each $n_\alpha$ by $C\brho$, and the fact that there are $\O(\ell)$ indices $\alpha$ in $I_i$ such $n_\alpha \neq 0$.
   \end{enumerate}
   \item We then define the “exterior” field $\Eext$ as:
\begin{equation}
\label{eq:defEext}
\Eext := \sum_{i \in I} \El^{(i, f)} + \sum_{i \in I} \sum_{\alpha \in I_i} \El^{(\alpha, p)} + \sum_{i \in I} \sum_{\alpha \in I_i} \El^{(\alpha, c)},
\end{equation}
knowing that $ \El^{(i, f)} = 0$ if $H_i$ is not in the “first layer”, and that $\El^{(\alpha, c)} = 0$ if $\alpha$ is not in “the first layer of the first layer”.
\item  We finally define the “total goal field” (or field at time $1$) $\El_1$ as: 
\begin{equation*}
\El_1 := \begin{cases} \Phi_1  & \text{ on } {\La_r} \\
\Eext & \text{ on }  \La_R \setminus \La_r \\
0 & \text{ outside } \La_R.
\end{cases}
\end{equation*}
The field $\El_1$ is defined piecewise but we have made it so that the normal components of each piece agree along the common boundary.
\end{enumerate}

\subparagraph{Properties of the goal field.}
\begin{enumerate}
   \item By construction, we have:
   \begin{equation}
   \label{diveext}
-\div(\Eext) = \left(\muext - \Leb \right) \text{ in $\La_R \setminus \La_r$},
   \end{equation}
    with a normal component that vanishes on the outer boundary $\partial \La_R$ and matches the one of $\Phi_1$ along the inner boundary $\partial \La_r$. This ensures that the total goal field $\El_1$ satisfies:
   \begin{equation}
   \label{compatibleEl1}
   -\div(\El_1) = \left(\tmueta - \Leb\right) \text{ in } \La_R, \quad \El_1 \cdot \vec{n} = 0 \text{ on } \partial \La_R.
   \end{equation}
   \item We can easily estimate the energy $\int_{\La_R \setminus \La_r} |\Eext|^2$ of $\Eext$ using the controls mentioned above for each field.
   \begin{enumerate}
      \item First, we have:
      \begin{equation*}
      \int_{\La_R \setminus \La_r} \left|\sum_{i \in I} \El^{(i,f)}\right|^2 = \sum_{i \in I} \int_{H_i} \left|\El^{(i,f)}\right|^2 = \sum_{i \in I, \partial H_i \cap \partial \La_r \neq \emptyset} \int_{H_i} \left|\El^{(i,f)}\right|^2.
      \end{equation*}
      The first equality is due to the fact that the “flux correction” fields $\El^{(i, f)}$ are zero outside “their” rectangle $H_i$ and thus mutually orthogonal. The second equality is due to the fact that by construction, we are “correcting” the density $m_i$ only on rectangles that have a side included in $\partial \La_r$ - on the other ones, the density is simply $1$.

      Using \eqref{eq:energyElif}, noting that there are $\O(\frac{R}{\ell})$ rectangles $H_i$ in the “first layer”, and inserting \eqref{EnergyPhi1}, we find:
      \begin{equation}
      \label{sumElif}
\int_{\La_R \setminus \La_r} \left|\sum_{i \in I} \El^{(i,f)}\right|^2  \leq C \ell \int_{\partial \La_r} |\Phi_1|^2 + C \times \frac{R}{\ell} \times \ell^2 \brho^2 \leq C \ell \frac{{\brho\corT{^2}}}{\epsilon R} \WWW_2^2(\bXetta, \mm_0) + C R \ell \brho^2.
      \end{equation}

      \item Secondly, we have, for the same orthogonality reasons:
      \begin{equation}
      \label{sumepa}
  \int_{\La_R \setminus \La_r} \left|\sum_{i \in I} \sum_{\alpha \in I_i} \El^{(\alpha,p)}\right|^2  = \sum_{i \in I}  \sum_{\alpha \in I_i} \int_{\Ralpha} \left|\El^{(\alpha,p)}\right|^2 \leq C \times |\La_R \setminus \La_r| \times 1 = \O(\epsilon R^2).
      \end{equation}

      \item Finally, we can evaluate the total energy due to the completion of charges as:
\begin{equation}
\label{sumeca}
  \int_{\La_R \setminus \La_r} \left|\sum_{i \in I} \sum_{\alpha \in I_i} \El^{(\alpha,c)}\right|^2 \leq C \times \frac{R}{\ell} \times \ell \times \brho^2,
\end{equation}
using the fact that there are $\O(\frac{R}{\ell})$ rectangles $H_i$ in the “first layer” and that in each of this rectangle there is at most $\O(\ell)$ smaller rectangles $\Ralpha$ whose $n_\alpha \neq 0$.
   \end{enumerate}
In conclusion, combining \eqref{sumElif} and \eqref{sumepa} (the last estimate \eqref{sumeca} being of smaller order) we find:
\begin{equation}
\label{eq:EnerEext}
\int_{\La_R \setminus \La_r} |\Eext|^2 \leq  C \ell \frac{{\brho\corT{^2}}}{\epsilon R} \WWW_2^2(\bXetta, \mm_0) + C R \ell \brho^2 + C \epsilon R^2.
\end{equation}
\end{enumerate}

\begin{remark}
\label{rem:couldstophere}
For the sole purpose of Theorem \ref{theo:WDdensity}, one could in fact stop here. Indeed, we have constructed a point configuration $\tbX$ which approximates the initial one, has the right number of points, and for which we can exhibit a compatible electric field ($\El_1$) which is screened and whose energy is controlled in terms of $\WWW_2^2(\bXetta, \mm_0)$. This is enough material to bound Coulomb energies of stationary point processes in terms of Wasserstein distances, see the proof of Proposition \ref{prop:approximate2} below. However, for completeness and future reference, let us proceed with the proof of our “Wasserstein screening” as stated, which now requires that we extend our “goal field” into a family of measures and fields indexed by the time parameter $t$.
\end{remark}

\paragraph{Step 4. Constructing an interpolation}
\newcommand{\di}{\mathrm{dens}_i}
For each $i \in I$ we introduce the “density correction” at time $t \in [0,1]$ as:
\begin{equation}
\label{eq:defdi}
\di(t) := \frac{1}{|H_i|} \int_{\partial H_i \cap \partial \La_r} \left(\Phi_t - t \Phi_1 \right) \cdot \vec{n}.
\end{equation}
Again, $\di$ is actually $0$ unless $H_i$ belongs to the “first layer” of our subdivision. The quantity $\di$ can be positive or negative, but by the same argument as in Claim \ref{claim:sizemi}, we find that by choosing $K$ large enough in \eqref{screenWasscond}:
\begin{equation}
\label{controlondi}
\sup_{t \in [0,1]} \sup_{i \in I} \frac{1}{m_i} \left|\di(t)\right| \leq \frac{1}{100} \min\left(1, \frac{\eta_0^2}{ \|\chi\|_{\infty}}\right),
\end{equation}
with $\chi$ the bump function chosen for truncations (this observation will be useful later).
We can already note that for all $i \in I$ we have 
\begin{equation}
\label{eq:nocorrection01}
\di(0) = \di(1)= 0.
\end{equation} 
Moreover, in view of the definition \eqref{def:Phit} of $\Phi_t$:
\begin{equation}
\label{eq:diprime}
\frac{\dd}{\dd t} \di(t) = \frac{1}{|H_i|} \int_{\partial H_i \cap \partial \La_r} \left(\rm_t \bv_t - \Phi_1\right) \cdot \vec{n}.
\end{equation}
We now consider the corresponding “field correction”, defined as the solution to:
\newcommand{\Elc}{\El^{\mathrm{corr}, (i)}}
\begin{multline}
\label{eq:Elit}
- \div(\Elc_t) = \sum_{\alpha \in I_i} \frac{1}{m_i} \frac{\dd}{\dd t} \di(t) \times \delta^{(\eta_0)}_{p_\alpha} \text{ in } H_i, \\
 \Elc_t \cdot \vec{n} = \begin{cases} (\rm_t \bv_t - \Phi_1) \cdot \vec{n} & \text{ on } \partial H_i \cap \partial \La_r,\\
0 & \text{ on } \partial H_i \setminus \partial \La_r.
\end{cases}
\end{multline}
Note that by construction there are $|H_i| m_i$ rectangles $\Ralpha$ (and thus $|H_i| m_i$ points $p_\alpha$) within $H_i$, hence we are placing a total mass $|H_i| \times m_i \times \frac{1}{m_i} \times \frac{\dd}{\dd t} \di(t)$, which, by definition \eqref{eq:diprime} of $\frac{\dd}{\dd t} \di(t)$, is compatible with the boundary data chosen for $\Elc_t$.

\begin{claim}[Energy of the field correction]
\label{claimFieldCorrection}
For some $C$:
\begin{equation}
\label{eq:energyElit}
\sum_{i \in I} \int_{H_i} |\Elc_t|^2 \leq C \ell \int_{\partial \La_r} \left(\brho \rm_t |\bv_t|^2 + |\Phi_1|^2\right).
\end{equation}
\end{claim}
\begin{proof}[Proof of Claim \ref{claimFieldCorrection}]
For each $H_i$ in the first layer, we can decompose $\Elc_t$ as a sum of (using the terminology of \cite{petrache2017next}):
\begin{enumerate}
   \item A field that “corrects the flux” i.e.\ solves
   \begin{equation*}
   - \div(\El^{(i,f)}_t) = \frac{\dd}{\dd t} \di(t) \text{ in } H_i, \quad \El^{(i,f)}_t \cdot \vec{n} = \begin{cases}
  (\rm_t \bv_t - \Phi_1) \cdot \vec{n} & \text{ on } \partial H_i \cap \partial \La_r \\
   0 & \text{ on } \partial H_i \setminus \partial \La_r.
   \end{cases}
   \end{equation*}
   \item Fields that “place a charge without flux creation” in each $\Ralpha$ ($\alpha \in I_i$) i.e.\ solve
   \begin{equation*}
   - \div(\El^{(\alpha,p)}_t) = \frac{\dd}{\dd t} \di(t) \left( \frac{1}{m_i} \delta^{(\eta_0)}_{p_\alpha} - 1 \right) \text{ in } \Ralpha, \quad \El^{(\alpha,p)}_t \cdot \vec{n} = 0 \text{ on } \partial \Ralpha,
   \end{equation*}
   which again is solvable because $|\Ralpha| = \frac{1}{m_i}$.
\end{enumerate}
Using the corresponding energy estimates from \cite{petrache2017next} (as mentioned above), we find:
\begin{equation}
\label{energycorrect1}
\int_{H_i} |\Elc_t|^2 \leq C \ell \int_{\partial H_i \cap \partial \La_r} \left(|\rm_t|^2 |\bv_t|^2 + |\Phi_1|^2\right) + C \ell^2 \times \left|  \frac{\dd}{\dd t} \di(t) \right|^2, 
\end{equation}
which we can translate into \eqref{eq:energyElit} using the uniform bound \eqref{unifmt} on the density $\rm_t$. Indeed, using \eqref{eq:diprime} and Cauchy-Schwarz's inequality, we see that:
\begin{equation*}
\ell^2 \left|  \frac{\dd}{\dd t} \di(t) \right|^2 \leq C \frac{\ell^2}{\ell^4} \times \ell \times \int_{\partial H_i \cap \partial \La_r} \left(|\rm_t|^2 |\bv_t|^2 + |\Phi_1|^2\right), 
\end{equation*}
which is negligible compared to the first term in the right-hand side of \eqref{energycorrect1}. Summing \eqref{energycorrect1} over $i \in H_i$ yields \eqref{eq:energyElit}.
\end{proof}

\subparagraph{Interpolation measure and field}
\newcommand{\tv}{\widetilde{\bv}}
\newcommand{\trm}{\tilde{\rm}}
\newcommand{\trmext}{\tilde{\rm}^{\mathrm{ext}}}

For $t \in [0,1]$, we let $\theta(t) := 1 - (1-t)^2$ (cf.\ the proof of Lemma \ref{lem:W2Hm1}) and  define:
\begin{itemize}
   \item The measure $\muext_t$ as:
\begin{equation}
\label{defmuextT}
\muext_t := t \muext + \left(1-t \right) \Leb_{\La_R \setminus \La_r} + \sum_{i \in I} \sum_{\alpha \in I_i} \frac{1}{m_i} \di(t) \delta^{(\eta_0)}_{p_\alpha},
\end{equation}
and the measure $\tmu_t$ (total measure at time $t$)
\begin{equation}
\label{def:tmut}
\tmu_t := \mm_{\theta(t)} \1_{\La_r} + \muext_{\theta(t)} \1_{\La_R \setminus \La_r},
\end{equation}
where we recall that $\mm_t$ was chosen in Step 0. 

   \item The fields $\Eext_t$ as :
   \begin{equation}
   \label{defElext}
   \Eext_t := \theta'(t) \left(\Eext + \sum_{i \in I} \Elc_{\theta(t)}\right),
   \end{equation}
   with $\Eext$ as in \eqref{eq:defEext} and $\Elc$ as in \eqref{eq:Elit}.
   \item The velocity field $\tbv_t$ as:
   \begin{equation}
   \label{deftv}
\tv(t) :=  \begin{cases} \theta'(t) \bv_{\theta(t)} & \text{ on } \La_r \\
                     \frac{\Eext_t}{\rmuext_{\theta(t)}}  & \text{ on } \La_R \setminus \La_r.
            \end{cases}
   \end{equation}
   where $\rmuext_t$ denotes the density of $\muext_t$, which was defined in \eqref{defmuextT}.
\end{itemize}
Let us immediately note that in view of our definitions (see in particular \eqref{eq:nocorrection01}) we have:
\begin{equation*}
\tmu(0) = \Leb_{\La_R}, \quad \tmu(1) = \tmueta,
\end{equation*}
so $\tmu_t$ interpolates between the Lebesgue measure on $\La_R$ and our spread-out point configuration $\tmueta$.
\begin{claim}
\label{claim:bonneinterpol}
The following holds for all $t \in (0,1)$
\begin{enumerate}
   \item $\muext_t$ is a positive measure, and thus so is $\tmu_t$.
   \item The measure $\muext_t$ has a density $\rmuext$ with respect to the Lebesgue measure such that:
   \begin{equation}
   \label{rmuextLB}
        \frac{1}{C} \left(1 - \theta(t)\right) \leq \rmuext_t \leq C' \brho
   \end{equation}
   \item The pair $t \mapsto (\tmu_t, \tv_t)$ solves the continuity equation:
   \begin{equation*}
   \frac{\dd}{\dd t} \tmm_t = - \div\left( \tmm_t \tv_t \right).
       \end{equation*}
\end{enumerate}
\end{claim}
\begin{proof}[Proof of Claim \ref{claim:bonneinterpol}]
The only possible negative contribution to $\muext_t$ comes from the density correction, which is never too negative in view of \eqref{controlondi}. More precisely
\begin{enumerate}
   \item If $t \in [0, \hal]$, we have $(1-t) \Leb_{\La_R \setminus \La_r} \geq \hal \Leb_{\La_R \setminus \La_r}$ and by \eqref{controlondi}:
   \begin{equation*}
   \sup_{i \in I} \sup_{\alpha \in I_i} \frac{1}{m_i} \left|\di(t)\right| \delta^{(\eta_0)}_{p_\alpha} \leq    \sup_{i \in I} \sup_{\alpha \in I_i} \frac{1}{m_i} \left|\di(t)\right| \times \frac{1}{\eta_0^2} \|\chi\|_{\infty} \leq \frac{1}{100}.
   \end{equation*}
   \item If $t \in [\hal, 1]$, we have 
   \begin{equation*}
\muext_t \geq \hal \muext - \sum_{i \in I} \sum_{\alpha \in I_i} \frac{1}{m_i} \left|\di(t)\right| \delta^{(\eta_0)}_{p_\alpha} \geq \sum_{i \in I} \sum_{\alpha \in I_i} \delta^{(\eta_0)}_{p_\alpha} \left(1 - \frac{1}{m_i} \left|\di(t)\right|\right),
   \end{equation*}
   but from \eqref{controlondi} we know that $\sup_{i \in I} \sup_{\alpha \in I_i} \frac{1}{m_i} \left|\di(t)\right| \leq \frac{1}{100}$.
\end{enumerate}
So in both cases, we have enough positive mass (coming either from the Lebesgue measure of from the spread-out charges) to cover any possible negative contribution from the $\di(t)$
's.

By a similar argument, we bound the density $\rmuext_t$ from below (using the Lebesgue part) and from above (using the density bound on $\muext$) as follows:
\begin{equation*}
\frac{99}{100} (1-\theta(t)) \leq \rmuext_t \leq (1 - \theta(t)) + C \theta(t) \brho \leq C' \brho,
\end{equation*}
which proves the second item. 

We can directly compute, for $t \in (0,1)$, in view of \eqref{def:tmut} and the definitions above:
\begin{equation}
\label{dddttmmt}
\frac{\dd}{\dd t} \tmm_t = 
\begin{cases}
\theta'(t) \frac{\dd}{\dd t} \mm_{\theta(t)} & \text{ inside } \La_r \\
\theta'(t) \left(\muext - \Leb_{\La_R \setminus \La_r}\right) + \sum_{i \in I} \sum_{\alpha \in I_i} \theta'(t) \frac{\dd}{\dd t} \di(\theta(t)) \delta^{(\eta_0)}_{p_\alpha} & \text{ in } \La_R \setminus \La_r. 
\end{cases}
\end{equation}
By construction, inside $\La_r$ $(\mm_t, \bv_t)$ solves a continuity equation which implies that:
\begin{equation*}
-\div(\tmm_t \tv_t) = -\div\left(\mm_{\theta(t)} \theta'(t) \bv_{\theta(t)}\right) =  \theta'(t) \frac{\dd}{\dd t} \mm_{\theta(t)},
\end{equation*}
this last expression being exactly what we find in \eqref{dddttmmt}. On the other hand, in $\La_R \setminus \La_r$ we have (in view of \eqref{def:muext}, \eqref{deftv} and of the expressions \eqref{diveext}, \eqref{eq:Elit}):
\begin{equation*}
-\div(\tmm_t \tv_t) = -\div(\muext_{\theta(t)} \tv_t) = -\div(\Eext_t) = \theta'(t) \left( \muext - \Leb_{\La_R \setminus \La_r}\right) + \sum_{i \in I} \sum_{\alpha \in I_i} \theta'(t) \frac{\dd}{\dd t} \di(\theta(t)) \delta^{(\eta_0)}_{p_\alpha},
\end{equation*}
which coincides with the expression found in \eqref{dddttmmt}. Finally, it remains to observe that along $\partial \La_r$ we have:
\begin{multline*}
\trmext_t \tv_t \cdot \vec{n} = \Eext_t \cdot \vec{n} = \theta'(t) \Eext \cdot \vec{n} + \sum_{i \in I} \theta'(t) \Elc_{\theta(t)} \cdot \vec{n} 
= \theta'(t) \Phi_1 \cdot \vec{n} + \theta'(t) \left( \rm_{\theta(t)} \bv_{\theta(t)} - \Phi_1 \right) \cdot \vec{n} \\ = \theta'(t) \rm_{\theta(t)} \bv_{\theta(t)} \cdot \vec{n},
\end{multline*}
thus ensuring that no divergence is created along $\partial \La_r$.
\end{proof}

\paragraph{Step 5. Conclusion}
We may now estimate the quadratic cost associated to the family $(\tmm_t, \tv_t)$. By construction we have
\begin{equation*}
\int_{0}^1  \left(\int_{\La_R}  |\tv(t)|^2 \dd \tmm_t \right) \dd t = \int_{0}^1  \theta'(t)^2  \left(\int_{\La_r} |\bv_{\theta(t)}|^2 \dd \mm_{\theta(t)} \right) \dd t + \int_{0}^1  \left(\int_{\La_R \setminus \La_r} |\Eext_t|^2 \frac{1}{\rmuext_{\theta(t)}} \right) \dd t.
\end{equation*}

On the one hand, a simple change of variables, a bound on $\theta'$ and \eqref{costmmt} yield:
\begin{equation}
\label{mainpart}
\int_{0}^1  \theta'(t)^2  \left(\int_{\La_r} |\bv_{\theta(t)}|^2 \dd \mm_{\theta(t)} \right) \dd t \leq C \WWW_2^2(\bXetta, \mm_0).
\end{equation}
On the other hand, using the definition of $\Eext_t$, \eqref{eq:EnerEext} and \eqref{eq:energyElit}, the lower bound \eqref{rmuextLB} on $\rmuext$, as well as \eqref{eq:meanValue} and \eqref{EnergyPhi1} to control the flux along $\partial \La_r$, we obtain (note that $\frac{(\theta')^2}{1 - \theta}$ is bounded):
\begin{multline*}
\int_{0}^1  \left(\int_{\La_R \setminus \La_r} |\Eext_t|^2 \frac{1}{\rmuext_{\theta(t)}} \right) \dd t
 \leq C \int_{0}^1 \frac{\theta'(t)^2}{(1-\theta(t))} \left(\int_{\La_R \setminus \La_r} |\Eext|^2 + \sum_{i \in I} \int_{H_i} |\Elc|^2 \right) \dd t \\
 \leq
 C \ell \frac{{\brho\corT{^2}}}{\epsilon R} \WWW_2^2(\bXetta, \mm_0) + C R \ell \brho^2 + C \epsilon R^2 + C \ell  \int_{0}^1  \int_{\partial \La_r} \left( {\brho} \rm_t |\bv_t|^2 + |\Phi_1|^2\right) \dd t \\
 \leq
  C \ell \frac{{\brho\corT{^2}}}{\epsilon R} \WWW_2^2(\bXetta, \mm_0) + C R \ell \brho^2 + C \epsilon R^2,
\end{multline*} 
which combined with \eqref{mainpart} shows that:
\begin{equation*}
\WWW_2^2(\tbX, \Leb_{\La_R}) \leq C \WWW_2^2(\bXetta, \mm_0) \left(1 + \epsilon \brho\corT{^2} \right) + C R^2 \left( \epsilon + \epsilon^2 \brho^2 \right).
\end{equation*} 
Using \eqref{trianginegXetta} to compare $\WWW_2^2(\bXetta, \mm_0)$ and $\WWW_2^2(\bX, \mm_0)$ concludes the proof of \eqref{WassTbX}.
\end{proof}

\subsection{Constructing an approximate sequence}
In a similar fashion to Section \ref{sec:E1W2}, we reduce the proof of Theorem \ref{theo:WDdensity} to the following proposition.
\begin{proposition}
\label{prop:approximate2}
Let $\bbX$ be a stationary point process satisfying the uniform density bound \eqref{eq:densitybound} and such that $\Ww_2^2(\bbX, \Leb)$ is finite. Then there exists a sequence $(\bbX_n)_{n \geq 1}$ of stationary point processes such that:
\begin{enumerate}
   \item $\bbX_n$ converges in distribution to $\bbX$ as $n \to \infty$.
   \item $\Coul_{\eta_0}(\bbX_n) \leq C \brho  \left(\Ww^2_2(\bbX, \Leb) + 1\right) (1+o_n(1))$.
\end{enumerate}
\end{proposition}
Since the Coulomb energy is lower semi-continuous on the space of stationary point processes (see \cite[Lemma 3.9]{leble2017large} \corT{where $\widetilde{\mathbb{W}}$ corresponds to our $\Coul$}), this implies~\eqref{eq:CoulWwbrho}.

\begin{proof}[Proof of Proposition \ref{prop:approximate2}]
The proof of Proposition \ref{prop:approximate2} is almost identical to the proof of Proposition \ref{prop:approximate}, replacing Lemma \ref{lem:screening} (screening of electric fields) by Lemma \ref{lem:ScreeningWass} (screening in Wasserstein space) and Corollary \ref{coro:W2Hm1bigbox} ($W_2$-$H^{-1}$ inequality for screened configurations) by Corollary \ref{coro:pointprocessWassH1} (the reverse inequality).

\paragraph{Step 1. Applying the screening.}
\newcommand{\btEln}{{\tEl}^{\bullet}}
\begin{lemma}
\label{lem:applyscreening2}
Let $c := \frac{1}{10}$. For all $n \geq 1$, there exists an event $\Omega_n$ and a random point configuration $\btbX_n$ in $\La_n$ such that:
\begin{enumerate}
   \item $\lim_{n \to \infty} \Proba(\bbX \in \Omega_n) = 1$
   \item $\btbX_n$ has almost surely exactly $n^\d$ points in $\La_n$.
   \item $\btbX_n \cap \La_{n - n^{1-c}}=\bbX \cap \La_{n - n^{1-c}}$ on $\Omega_n$.
   \item There exists a random local electric field $\btEln$ compatible with $\btbX_n$ in $\La_n$ and screened almost surely such that:
   \begin{equation} \label{W2tPpndeuxieme}
   \frac{1}{n^\d} \E \left[ \int_{\La_n} |\btEln_{\eta_0}|^2 \right] \leq C \brho \left(\Ww^2_2(\bbX, \Leb) + 1 \right).
   \end{equation}
\end{enumerate}
\end{lemma}
\begin{proof}[Proof of Lemma \ref{lem:applyscreening2}]
Let $\epsilon_n = n^{-c}$ with $c = \frac{1}{10}$. 

\paragraph{The screenability conditions are often satisfied.}
We choose for $\Omega_n$ the “screenability conditions” of Lemma \ref{lem:ScreeningWass}, i.e.\ we say that a point configuration $\bX$ belongs to $\Omega_n$ when \eqref{eq:brhoKepsilon}, \eqref{WassCondNumberPoints} are satisfied (with $\epsilon = \epsilon_n$ and $R = n$) and when there exists a measure $\mm_0$ on $\R^2$ whose mass is equal to $|\bX\cap\La_n|$, which coincides with the Lebesgue measure on $\La_{n - n^{1-c}}$ and such that $\mm_0 \leq \Leb$, and a coupling $\Pi$ between $\bX\cap\La_n$ and $\mm_0$ satisfying \eqref{screenWasscond}. We want to prove that $\Proba(\bbX \notin \Omega_n)$ tends to $0$ as $n \to \infty$.

First, since $\epsilon_n \to 0$ in such a way that $\epsilon_n^2 n \to \infty$, it is clear that the condition \eqref{eq:brhoKepsilon} is satisfied for all $n$ large enough (depending on $\brho$).

To show that big excesses of points in $\La_n$ are unlikely and prove that \eqref{WassCondNumberPoints} is often satisfied, we can use the following simple observation: if for some $R \geq 1$ we have $|\bX \cap \La_R| \geq M R^2$ with $M \geq 10$, then for all couplings $\Pi$ between $\bX$ and $\Leb$:
\begin{equation*}
\iint_{\La_{R} \times \R^2} |x-y|^2 \dd \Pi(x,y) \geq \lambda M^2 R^4
\end{equation*}
(for some universal constant $\lambda$). We have chosen $\epsilon_n$ in such a way that, as $n \to \infty$:
\begin{equation*}
M_n := \frac{\epsilon_n^7 n^4}{K \brho n^2} = \frac{1}{K \brho} n^{\frac{13}{10}} \to \infty,
\end{equation*} 
and we may thus write, using the definition of $\Ww_2$ and Markov's inequality:
\begin{multline}
\label{raretoomanypoints}
\Proba\left[ |\bbX \cap \La_n| \geq \frac{\epsilon_n^7 n^4}{K \brho}  \right] = \Proba\left[ |\bbX \cap \La_n| \geq M_n n^2  \right] \leq \Proba\left[ \inf_{\Pi \in \Cpl(\bbX, \Leb)} \iint_{ \La_{n} \times \R^2} |x-y|^2 \dd \Pi(x,y) \geq \lambda M_n^2 n^4 \right] \\
\leq \frac{\Ww_2^2(\bbX, \Leb) \times n^2}{\lambda M_n^2 n^4} \to 0,
\end{multline}
so the condition \eqref{WassCondNumberPoints} is satisfied with probability $1 - o_n(1)$. We now focus  on the conditions \eqref{screenWasscond} and \eqref{secondSC}, which involve finding a good coupling.

\newcommand{\Pig}{\Pi^{\mathrm{glob}}}
\newcommand{\Pil}{\Pi^{\mathrm{loc}}}
\newcommand{\tPig}{\tilde{\Pi}^{\mathrm{glob}}}
\newcommand{\tPil}{\tilde{\Pi}^{\mathrm{loc}}}

By assumption, $\Ww_2^2(\bbX, \Leb)$ is finite, hence for almost every realisation $\bbX$ there exists a (global) coupling $\Pig(\cdot, \cdot ; \bbX)$ between $\bbX$ and $\Leb$, such that for all $r > 0$:
\begin{equation}
\label{meancost_squares}
\E\left[ \frac{1}{r^2}  \iint_{\Lambda_r\times \R^2}|x-y|^2 \tilde{\Pi}(dx,dy; \bbX) \right] = \Ww_2^2(\bbX, \Leb) < + \infty.
\end{equation}
This global (random) coupling induces a “local” (random) coupling $\Pil_n(\cdot, \cdot ; \bbX)$ between the measure $\bbX \cap \La_n$ on the one hand and a certain measure $\mm_0$, which is such that $\mm_0 \leq \Leb$ and whose mass is equal to $|\bbX \cap \La_n|$. A simple Markov's bound yields:
\begin{multline*}
 \Proba\left[\iint_{\Lambda_n \times \Rd}|x-y|^2 \Pil_n(dx,dy; \bbX) \geq \frac{1}{K \brho} \epsilon_n^7 n^4 \right] = \Proba\left[\iint_{\Lambda_n \times \Rd}|x-y|^2 \Pig(dx,dy; \bbX) \geq \frac{1}{K \brho} \epsilon_n^7 n^4 \right]
 \\
 \leq  \frac{K \brho n^2 \Ww_2^2(\bbX, \Leb)}{n^{4-7c}}.
\end{multline*}
By our choice of $c$, the right-hand side tends to $0$ as $n \to \infty$, hence the condition \eqref{screenWasscond} is satisfied with probability $1 - o_n(1)$.

 It remains to prove that \eqref{secondSC} is often satisfied, which involves finding an $L^{\infty}$ bound on the displacement of the mass inside $\La_{n(1- \epsilon_n)}$ under $\Pil_n$. We will use a result which is stated for bounded densities, so we first turn the (random) coupling $\Pig$ between $\bbX$ and the Lebesgue measure into a (random) coupling $\tPig(\cdot, \cdot ; \bbX)$ between $\bbX_{\eta_0}$ and $\Leb$ by going through the obvious coupling between $\bbX$ and $\bbX_{\eta_0}$. On each square $\La_r$, the additional cost is $\O(|\bbX \cap \La_r|)$, and thus using Markov's inequality again, we can ensure that:
\begin{equation}
\label{couttPig}
\Proba \left[ \iint_{\Lambda_{2n}\times \R^2} |x-y|^2 \tPig(\dd x,\dd y ; \bbX) \leq n^{2+c} \right] = 1 - o_n(1).
\end{equation} 
 We now appeal to the $L^{\infty}$ estimate of \cite[Lemma 3.1]{Michael_GOLDMAN_2020}, which ensures that:
\begin{equation}
\label{Linfinty}
\sup_{(x,y)\in\mbox{supp}(\tPig(\cdot, \cdot ; \bbX)), x \in \B_{\frac{3}{4} \cdot 2n} \text{ or } y\in \B_{\frac{3}{4} \cdot 2n}} |x-y| \leq C n \left( \frac{1}{n^4} \iint_{\Lambda_{2n}\times \R^2} |x-y|^2 \tPig(\dd x, \dd y ; \bbX)) \right)^{\frac{1}{4}} 
\end{equation}
{(Note that \cite[Lemma 3.1]{Michael_GOLDMAN_2020} is written in terms of transport maps, whereas we use the language of couplings which is due to the density assumptions on the marginal measure the same. Indeed, our $\tPig(\cdot, \cdot ; \bbX)$ is the pushforward of either marginal via a map $T$ (resp.\ its inverse) onto the other marginal by the classical Brenier result.)}
Note that, at least for $n$ large enough, $n^{1-c} \geq C n \times \left( \frac{1}{n^4} n^{2+c} \right)^{\frac{1}{4}}$. Combining \eqref{couttPig} and \eqref{Linfinty} we deduce that with probability $1 - o_n(1)$, there is no couple $(x,y)$ in the support of $\tPig(\cdot, \cdot ; \bbX)$ such that $x \in \B_n$ or $y\in \B_n$, and with $|x-y| \geq n^{1-c}$. In particular, it means that all the Lebesgue mass on $\La_{n-n^{1-c}}$ is transported onto (spread-out) points in $\La_n$, and thus the measure $\mm_0$ defined above coincides with the Lebesgue measure on $\La_{n-n^{1-c}}$. This proves that the last screening condition \eqref{secondSC} is satisfied with probability $1 - o_n(1)$.

\paragraph{Construction of the approximate process in large boxes.}
If $\bX \in \Omega_n$ we can apply the “screening in Wasserstein space” and obtain a new configuration which we denote by $\tbX$. 

\corT{We define a new point configuration $\btbX_n$ by putting $\btbX_n=\tbX$ on $\Omega_n$ and $\btbX_n=\sum_{x\in \Lambda_n\cap \mathbb Z^2} \delta_x$ on $\Omega_n^c$. The choice on $\Omega_n^c$ is arbitrary as long as it satisfy item 2 and has finite energy as in \eqref{W2tPpndeuxieme} }
By design, the configurations $\tbX$ have exactly $n^\d$ points in $\La_n$ and coincide with the original configuration in $\La_{n - \epsilon_n n}$. This proves the second and third items. Moreover, \corT{the configurations $\btbX_n$} satisfy the density bound \eqref{eq:densitytbX} and the Wasserstein bound \eqref{WassTbX}, which together imply, in view of Corollary \ref{coro:pointprocessWassH1}, the existence of a screened electric field $\tEl$ compatible with $\tbX$ in $\La_n$ and such that:
\begin{equation*}
\frac{1}{n^\d} \int_{\La_n} |\tEl_{\eta_0}|^2 \leq C \left( \brho  \frac{1}{n^\d} \iint |x-y|^2 \Pi'(dx,dy; \bbX_{\eta_0}) + \frac{\brho}{n^\d} |\bbX\cap\Lambda_n| + 1 \right) 
\end{equation*}
In particular, we obtain the existence of a random local electric field $\btEl$ compatible with $\btbX_n$ in $\La_n$ and screened almost surely, such that:
\begin{equation*}
\frac{1}{n^\d} \E \left[ \int_{\La_n} |\btEl_{\eta_0}|^2 \right] \leq \frac{C \brho}{n^\d} \left( \E\left[ \iint_{\La_n \times \La_n} |x-y|^2 \Pi'(dx,dy; \bbX_{\eta_0})\right] +1 \right) 
\leq C \brho \left( \Ww_2^2(\bbX_{\eta_0},\Leb) + 1 \right).
\end{equation*}
Together with the triangular inequality $\Ww_2(\bbX_{\eta_0},\Leb)\leq \eta_0 + \Ww_2(\bbX,\Leb)$ it implies that:
\begin{equation*}
\frac{1}{n^\d} \E \left[ \int_{\La_n} |\btEl_{\eta_0}|^2 \right]  
\leq  C''  \brho  \left(  \Ww^2_2(\bbX, \Leb)  + 1 \right),
\end{equation*}
which proves \eqref{W2tPpndeuxieme}. 
\end{proof}

\newcommand{\bElglob}{\El^{\bullet, \mathrm{glob}}}
\paragraph{Step 2. Forming a stationary approximation and conclusion}
For $n \geq 1$, let $\btbX_n$ be given by Lemma \ref{lem:applyscreening2}. We use it to form a stationary point process $\bbXglob_n$ exactly as in Step 2 of the proof of Proposition~\ref{prop:approximate}. Since the screened electric fields $\btEl$ can be glued together, we may construct a stationary electric field $\bElglob$ compatible with $\bbXglob_n$ and such that:
\begin{equation*}
\Coul_{\eta_0}(\bbXglob_n) \leq \frac{1}{n^\d} \E \left[ \int_{\La_n} |\bElglob_{\eta_0}|^2 \right] \leq C \brho  \left( \Ww^2_2(\bbX, \Leb)  + 1 \right),
\end{equation*}
which concludes the proof of Proposition \ref{prop:approximate2} \corT{using the last item of Lemma \ref{lem:twoproperties}}.
\end{proof}

\section{\texorpdfstring{$\SHU \iff \EU$}{HU-* <=> Coul}: proof of Theorem \ref{theo:SHUEU}}
\label{sec:HUSSWYCoul}
\label{sec:Sodin}
\newcommand{\JU}{\mathrm{J}_1}
\newcommand{\hPsi}{\widehat{\Psi}}
\newcommand{\Ee}{\bar{\mathrm{E}}}
In this section, we prove the equivalence (in dimension $2$) between the quantitative hyperuniformity $\SHU$ introduced in Definition \ref{def:SHU} and finiteness of the regularized Coulomb energy. We rely on the results of the recent paper \cite{sodin2023random}, which is concerned with the following question (with a different terminology): given a stationary point process $\bbX$, when does there exist a compatible stationary electric field $\bEl$? Without any requirement on the field's regularity, the answer is: always, but this ceases to be true if one asks e.g.\ for $\E[|\bEl(0)|]$ to be finite. The main result of \cite{sodin2023random} expresses the equivalence between three conditions:
\begin{enumerate}
    \item The finitess of a certain integral against the spectral measure of $\bbX$, which we recall below in \eqref{def:SC}.
    \item Existence of a compatible stationary electric field with some regularity.
    \item Convergence in $L^2$ of the following random sum: 
    \begin{equation*}
   \Ee_R := \sum_{x \in \bX, 1 \leq |x| \leq R} \nabla \log(x).
    \end{equation*}
 \end{enumerate} 
 We will here use the equivalence between 1. and 2. but let us briefly comment on the third condition. If one tries to prove that $\bbX$ has finite (regularized) Coulomb energy, a first task would be to construct \emph{some} electric field compatible with $\bbX$. The most natural thing to do would be to define $\El(z)$ ($z \in \R^2$) as the (hypothetical) limit of:
 \begin{equation*}
    \Ee'_R(z) := \int_{x \in \R^2, |x| \leq R} - \nabla \log|z-x| \dd \left(\bbX - \Leb \right)(x).
 \end{equation*}
Note that $\Ee'_R(0)$ and $\Ee_R$ coincide because the contribution of Lebesgue in $\Ee'_R(0)$ vanishes by rotational symmetry. It is unclear that $\Ee'_R(z)$ or $\Ee_R$ has a limit as $R \to \infty$, in any case one could estimate the size of $\Ee'_R(z)$ by splitting the domain of integration into smaller regions, applying a Taylor's expansion on each one, and using “discrepancy estimates” i.e.\ the size of $|\bX \cap \Omega| - |\Omega|$ in each region $\Omega$ to control first order terms. Thus it is plausible that “good hyperuniformity” and “existence of good electric fields” are related properties.

\subsection{Preliminaries on the spectral measure}
A point process $\bbX$ is said have “finite second moment” when $\E[|\bbX \cap \BR|^2]$ is finite for all $r > 0$. The spectral measure of a stationary point process $\bbX$ (with finite second moment) is a locally finite measure $\rho$ (on $\Rd$, here $\d = 2$) such that for all compactly supported smooth functions $\varphi, \psi$
\begin{equation*}
\Cov\left( \sum_{x \in \bX} \varphi(x), \sum_{x \in \bX} \psi(x) \right) = \int_{\omega \in \R^2} \hat{\varphi}(\omega) \hat{\psi}(\omega) \dd \rho(\omega),
\end{equation*}
where $\hat{\varphi}(\omega) := \int_{\R^2} e^{-2i \pi \omega \cdot x}\varphi(x) \dd x$ is the Fourier transform of $\varphi$.

One can think of $\rho$ as the analogue of the “spectral measures” for second-order stationary processes indexed by $\mathbb{Z}$, which are probability measures on $\mathbb{S}^1$ (this is in fact not just an analogy). We refer to \cite[Sec.\ 2.1]{sodin2023random} and the references therein for more details. It is recalled there that $\rho$ is always “translation-bounded”, a fact for which we will need a quantitative version, expressed in the next lemma.
\begin{lemma}
\label{lem:TranslationBounded}
We have: 
\begin{equation}
\label{eq:TranslationBounded}
\sup_{a \in \R^2} \rho(\B_1(a)) \leq C \E\left[|\bbX \cap \B_1|^2\right].
\end{equation}
\end{lemma}
\begin{proof}
Let $\Psi := \chi \ast \chi$, where $\chi$ is a smooth bump function near $0$ as above in Section \ref{sec:CoulombEnergy}, in particular it is radially symmetric, supported in $B_1$ and has mass 1. We have $\hPsi(0) = 1$ and $\hPsi$ is smooth, so there exists $\epsilon > 0$ such that $\hPsi \geq \hal$ on $\B_\epsilon$. In particular, we have:
\begin{equation*}
\rho\left(\B_{\epsilon}(a)\right) = \int_{\omega \in \R^2} \1_{\B_\epsilon(a)}(\omega) \dd \rho(\omega) \leq C \int_{\omega \in \R^2} |\hat{\chi}|^2(\omega-a) \dd \rho(\omega) = C \Var\left[ \sum_{x \in \bX} \chi(x) \right] 
\leq C \E\left[|\bbX \cap \B_1|^2\right].
\end{equation*}
We have thus bounded $\rho(\B_\epsilon(a))$ independently of $a$, and since one can cover the unit ball by a finite amount of balls of radius $\epsilon$ we easily deduce \eqref{eq:TranslationBounded}.
\end{proof}
\corT{We will need below a control on the integral of $\omega \mapsto \frac{1}{|\omega|^3}$ against $\rho$. Paving the plane by balls and using \eqref{eq:TranslationBounded} on each ball, we can bound the integral by a sum over balls (which makes a convergent series appear) and get:}
\begin{corollary}
\begin{equation}
\label{eq:conseqTB}
\int_{\omega \in \R^2, |\omega| \geq 1} \frac{1}{|\omega|^3} \dd \rho(\omega) \leq  C \E\left[|\bbX \cap \B_1|^2\right].
\end{equation}
\end{corollary}

\subsection{The spectral condition and \texorpdfstring{$\SHU$}{*-hyperuniformity} are equivalent}

\begin{definition}[The spectral condition $\SC$]
Let $\bbX$ be a stationary point process with finite second moment.
We say that $\bbX$ satisfies the “spectral condition” $\SC$ of \cite{sodin2023random} when:
\begin{equation}
\label{def:SC}
\int_{0 \leq |\omega| \leq 1} \frac{\dd \rho(\omega)}{|\omega|^2} < + \infty.
\end{equation}
\end{definition}

\begin{lemma}
\label{lem:equivSCSHU}
For a stationary point process with finite second moment, we have $\SC \iff \SHU$. More precisely there exists a constant $C$ such that:
\begin{equation}
\label{eq:SCSHU}
\sum_{n \geq 0} \sigma(2^n) \leq C \int_{0 \leq |\omega| \leq 1} \frac{\dd \rho(\omega)}{|\omega|^2}, \quad \int_{0 \leq |\omega| \leq 1} \frac{\dd \rho(\omega)}{|\omega|^2} \leq C \sum_{n \geq 0} \sigma(2^n).
\end{equation}
\end{lemma}
\begin{proof}[Proof of Lemma \ref{lem:equivSCSHU}]
Recall that the rescaled number variance can be expressed in Fourier space as:
\begin{equation}
\label{sigmainFourier}
\sigma(r) = \int_{\omega \in \R^2} \jr(\omega) \dd \rho(\omega), \quad \text{where } \jr(\omega) = \frac{1}{|\omega|^2} \JU^2(r |\omega|), 
\end{equation}
with $\JU$ a Bessel function of the first kind (see \cite[Sec. 1]{coste2021order}). The following facts about $\jr$  are relevant for us:
\begin{claim}
\label{claim:jr}
There exists some constants $C, c$ such that:
\begin{enumerate}
   \item For $|\omega| \geq r^{-1}$, we have $\jr(\omega) \leq C \frac{1}{r|\omega|^3}$.
   \item For $|\omega| \leq r^{-1}$, we have $c r^2 \leq \jr(\omega) \leq C r^2$.
\end{enumerate}
\end{claim}
\begin{proof}
The first item follows from a classical decay estimate for Bessel functions, namely that $|\JU|(t) \leq \frac{C}{\sqrt{t}}$ for $|t| \geq 1$. The second estimate follows from the combination of:
\begin{enumerate}
   \item The asymptotics $\JU(x) \sim \corT{x}$ as $x \to 0^+$.
   \item The fact that $\JU(x)$ does not vanish for $x \leq 1$.
\end{enumerate}
Those two facts ensure the existence of $c, C > 0$ such that $\corT{c |x|^2 \leq \JU^2(|x|) \leq C|x|^2}$ for $|x| \leq 1$. Inserting this into the definition of $\jr$ proves the claim.
\end{proof}

\paragraph{$\SHU$ implies $\SC$.}
Since $\jr$ is non-negative and $\rho$ is a positive measure, for all $r > 0$ we can write:
\begin{multline*} 
\sigma(r) = \int_{\omega \in \R^2} \jr(\omega) \dd \rho(\omega) \geq \int_{\omega \in \R^2, |\omega| \in (\hal r^{-1}, r^{-1}]} \jr(\omega) \dd \rho(\omega) \geq c r^2 \int_{\omega \in \R^2, |\omega| \in (\hal r^{-1}, r^{-1}]} \dd \rho(\omega) \\
\geq \frac{c}{4} \int_{\omega \in \R^2, |\omega| \in (\hal r^{-1}, r^{-1}]} \frac{1}{|\omega|^2} \dd \rho(\omega),
\end{multline*}
using Claim \ref{claim:jr} (second item) in the previous-to-last inequality. Taking $r = 2^{n}$ and summing over $n \geq 0$, we obtain:
\begin{equation*}
\sum_{n = 0}^{+ \infty} \sigma(2^n) \geq \frac{c}{4} \sum_{n =0}^{+ \infty} \int_{\omega \in \R^2, |\omega| \in (2^{-n-1}, 2^{-n}]} \frac{1}{|\omega|^2} \dd \rho(\omega) =  \frac{c}{4} \int_{0 < |\omega| \leq 1}  \frac{1}{|\omega|^2} \dd \rho(\omega)
\end{equation*}
This gives the right-hand inequality in \eqref{eq:SCSHU} (hyperuniformity, or the mere boundedness of the rescaled number variance (or even less), ensures that $\rho$ has no atom at $0$, see \cite[Sec. 2.3]{sodin2023random}).

\paragraph{$\SC$ implies $\SHU$.} Conversely, let us start again from \eqref{sigmainFourier} and decompose $\sigma(r)$, using Claim \ref{claim:jr} as:
\begin{multline}
\label{SWYdeco}
\sigma(r) = \int_{\omega \in \R^2, |\omega| \leq r^{-1}} \jr(\omega) \dd \rho(\omega) + \int_{\omega \in \R^2, r^{-1} \leq |\omega| \leq 1} \jr(\omega) \dd \rho(\omega) + \int_{\omega \in \R^2, 1 \leq |\omega|} \jr(\omega) \dd \rho(\omega) \\
\leq C r^2 \int_{\omega \in \R^2, |\omega| \leq r^{-1}}  \dd \rho(\omega) + C \int_{\omega \in \R^2, r^{-1} \leq |\omega| \leq 1} \frac{1}{r |\omega|^3} \dd \rho(\omega) + C \frac{1}{r} \int_{\omega \in \R^2, 1 \leq |\omega|} \frac{1}{|\omega|^3} \dd \rho(\omega).
\end{multline}
Taking $r = 2^n$ and summing over $n \geq 0$, we obtain:
\begin{itemize}
   \item On the one hand:
   \begin{multline}
   \label{SWY1}
\sum_{n = 0}^{+ \infty} 2^{2n} \int_{\omega \in \R^2, |\omega| \leq 2^{-n}}  \dd \rho(\omega) \leq  \int_{\omega \in \R^2, |\omega| \leq 1} \left( \sum_{n \leq - \log_2 |\omega|} 2^{2n} \right) \dd \rho(\omega) \leq \int_{\omega \in \R^2, |\omega| \leq 1} \frac{1}{|\omega|^2} \dd \rho(\omega).
   \end{multline}
   \item On the other hand:
   \begin{multline}
   \label{SWY2}
\sum_{n = 0}^{+ \infty} \int_{\omega \in \R^2, 2^{-n} \leq |\omega| \leq 1} \frac{1}{2^n |\omega|^3} \dd \rho(\omega) \leq C \int_{\omega \in \R^2, 0 < |\omega| \leq 1}  \left( \sum_{n \geq - \log_2 |\omega|} 2^{-n} \right) \frac{1}{|\omega|^3} \dd \rho(\omega) \\
\leq C \int_{\omega \in \R^2, 0 < |\omega| \leq 1} \frac{1}{|\omega|^2} \dd \rho(\omega).
   \end{multline}
   \item Finally: 
   \begin{equation}
   \label{SWY3}
\sum_{n = 0}^{+ \infty} \frac{1}{2^{n}} \int_{\omega \in \R^2, 1 \leq |\omega|} \frac{1}{|\omega|^3} \dd \rho(\omega) \leq C \int_{\omega \in \R^2, 1 \leq |\omega|} \frac{1}{|\omega|^3} \dd \rho(\omega),
   \end{equation}
    and this last integral is always finite, see \eqref{eq:conseqTB}.
\end{itemize}
Combining \eqref{SWYdeco}, \eqref{SWY1}, \eqref{SWY2} and \eqref{SWY3} yields the left-hand inequality in \eqref{eq:SCSHU}.
\end{proof}

\subsection{The spectral condition implies finite regularized Coulomb energy}
\newcommand{\rhoEl}{\rho_{\bEl}}
\begin{proposition}
\label{prop:SCCoul}
Let $\bbX$ be a stationary point process with finite second moment such that $\SC$ holds. Then we have:
\begin{equation}
\label{eq:SCCoul}
\Coul_1(\bbX) \leq C  \int_{0 \leq |\omega| \leq 1} \frac{\dd \rho(\omega)}{|\omega|^2} + \E\left[|\bbX \cap \B_1|^2\right]. 
\end{equation}
\end{proposition}
\begin{proof}[Proof of Proposition \ref{prop:SCCoul}]
From \cite[Sec. 5]{sodin2023random}, we know the following:
\begin{enumerate}
   \item If $\SC$ holds, then there exists a stationary vector field (denoted by $\mathcal{V}$ in \cite{sodin2023random}, but which we denote by $\bEl$ for consistency) which is compatible with $\bbX$. \corT{The authors of \cite{sodin2023random} do not use the term “compatible”, but their main concern (see \cite[Question 1]{sodin2023random}) is to find a “stationary vector field $V_\Lambda$ such that $\div V_\Lambda = n_\Lambda - c_\Lambda m$”, where:
   \begin{itemize}
      \item $n_\Lambda$ is the purely atomic measure associated to a point process $\Lambda$
      \item $c_\Lambda$ is the intensity of the point process (in our case, $c_\Lambda$ is always $1$).
      \item $m$ is the Lebesgue measure,
   \end{itemize}
   so their question is indeed equivalent to solving $\div V_\Lambda = n_\Lambda - \Leb$, which up to a multiplicative constant and a change of notation is the same as solving the compatibility condition \eqref{def:compatible}.
   }

   \item This vector field has a spectral measure $\rhoEl$ which is given by $\dd \rhoEl(\omega) = \frac{\dd \rho(\omega)}{|\omega|^2}$. 
\corT{Heuristically}, it means that, for all $x \in \Rd$ we have:
\begin{equation}
\label{eq:El0Elx}
\E\left[\bEl(0) \cdot \bEl(x)\right] = \int_{\omega \in \R^2} e^{-i 2\pi x \cdot \omega} \frac{\dd \rho(\omega)}{|\omega|^2}. 
\end{equation}
\end{enumerate}
Note that taking $x = 0$ in \eqref{eq:El0Elx} \corT{would} yield:
\begin{equation*}
\E\left[|\bEl(0)|^2\right] = \int_{\omega \in \R^2} \frac{\dd \rho(\omega)}{|\omega|^2},
\end{equation*} 
however this integral might be infinite despite the spectral condition $\SC$, as nothing guarantees that $\int_{\omega \in \R^2, |\omega| \geq 1} \frac{\dd \rho(\omega)}{|\omega|^2} < + \infty$. Such a high-frequency integrability is related to the short-distance properties of the point process, which we now discard by applying a regularization.

Let $\chi$ be the radially symmetric bump function used for regularization in Section \ref{sec:CoulombEnergy}, and let $\bhEl := \chi \ast \bEl$. We have:
\begin{equation*}
\E\left[|\bhEl(0)|^2\right] = \E\left[|\chi \ast \bEl(0)|^2\right] = \E\left[ \iint_{u,v \in \B_{1}} \chi(u) \chi(v) \bEl(u) \cdot \bEl(v) \dd u \dd v \right] = \int_{\omega \in \R^2} |\hat{\chi}|^2(\omega) \frac{\dd \rho(\omega)}{|\omega|^2},
\end{equation*}
\corT{this last equality being rigorously derived in \cite[Lemma 4.1]{sodin2023random2}, \cite[Prop. 5.9]{sodin2023random}.}

Now, since $\chi$ is smooth and compactly supported, its Fourier transform is bounded near $0$ \emph{and} decays \corT{faster than any polynomial} at infinity. We thus obtain:
\begin{equation*}
\E\left[|\bhEl(0)|^2\right] \leq C \int_{\omega \in \R^2, |\omega| \leq 1} \frac{\dd \rho(\omega)}{|\omega|^2} + C \int_{\omega \in \R^2, 1 \leq |\omega|} \frac{\dd \rho(\omega)}{|\omega|^3},
\end{equation*}
and using \eqref{eq:conseqTB} to bound the last integral yields:
\begin{equation}
\label{bhEl}
\E\left[|\bhEl(0)|^2\right] \leq C \int_{\omega \in \R^2, 0 \leq |\omega| \leq 1} \frac{\dd \rho(\omega)}{|\omega|^2} + C \E\left[|\bbX \cap \B_1|^2\right].
\end{equation} 
Now set $\bEl := \bhEl - \sum_{x \in \bbX} \mathsf{f}_1(\cdot - x)$ with $\mathsf{f}$ as in \eqref{def:feta}. Then by construction $\bhEl$ coincides with the truncation $\bEl_1$ as in \eqref{def:Eleta} with $\eta = 1$. Moreover $\bEl$ is stationary and compatible with $\bbX$. We have thus checked the existence of a stationary electric field $\bEl$ compatible with $\bbX$ such that $\E\left[|\bEl_1(0)|^2\right]$ is finite and bounded by the right-hand side of \eqref{bhEl}, which ensures that the regularized Coulomb energy is finite, and moreover \eqref{eq:SCCoul} holds.
\end{proof}

\subsection{Finite regularized Coulomb energy implies the spectral condition}
\label{sec:CoulSC}
We will use \cite[Thm. 5.4]{sodin2023random}, which we here reformulate in our language.
\begin{lemma}
\label{lem:SodinEnough}
Let $\bbX$ be a stationary point process with finite second moment, and let $\bEl$ be a random stationary electric field satisfying the following properties:
\begin{enumerate}
   \item It is compatible with $\bbX$.
   \item For all smooth test functions $\varphi_1, \varphi_2$ compactly supported on $\R^2$, letting $\varphi = (\varphi_1, \varphi_2)$, we have:
   \begin{equation}
   \label{eq:finiteCondition}
   \E\left[  \left| \int_{\R^2} \bEl \cdot \varphi \right|^2 \right] < + \infty.
   \end{equation}
\end{enumerate}
Then the spectral condition $\SC$ holds.
\end{lemma}
From Lemma \ref{lem:SodinEnough} we can deduce that $\EU \implies \SC$, because existence of regularized electric fields in $L^2$ is enough to guarantee \eqref{eq:finiteCondition}, as we now show.
\begin{proposition}
\label{prop:CoulSC}
Assume that $\EU$ holds. Then the spectral condition $\SC$ holds.
\end{proposition}
\newcommand{\heta}{h^{\sbullet}_{\eta}}
\begin{proof}[Proof of Proposition \ref{prop:CoulSC}]
By Definition \ref{def:RegularizedEnergy}, $\EU$ means that there exists a random stationary  electric field $\bEl$ compatible with $\bbX$ and $\eta \in (0, 1)$ such that $\E\left[|\bEleta(0)|^2\right] < + \infty$. For any test vector field $\varphi$ as in Lemma \ref{lem:SodinEnough}, we write:
\begin{equation*}
\left| \int_{\R^2} \bEl \cdot \varphi \right|^2 \leq C \|\varphi\|^2_{L^{\infty}} \left( \| \bEleta \|^2_{L^1(\BR)}  +  \| \bEl - \bEleta \|^2_{L^1(\BR)} \right),
\end{equation*}
where $r > 0$ is large enough to contain the support of $\varphi$. On the one hand, using Cauchy-Schwarz's inequality and stationarity, we have:
\begin{equation*}
\E \left[ \| \bEleta \|^2_{L^1(\BR)}  \right] \leq |\BR| \times \E\left[ \int_{\BR} |\bEleta|^2  \right] \leq |\BR| \times |\BR|  \times \E\left[|\bEleta(0)|^2\right].
\end{equation*}
On the other hand, we have by definition of the regularization $\bEleta$ (see \eqref{def:Eleta}):
\begin{equation*}
\| \bEl - \bEleta \|_{L^1(\BR)} \leq |\bbX \cap \B_{r+1}| \times |\nabla \feta|_{L^1}, 
\end{equation*}
where $|\nabla \feta|_{L^1}$ is a finite quantity depending on $\eta$ and we thus obtain:
\begin{equation*}
\E \left[ \| \bEl - \bEleta \|^2_{L^1(\BR)}  \right] \leq C_\eta \E\left[  |\bbX \cap \B_{r+1}|^2 \right], 
\end{equation*}
but finiteness of the (regularized) Coulomb energy implies that the point process has finite second moment (see Lemma \ref{lem:twoproperties})
hence that $\E\left[  |\bbX \cap \B_{r+1}|^2 \right]$ is finite for all $r$. Thus \eqref{eq:finiteCondition} holds.
\end{proof}

\section{Construction of counter-examples}
\label{sec:constructions}
\newcommand{\LakL}{\Lambda^{(k)}_N}
\newcommand{\LaL}{\Lambda_N}
\newcommand{\LaUL}{\Lambda^{(1)}_N}
\newcommand{\LaZL}{\Lambda^{(0)}_N}
\newcommand{\pL}{\alpha_N}
\newcommand{\tp}{{\p}_N}
\newcommand{\CkN}{C^{(k)}_N}
\newcommand{\CZN}{C^{(0)}_N}

For $N \geq 10$ we let $\LaL=\LaL^{(0)}$ be the (half-open) cube $\LaL:= [-\hal N, \hal N)^2$, and for all $k \in N\Z^2$ we let $\LakL := k + \LaZL$, which is a cube of sidelength $N$ centered at $k$.

\subsection{\texorpdfstring{$\WD$ does not imply $\EU$}{Wass does not imply Coul}}
\label{WDnotEU}
We will construct a stationary point process as a mixture of “elementary blocks” - themselves stationary point processes, which we now define.

\subsubsection*{Definition of the elementary blocks}
We use here a construction of \cite[\corT{Sec. 4}]{HUPL}. We first define a “displacement field” $\tp : \R^2 \mapsto \R^2$ by setting $\tp(x) := k-x$ for all $k \in N\Z^2$ and all $x$ in $\LakL$. { Thinking about $\tp$ as a device to transport/displace points, i.e.\ the point at $x$ is transported to $x+\tp(x)$,} applying $\tp$ corresponds to “sending all the points of the cube $\LakL$ to its center”. This (deterministic) displacement field is clearly $N\Zd$-{periodic}. Finally, we set:
 \begin{equation*}
\bbX_N := \{\corT{x + \tp(x) + \tau}, \ x \in \Z^2 \},
 \end{equation*} 
\corT{where $\tau$ is a random vector uniformly distributed in $\LaL$. The point process $\bbX_N$ is then stationary.}
By construction, the size of all displacements is bounded by $\sqrt{2} N$, thus $\Pp_N={\mathrm{ Law}}(\bbX_N)$ is at finite distance of the lattice - hence of the Lebesgue measure, and we have:
\begin{equation}
\label{PpNLeb}
\Ww^2_2(\bbX_N, \Leb) \leq C N^2.
\end{equation}

{Note that $\bbX_N$ is not a simple point process. One can easily change this construction to obtain a simple point process by sending the points of $\LakL$ independently to uniform points in a small ball around $k$. The following arguments  only need minor adaptations.  }


\subsubsection*{Lower bound on the Coulomb energy}
\begin{lemma}
\label{lem:toomanyhasacost}
Let $\bX$ be a point configuration and $\El$ a compatible electric field. Assume that for some $z \in \R^2$ and for $M \geq 100$, we have: $|\bX \cap \B_1(z)| \geq M^2$. Then for some $c > 0$ and for all $\eta \in (0,1)$:
\begin{equation}
\label{costtoomany}
\int_{\B_{M}(z)} |\Eleta|^2 \geq c M^4 \log M,
\end{equation}
\end{lemma}
The assumption expresses the fact that there are “way too many ($M \geq 100$) points” within $\B_1$ and the conclusion is that the electric energy in the ball of radius $M$ is bounded below by $M^2 \log M$ - significantly more than the volume of that ball. 
\begin{proof}
The proof uses a standard technique in the analysis of Coulomb gases and relies on the fact that since $-\div(\El) = 2\pi \left( \bX - \Leb \right)$, any “discrepancy” between the number of points of $\bX$ and the mass of $\Leb$ in a given region can be felt by integrating the normal coordinates of the field along the boundary of that region, and then eventually within the electric energy itself.

For $k$ between $1$ and $\log \frac{M}{10} - 1$, for all $t$ between $2^{k}$ and $2^{k+1}$ we have:
\begin{equation}
\label{eq:IPPDiscr}
\int_{\partial \B_{t}(z)} \Eleta \cdot \vec{n} = \int_{\B_{t}(z)} \div \Eleta = \left| \bXeta \cap \B_{t}(z)\right| - \pi t^2 \geq M^2 -  \pi \frac{M^2}{100} \geq \hal M^2.
\end{equation}
Let $\AK$ be the annulus $\B_{2^{k+1}}(z) \setminus \B_{2^{k}}(z)$. Integrating \eqref{eq:IPPDiscr} over $t$ and using Cauchy-Schwarz's inequality, we obtain:
\begin{equation*}
\int_{\AK} |\Eleta|^2 \geq 2^{-2k} \left( \int_{t = 2^{k}}^{2^{k+1}} \dd t \int_{\partial \B_{t}(z)} \Eleta \cdot \vec{n} \right)^2 \geq 2^{-2k} \left( 2^{k} \times c M^2 \right)^2 \geq c' M^4,
\end{equation*}
and thus a sum over $k$ yields, as claimed in \eqref{costtoomany}:
\begin{equation*}
\int_{\B_{M}(z)}  |\Eleta|^2 \geq \sum_{k = 1}^{\log \frac{M}{10} - 1} c' M^4  \geq c'' M^4 \log M.
\end{equation*}
\end{proof}

Using Lemma \ref{lem:toomanyhasacost} and the definition of $\bbX_N$ we deduce the following lower bound on its regularized Coulomb energy.
\begin{corollary}
For some constant $c > 0$ and for all $\eta \in (0,1)$ we have:
\begin{equation}
\label{eq:CoulPpN}
\Coul_\eta(\bbX_N) \geq c N^2 \log N.
\end{equation}
\end{corollary}
\begin{proof}
Under $\Pp_N$, almost surely there is a point $z$ in $\La_{2N}$ which receives all the points from one of the squares, and is thus such that $|\bX \cap \B_1(z)| \geq N^2$, and then by Lemma \ref{lem:toomanyhasacost} we know that if $\El$ is any electric field compatible with $\bX$ then:
\begin{equation*}
\int_{\B_{N}(z)} |\rEleta|^2 \geq c N^4 \log N.
\end{equation*}
Consequently, if $\El$ is a stationary random electric field compatible with $\Pp_N$ we almost surely have: 
\begin{equation*}
\int_{\La_{10N}} |\rEleta|^2 \geq cN^4 \log N
\end{equation*}
(because $\La_{10N}$ contains all the possible disks $\B_{N}(z)$ for $z$ in $\La_{2N}$) and thus:
\begin{equation*}
\E \left[ \int_{\La_{10N}} |\rEleta|^2  \right] = (10N)^2 \E \left[ |\rEleta(0)|^2  \right] \geq cN^4 \log N,
\end{equation*}
which implies \eqref{eq:CoulPpN} by Definition \ref{def:RegularizedEnergy}.
\end{proof}

\subsubsection*{Application: construction of the process}
Let $\{\alpha_N\}_{N \geq 1}$ be a sequence of positive coefficients such that:
\begin{equation*}
\sum_{N \geq 1} \alpha_N = 1, \quad \sum_{N \geq 1} \alpha_N N^2 < + \infty, \quad \sum_{N \geq 1} \alpha_N N^2 \log N = + \infty.
\end{equation*}
Define $\bbX$ as being equal to $\bbX_N$ with probability $\alpha_N$. In view of \eqref{PpNLeb} we have:
\begin{equation*}
\Ww^2_2(\bbX, \Leb) \leq C \sum_{N \geq 1} \alpha_N N^2 < + \infty
\end{equation*}
by our choice of coefficients, hence the process $\bbX$ is at finite $2$-Wasserstein distance of $\Leb$. On the other hand, by \eqref{eq:CoulPpN} and our choice of coefficients we have for all $\eta \in (0,1)$:
\begin{equation*}
\Coul_\eta(\bbX) \geq c \sum_{N \geq 1} \alpha_N N^2 \log N = + \infty.
\end{equation*}

\corT{\subsubsection*{Auxiliary fact: control of the discrepancy by the Coulomb energy}
We sketch here the proof of a useful inequality mentioned in Lemma \ref{lem:twoproperties}, namely the fact that the (regularized) Coulomb energy of a point process $\bbX$ can be used to control the second moment of the discrepancies of $\bbX$ (i.e. the variance of the number of points in large disks). The starting point is the inequality written above in \eqref{eq:IPPDiscr} namely:
\begin{equation*}
\int_{\partial \B_{t}(z)} \Eleta \cdot \vec{n} = \int_{\B_{t}(z)} \div \Eleta = \left| \bXeta \cap \B_{t}(z)\right| - |\B_t(z)|.
\end{equation*}
Taking the second moment and applying Cauchy-Schwarz's inequality yields:
\begin{equation*}
\E\left[  \left(\left| \bXeta \cap \B_{t}(z)\right| - \left|\B_t(z)\right| \right)^2 \right] = \E\left[  \left(\int_{\partial \B_{t}(z)} \Eleta \cdot \vec{n} \right)^2 \right] \leq |\partial \B_{t}(z)| \times \E\left[  \int_{\partial \B_{t}(z)} \left|\Eleta\right|^2 \right].
\end{equation*}
Now, using stationarity, we have $\E\left[  \int_{\partial \B_{t}(z)} \left|\Eleta\right|^2 \right] = |\partial \B_{t}(z)| \times \E\left[ \left|\Eleta(0)\right|^2 \right] = \pi t \times \Coul_\eta(\bbX)$ and we thus obtain, as claimed:
\begin{equation*}
\E\left[  \left(\left| \bXeta \cap \B_{t}(z)\right| - \left|\B_t(z)\right| \right)^2 \right] \leq C t^2 \Coul_\eta(\bbX).
\end{equation*}
}

\subsection{Hyperuniformity does not imply finite Wasserstein distance}
We construct here a stationary, hyperuniform point process which has infinite $1$-Wasserstein distance to Lebesgue. As before, it is obtained as a mixture of certain “building blocks” - but here we rely on a different family of elementary blocks, which were used in \cite[Sec. 6.2]{leble2016logarithmic} as finite-energy, hyperuniform approximations of the Poisson point process.

\subsubsection*{Definition of the elementary blocks}
\newcommand{\BkN}{\mathbf{B}^{(k)}_{N}}
\newcommand{\BN}{\mathbf{B}_{N}}
\newcommand{\bbXp}{\bar{\bX}^{\sbullet}}
For $N \geq 1$, we let $\{\BkN\}_{k \in N\Z^2}$ be a family of i.i.d Bernoulli/Binomial point processes with $N^2$ points in $\LaL$ (i.e.\ each $\BkN$ has the law of $N^2$ i.i.d. uniformly distributed points in $\LaL$). We let $\bbXp_N$ be the point process obtained by pasting the configuration $\BkN$ onto $\LakL$ for all $k \in N \Z^2$. Finally, we let $\bbX_N$ be the \emph{stationary} version obtained by averaging $\bbXp_N$ over translations in $\LaL$.

\subsubsection*{Upper bound on the number variance}
\begin{lemma}
\label{lem:numbervariancePoissonHU}
We have, for some $C > 0$ independent of $N$ and for all $r \geq 1$:
\begin{equation*}
\Var[ |\bbX_N \cap \BR| ] \leq C \begin{cases}
r^2 & \text{ if } r \leq N \\
rN & \text{ if } r \geq N.
\end{cases}
\end{equation*}
\end{lemma}
\begin{proof}[Proof of Lemma \ref{lem:numbervariancePoissonHU}]
Write $\bbX_N$ as $\bbXp_N + \bshift_N$, where $\bbXp_N$ is the “non-averaged” process and $\bshift_N$ is an independent shift uniformly distributed in $\LaL$. Fix $\shift \in \LaL$ and let us bound the variance of $\bbXp_N + \shift$.
\begin{itemize}
   \item If $r\leq N$, we have $\Var\left[|(\bbXp_N + \shift) \cap \BR| \right] \leq Cr^2$. Indeed:
   \begin{itemize}
      \item If $\BR$ is entirely contained within one of the shifted squares $\LakL + \shift$ for $k \in N\Z^2$, then the number of points in $\BR$ follows a binomial distribution of parameters $\left(N^2, \frac{|\BR|}{N^2}\right)$ and is thus $\O(r^2)$.
       \item If $\BR$ intersects several of the shifted squares, the number of points in $\BR$ is the sum of (at most four) independent Binomial random variables whose variances add up to $\O(r^2)$ again. 
    \end{itemize} 
   \item If $r>N$, only the shifted squares $\LakL + \shift$ ($k \in N\Z^2$) that intersect the boundary $\partial \BR$ contribute to the number variance - all the shifted squares that are included in $\partial \BR$ contain exactly $N^2$ points. There are $\O\left(\frac{r}{N}\right)$ many shifted squares intersecting $\partial \BR$, which behave independently, and the contribution to the number variance in each of these boxes is bounded by $N^2$. Hence the number variance of $\bbXp_N + \shift$ in $\BR$ is bounded by $\O\left(\frac{r}{N}\right) \times N^2= \O(rN)$.
\end{itemize}
Averaging over the choice of $\shift$ concludes the proof of Lemma \ref{lem:numbervariancePoissonHU}.
\end{proof}

\subsubsection*{Lower bound on the Wasserstein distances}
Recall  from \eqref{wpperunit} that if $\bX$ is a point configuration in $\R^2$, $\w_1(\bX, \Leb)$ denotes the $1$-Wasserstein distance \emph{per unit volume} between $\bX$ and $\Leb$, namely:
\begin{equation*}
\w_1(\bX, \Leb) := \inf_{\Pi \in \Cpl(\bX, \Leb)} \limsup_{n \to \infty} \frac{1}{n^\d} \int_{\La_n \times \R^2} |x-y| \dd \Pi(x,y).
\end{equation*}

\begin{lemma}
\label{lem:costPoissonHU}
There exists $c > 0$ independent of $N$ such that:
\begin{equation}
\label{eq:costPoisson}
\E\left[ \w_1(\bbX_N, \Leb) \right] \geq c \log^{\hal} N.
\end{equation}
\end{lemma}
\begin{proof}[Proof of Lemma \ref{lem:costPoissonHU}]
By translation-invariance of the Lebesgue measure, it is enough to prove the lower bound \eqref{eq:costPoisson} for the non-averaged process $\bbXp_N$. Using \cite[Lemma 2.6]{erbar2023optimal} with { $\vartheta(x-y)=|x-y|$}, we know that $\E\left[ \w_1(\bbXp_N, \Leb) \right]$ is bounded below by:
\begin{equation}
\label{obs1}
\E\left[ \w_1(\bbXp_N, \Leb) \right] \geq \frac{1}{N^2} \E\left[ \inf_{\Pi \in \Cpl(\bbXp_N, \Leb)} \int_{\La_N \times \R^2} |x-y| \dd \Pi(x,y) \right]
\end{equation}
(we could in fact replace $N$ by any $R > 0$ in the right-hand side, but this choice is natural and convenient). Next, observe that for all point configuration $\bX$, and for all couplings $\Pi$ between $\bX$ and $\Leb$, if $f : \R^2 \to \R$ is a $1$-Lipschitz function supported on $\La_N$, we have:
\begin{equation}
\label{obs2}
\int_{\La_N \times \R^2} |x-y| \dd \Pi(x,y) \geq \int_{\La_N \times \La_N} \left(f(x) - f(y)\right) \dd \Pi(x,y) = \int_{\La_N} f(x) \dd \bX(x) - \int_{\La_N} f(x) \dd x.
\end{equation}
Since the law of $\bbXp_N$ in $\La_N$ is given by a binomial point process $\BN$ (with $N^2$ i.i.d.\ points), it is thus sufficient for us to produce such a function $f$ with:
\begin{equation}
\label{eq:goodf}
\E \left[ \int_{\La_N} f(x) \dd \BN(x) - \int_{\La_N} f(x) \dd x \right] \geq c N^2 \log^{\hal} N,
\end{equation}
as in view of \eqref{obs1}, \eqref{obs2} this will imply a lower bound of the same type for $\E\left[ \w_1(\bbXp_N, \Leb) \right]$. 

The construction of a function $f$ satisfying \eqref{eq:goodf} was done in \cite{HuMaOt23} in the case of a Poisson point process in $\La_N$, but an inspection of the proof shows that it also applies for a Binomial point process (which is not surprising as the two are very closely related). Indeed, \cite[Lemma 2.7]{HuMaOt23} only uses some moment estimates and some invariance properties of the Poisson process under measure preserving transformation that we can easily check for our case. 

More precisely, for a cube $Q \subset \R^2$ call its "right half" $Q_r$ and its "left half" $Q_l$ and define $N(Q)$ as the difference of number of points in the right and left half, namely:
\begin{equation*}
N(Q) := |\BN \cap Q_r|- |\BN \cap Q_l|.
\end{equation*} 
Moreover, fix a reference function $\hat\zeta$ with 
\begin{align*}
{\mbox{ supp}}\hat\zeta\in(0,1)^2,\quad\int\hat\zeta\,d x=0,\quad
\int_{ x_1>\frac{1}{2}}\hat\zeta d x
-\int_{x_1<\frac{1}{2}}\hat\zeta d x=1.
\end{align*}
For every cube $Q$ we define $\zeta_Q$ via the transformation
\begin{align*}
\zeta_Q(A_Q\hat x)=\hat\zeta(\hat x)\;\mbox{where}\;A_Q\;\mbox{is the affine map such that}\;Q=A_Q(0,1)^2.
\end{align*}
For the argument of \cite[Lemma 2.7]{HuMaOt23} to apply, we need to check that for any dyadic cubes $Q, Q'\subset \Lambda_N$ (i.e.\ $Q, Q'$ are obtained by dyadic subdivision of $\LaL$) we have:
\begin{align}
\mathbb{E}N_Q^2=|Q|,\quad\mathbb{E}N_Q N_{Q'}=0\;\mbox{for}\;Q\not=Q',\quad
\mathbb{E}N_Q\int\zeta_Q d\mu =|Q|, \quad \E N_Q^4\lesssim |Q|^2, \quad \E(\int \zeta_Q \dd \BN)^2\lesssim |Q|.
\end{align}
These estimates follow from direct computations. The only other ingredient needed in the proof of \cite[Lemma 2.7]{HuMaOt23} is the invariance of the law of the Binomial process under measure preserving transformations of the Lebesgue measure which induces martingale properties w.r.t.\ dyadic subdivisions, see discussion below \cite[(2.38)]{HuMaOt23}.
\end{proof}

\subsubsection*{Application: construction of the process}
Let $\{\alpha_N\}_{N \geq 1}$ be a sequence of positive coefficients such that:
\begin{equation*}
\sum_{N \geq 1} \alpha_N = 1, \quad \sum_{N \geq 1} \alpha_N \log^{\hal}(N)  = + \infty,
\end{equation*}
and define the point process $\bbX$ as being given by the random variable $\bbX_N$ with probability $\alpha_N$. We now check that $\bbX$ satisfies our conclusions.

On the one hand, for $r \geq 1$, we can compute the rescaled number variance of $\bbX$ in $\BR$ using Lemma~\ref{lem:numbervariancePoissonHU}:
\begin{equation*}
\sigma(r) = \frac{1}{\pi r^2} \left(\sum_{N \geq 1, r \leq N} \alpha_N r^2 + \sum_{N \geq 1, r \geq N} \alpha_N N r \right) \leq \sum_{N \geq r} \alpha_N + \sum_{N \leq r} \alpha_N \frac{N}{r} = o(1) \text{ as } r \to \infty,
\end{equation*}
hence $\bbX$ is indeed hyperuniform. On the other hand, using Lemma \ref{lem:costPoissonHU} we get:
\begin{equation*}
\E\left[ \w_1(\bbX, \Leb) \right] = \sum_{N \geq 1} \alpha_N \E\left[ \w_1(\bbX_N, \Leb) \right] \geq c \sum_{N \geq 1} \alpha_N \log^{\hal} N = + \infty,
\end{equation*}
hence $\Ww_1(\bbX, \Leb) = + \infty$.

\subsubsection*{Further comments}
In Lemma \ref{lem:costPoissonHU} we have found a lower bound on the $1$-Wasserstein cost per unit volume $\E\left[ \w_1(\bbX_N, \Leb) \right]$, using the local lower bound implied by \eqref{eq:goodf} on the distance between a Binomial process and the Lebesgue measure. In short, we have proven that:
\begin{equation}
\label{wpBinomial}
\E\left[ \w_1(\bbX_N, \Leb) \right] \geq c \log^{\frac{1}{2}} N,\quad \E\left[ \w^{p}_p(\bbX_N, \Leb) \right] \geq c \log^{\frac{p}{2}} N \text{ for } p \geq 1,
\end{equation}
the second inequality following from the first one by Hölder's inequality. In fact, the celebrated result of Ajtai-Komlos-Tusnady \cite{ajtai1984optimal} imply that an upper bound of the same order of magnitude as in \eqref{wpBinomial} holds, hence that $\E\left[ \w^{p}_p(\bbX_N, \Leb) \right]$ is exactly of order $\log^{\frac{p}{2}} N$. By choosing the coefficients $\alpha_N$ appropriately depending on $p_\star \geq 1$, it is thus possible to produce stationary point processes which are hyperuniform, have a finite $p$-Wasserstein distance to Lebesgue for $p < p_\star$ and infinite $p$-Wasserstein distance for $p \geq p_\star$.

\subsection*{Acknowledgements}
The authors would like to thank Raphaël Lachièze-Rey \& Yogeshwaran D for communicating an early draft of \cite{lachieze2024hyperuniformity}, and for interesting discussions about these topics.

\corT{We thank the referees for their very careful reading and helpful criticism of our paper.}

\emph{MH is supported by the Deutsche Forschungsgemeinschaft (DFG, German Research Foundation) through the SPP 2265 \textit{ Random Geometric Systems} as well as under Germany's Excellence Strategy EXC 2044 -390685587, Mathematics M\"unster: Dynamics--Geometry--Structure.}

\emph{TL acknowledges the support of JCJC grant ANR-21-CE40-0009 from Agence Nationale de la Recherche. This work was completed during a visit of TL to Mathematics Münster thanks to the “Münster research fellow” program.}

\bibliographystyle{alpha}
\bibliography{HUCW}
\end{document}